\documentclass[twoside]{amsart}
\usepackage{import}

\usepackage[utf8]{inputenc}
\usepackage{amsfonts}
\usepackage{graphicx}
\usepackage{amsmath, amssymb}
\usepackage{hyperref,cleveref}
\usepackage{mathrsfs}
\usepackage{amsthm}
\usepackage{multirow}
\usepackage{centernot}
\usepackage{tikz}
\usepackage{color}
\usetikzlibrary{arrows}
\usepackage{pgf}
\usepackage{mathtools}
\usepackage{enumerate}
\usepackage{caption}
\usepackage{subcaption}
\usepackage{fancyhdr}
\usepackage[myheadings]{fullpage}

\newtheorem{theorem}{Theorem}[section]
\newtheorem{lemma}[theorem]{Lemma}
\newtheorem{proposition}[theorem]{Proposition}
\newtheorem{corollary}[theorem]{Corollary}
\newtheorem*{claim}{Claim}

\theoremstyle{definition}
\newtheorem{definition}[theorem]{Definition}
\newtheorem*{notation}{Notation}

\newtheorem*{limcon*}{Limit Conditions}
\newtheorem{example}{Example}[section]
\newtheorem{remark}{Remark}

\newcommand{\R}{\mathbb{R}}
\newcommand{\C}{\mathbb{C}}

\newcommand{\E}{\mathbb{E}}

\newcommand{\HH}{\mathbb{H}}
\newcommand{\Y}{\mathbb{Y}}
\newcommand{\MP}{\mathbb{MP}}
\renewcommand{\P}{\mathbb{P}}
\newcommand{\Z}{\mathbb{Z}}

\renewcommand{\l}{\lambda}
\renewcommand{\L}{\Lambda}
\renewcommand{\vec}[1]{\boldsymbol{#1}}
\newcommand{\cov}{\mathrm{Cov}}
\newcommand{\e}{\varepsilon}

\newcommand{\x}{\mathbf{x}}
\newcommand{\y}{\mathbf{y}}

\newcommand{\1}{\mathbf{1}}

\newcommand{\fB}{\mathfrak{B}}
\newcommand{\fC}{\mathfrak{C}}
\newcommand{\fD}{\mathfrak{D}}

\newcommand{\fH}{\mathfrak{H}}
\newcommand{\fP}{\mathfrak{P}}

\newcommand{\fZ}{\mathfrak{Z}}
\newcommand{\fa}{\mathfrak{a}}

\newcommand{\fp}{\mathfrak{p}}
\newcommand{\fq}{\mathfrak{q}}

\newcommand{\ft}{\mathfrak{t}}

\newcommand{\cA}{\mathcal{A}}
\newcommand{\cB}{\mathcal{B}}
\newcommand{\cC}{\mathcal{C}}
\newcommand{\cD}{\mathcal{D}}
\newcommand{\cE}{\mathcal{E}}

\newcommand{\cG}{\mathcal{G}}
\newcommand{\cH}{\mathcal{H}}
\newcommand{\cI}{\mathcal{I}}
\newcommand{\cJ}{\mathcal{J}}
\newcommand{\cK}{\mathcal{K}}
\newcommand{\cL}{\mathcal{L}}
\newcommand{\cM}{\mathcal{M}}

\newcommand{\cP}{\mathcal{P}}
\newcommand{\cQ}{\mathcal{Q}}

\newcommand{\cS}{\mathcal{S}}
\newcommand{\cT}{\mathcal{T}}

\newcommand{\sC}{\mathscr{C}}
\newcommand{\sL}{\mathscr{L}}
\newcommand{\sZ}{\mathscr{Z}}

\renewcommand{\i}{\mathbf{i}}
\newcommand{\hatR}{\widehat{\R}}
\newcommand{\wzeta}{\widetilde{\zeta}}

\newcommand{\defeq}{:=}
\newcommand{\Ups}{\Upsilon}

\DeclareMathOperator{\supp}{\mathrm{supp}}
\newcommand{\cl}{\mathrm{cl}}
\newcommand{\intr}{\mathrm{int}}

\makeatletter
\newcommand{\vast}{\bBigg@{4}}
\newcommand{\Vast}{\bBigg@{5}}
\makeatother

\renewcommand{\vec}[1]{\boldsymbol{#1}}
\newcommand{\ldeg}{\mathrm{ldeg}}

\newcommand{\cotimes}{\,\widehat{\otimes}\,}

\newcommand{\hloz}{\mathbin{\rotatebox[origin=c]{90}{$\Diamond$}}}
\newcommand{\rloz}{\mathbin{\rotatebox[origin=c]{33}{$\Diamond$}}}
\newcommand{\lloz}{\mathbin{\rotatebox[origin=c]{-33}{$\Diamond$}}}

\newcommand{\cpar}{%
  \raisebox{0.1em}{\rlap{\rotatebox{-45}{\rule[.10ex]{.4pt}{.5657em}}}%
  \kern.04em%
  \rlap{\kern.38em\raisebox{0.39em}{\rule{.4em}{.4pt}}}%
  \rule{.4em}{.4pt}\kern-.05em%
  \rotatebox{-45}{\rule[.1ex]{.4pt}{.5657em}}}}

\newcommand{\lpar}{%
  \raisebox{-0.1em}{\rlap{\rule[.05ex]{.4pt}{.45em}}%
  \kern-.0em%
  \rlap{\kern.0em\raisebox{0.45em}{\rotatebox{45}{\rule{.5657em}{.4pt}}}}%
  \rotatebox{45}{\rule{.5657em}{.4pt}} \kern-.37em%
  \raisebox{.4em}{\rule[.05ex]{.4pt}{.45em}}}}
  
\newcommand{\rpar}{%
  \raisebox{0.1em}{\rlap{\rule[.05ex]{.4pt}{.45em}}%
  \kern.0em%
  \rlap{\kern.0em\raisebox{0.45em}{\rule{.4em}{.4pt}}}%
  \rule{.4em}{.4pt}\kern-.03em%
  \rule[.05ex]{.4pt}{.45em}}}

\numberwithin{equation}{section}

\title{\uppercase{Global Universality of Macdonald Plane Partitions}}
\author{Andrew Ahn}
\date{}

\makeatletter
\let\runtitle\@title
\makeatother
\begin{document}
\maketitle

\begin{abstract}
We study scaling limits of periodically weighted skew plane partitions with semilocal interactions and general boundary conditions. The semilocal interactions correspond to the Macdonald symmetric functions which are $(q,t)$-deformations of the Schur symmetric functions. We show that the height functions converge to a deterministic limit shape and that the global fluctuations are given by the $2$-dimensional Gaussian free field as $q,t\to 1$ and the mesh size goes to $0$. Specializing to the noninteracting case, this verifies the Kenyon-Okounkov conjecture for the case of the $r^{\mathrm{volume}}$ measure under general boundary conditions. Our approach uses difference operators on Macdonald processes.\\


\end{abstract}

\section{Introduction} \label{sec:intro}

Given Young diagrams $\mu \subset \lambda$, a \textit{skew plane partition} supported in the skew diagram $\lambda/\mu$ is an array of nonnegative numbers $(\pi_{i,j})_{(i,j) \in \lambda/\mu}$ weakly decreasing in each index. For the purposes of this article we have $\lambda = N^M = (\underbrace{N,\ldots,N}_{M ~\footnotesize \mbox{times}})$. By viewing $\pi_{i,j}$ as the number of unit cubes on $(i,j)$, we may interpret a skew plane partition as a discrete, stepped surface in $\R^3$. The \textit{volume} of a skew plane partition is the number of unit cubes, that is $\sum_{i,j} \pi_{i,j}$. The projected image of this stepped surface further admits the interpretation of a skew plane partition as a lozenge tiling; a tiling of the triangular lattice by rhombi of three types. A fourth alternative perspective is that a skew plane partition can be viewed as a dimer covering of the honeycomb lattice.

\begin{figure}[h]
\centering
  \includegraphics[scale=0.5]{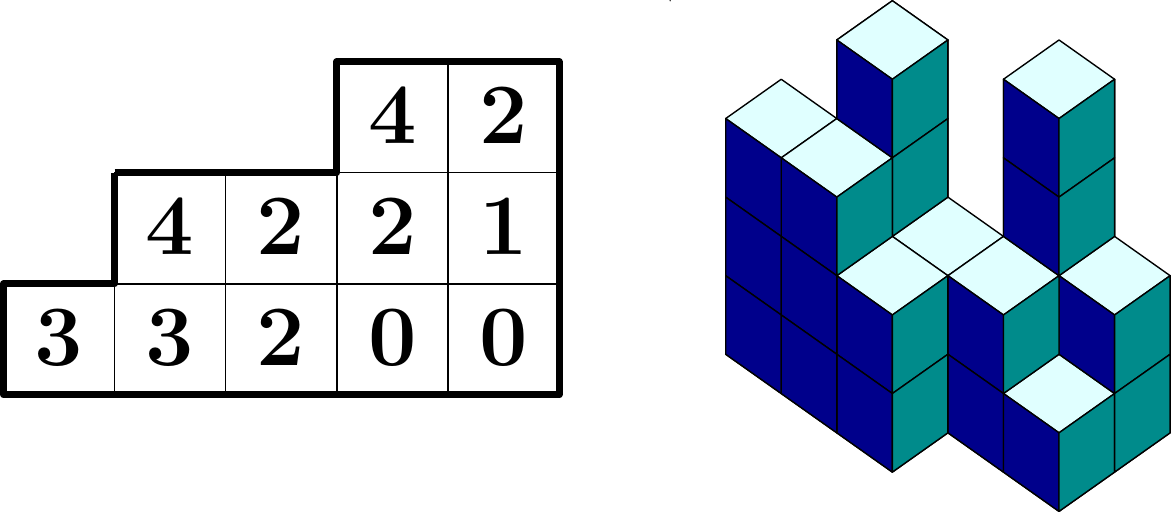}
  \caption{Skew plane partition with support $(5,5,5)/(3,1,0)$.}
\end{figure}

The central objects of this article are \textit{Macdonald plane partitions}, a broad class of measures on skew plane partitions which are also \emph{Macdonald processes}; stochastic processes with special algebraic properties. Macdonald processes were introduced in \cite{BC}, with asymptotics accessible through the method of difference operators. Arising in directed polymers, random matrices, and dimer models to name a few, these stochastic processes and their degenerations have found applications in a variety of probabilistic models, e.g. \cite{BC}, \cite{BG}, \cite{Di}. More recently, a class of difference operators for Macdonald symmetric functions which directly accesses moments of Macdonald processes was discovered by Negut \cite{N} and applied to the study of a finite-difference limit of the $\beta$-Jacobi corners process in \cite{GZ} and \cite[Appendix 1]{FLD} by Borodin, Gorin and Zhang.

From the methods perspective, the aim of this article was to further develop the machinery of Negut’s difference operators for the extraction of global asymptotics of Macdonald processes. One achievement is the extension of Negut’s difference operators to general Macdonald processes with multiply-peaked boundaries; this is essential to analyze skew plane partitions whenever $\mu$ is not the empty diagram. Yet another is that we access observables at \emph{singular points} of Macdonald processes; distinguished points where the model exhibits unbounded and singular behavior. Altogether, our analysis provides a unified framework for the study of a general class of Macdonald processes.

While the application of this method to Macdonald plane partitions illustrates the breadth of the approach, the focus on Macdonald plane partitions is motivated in part by the long-standing conjecture of Kenyon and Okounkov (KO conjecture) \cite[Section 1.5, page 15] {KO} on Gaussian free field fluctuations of periodic dimer models which we recall below. More specifically, the Macdonald plane partitions provide a rich family of \emph{non-uniform} models which are situated in a space extending the domain of KO conjecture. Our goal was to demonstrate that (the appropriate extension of) KO conjecture continues to hold for the broadest class of non-uniform models which are accessible via the Macdonald processes approach. Yet another point of interest for Macdonald plane partitions is in their connection to random matrices. In particular, they may be viewed as discrete realizations of eigenvalues processes for products of random matrices; we provide more details below.

KO conjecture was stated in their seminal paper \cite{KO} which established a general limit shape theorem for dimers on $\Z \times \Z$ periodic, bipartite graphs (see also \cite{KOS}); in more detail, one can associate a natural height function to periodic, bipartite dimer models and the limit shape theorem states that the height function converges, as the mesh size goes to $0$, to the solution of some variational problem. We note that \cite{KO} was preceded by a history of works which was initiated by Cohn, Kenyon and Propp in \cite{CKP} where the limit shape phenomenon was established for uniform domino tilings (i.e. square lattice dimer models). Complementing the limit shape theorem, Kenyon and Okounkov conjectured that the height function of uniform dimer models exhibit Gaussian free field fluctuations in the limit as the mesh size goes to $0$. Moreover, they gave a conjectural description of the complex coordinates which in the case of lozenge tilings admits a nice geometric interpretation in terms of the local proportions of lozenges $\lloz, \rloz, \hloz$ (see Section \ref{sec:modelresults}). Though a general proof of KO conjecture remains undiscovered, the conjecture has been verified for \emph{uniform} domino and lozenge tiling models for an assortment of domains, see \cite{K1}, \cite{K2}, \cite{P}, \cite{BuG}, \cite{BuG2}.

While KO conjecture was stated for uniform dimer models, the conjecture can be readily extended to non-uniform models which emulate a volume constraint (see \cite[Section 2.4]{BGR} and \cite{CK}). The simplest such model is the $r^{\mathrm{volume}}$ measure which is a measure on skew plane partitions with fixed support $\lambda/\mu$ and probability $\P(\pi) \propto r^{\mbox{volume of $\pi$}}$. Originally introduced by Vershik \cite{Ver} when $\mu$ is the empty partition, the limit shape and local asymptotics of the $r^{\mathrm{volume}}$ measure have been thoroughly studied (see \cite{OR03}, \cite{OR07}, \cite{BMRT}, \cite{M}). Despite the abundance of literature on this simple model, there are no results on the global fluctuations in the literature even for the case of ordinary (when $\mu$ is empty) partitions. Since this gap in the literature is unfortunate, the present article fills this vacancy and proves KO conjecture for $r^{\mathrm{volume}}$ using the approach of Macdonald processes. Moreover, as far as the author is aware, this article provides the first non-uniform lozenge tiling model for which KO conjecture is true.

However, the Macdonald processes approach applies to a far more general family of measures beyond the $r^{\mathrm{volume}}$ measures. To demonstrate this generality, we consider the most inclusive set of measures on random plane partitions for which the Macdonald processes approach applies. Simultaneously, we sought to push the boundaries for which KO conjecture holds. From this investigation, we find that KO conjecture encompasses a menagerie of models exhibiting a variety of features such as \emph{periodic weighting} and \emph{semilocal interactions}; we note that periodically weighted variants of $r^{\mathrm{volume}}$ were expected to satisfy KO conjecture but the inclusion of models with semilocal interactions is a novelty for lozenge tiling models. In other words, we find \emph{universality} in the global fluctuations of Macdonald plane partitions. Furthermore, the Macdonald processes approach provides an explicit description of the limit shape in terms of its moments going beyond the general description given in \cite{KO}.

We now comment further on some of the aforementioned features in Macdonald plane partitions. In one direction of generality, Macdonald plane partitions contain \emph{periodically weighted} variants of the $r^{\mathrm{volume}}$ measure. These measures have weights with $p\Z \times \Z$ periodicity rather than the $\Z \times \Z$ periodicity of the $r^{\mathrm{volume}}$ measure. The class of models studied in this article supports general skew diagrams which lead to exotic limit shapes, see Section \ref{sec:frzbdry} and Figure \ref{fig:limshape}. The presence of periodically varying weights produce cusps in the frozen boundary whose placement is determined by the changes in slope of the boundary. We note that the limit shape phenomenon and local asymptotics for the two-periodic case, near special cusp points, were studied for fairly specific boundaries in \cite{M2}. This is the first work to consider general boundary conditions for arbitrary period lengths. The analysis of $p$-periodically weighted models also introduces new phenomenon in which the integral formulas of the moments contain $p$th roots of rational functions, see e.g. Section \ref{sec:cpx}.

\begin{figure}[h]
\centering
\begin{subfigure}{.5\textwidth}
  \centering
  \includegraphics[width=.8\linewidth]{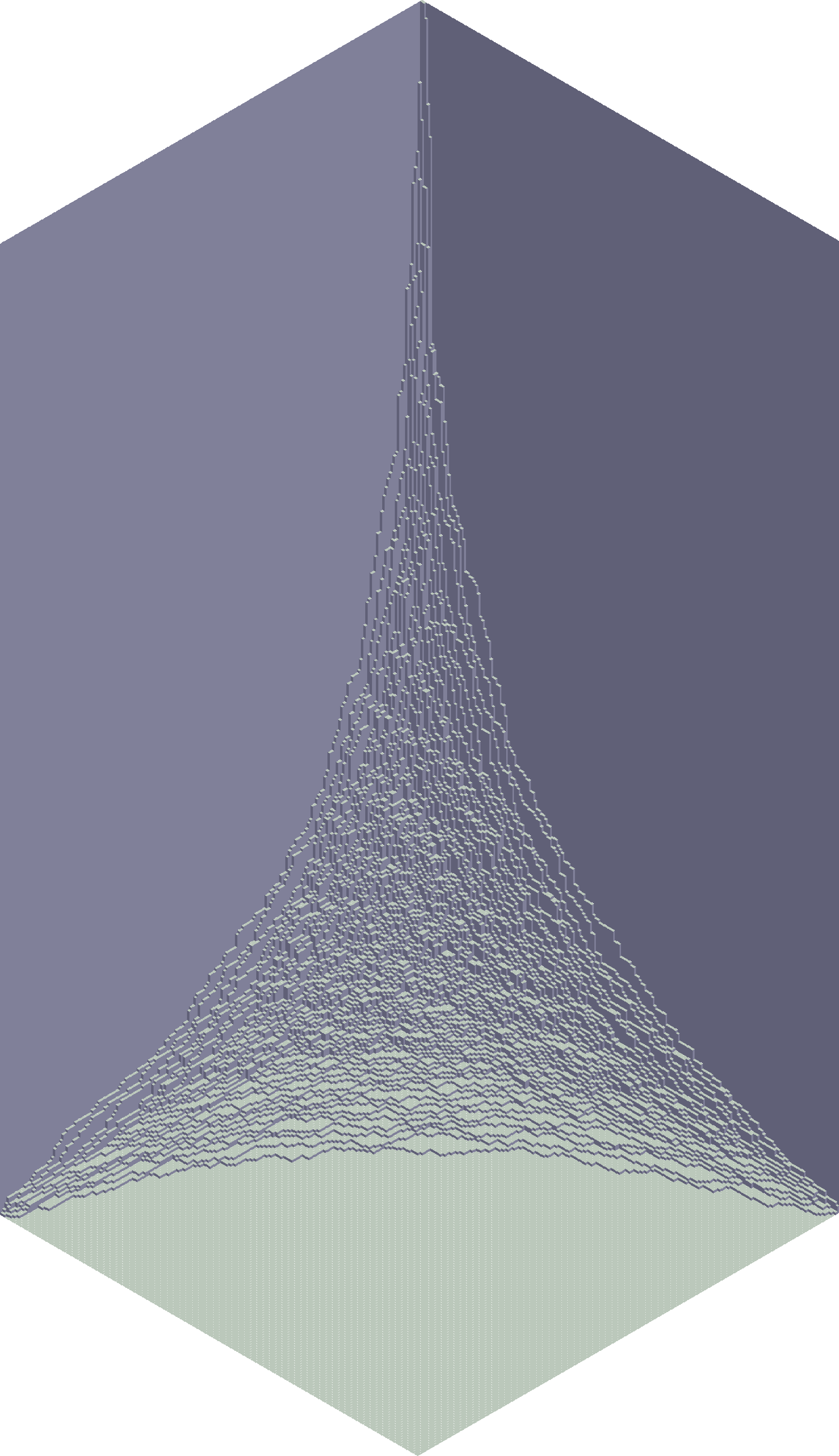}
  \caption{Limit shape for $r^{\mathrm{volume}}$ ordinary partition.}
\end{subfigure}%
\begin{subfigure}{.5\textwidth}
  \centering
  \includegraphics[width=.8\linewidth]{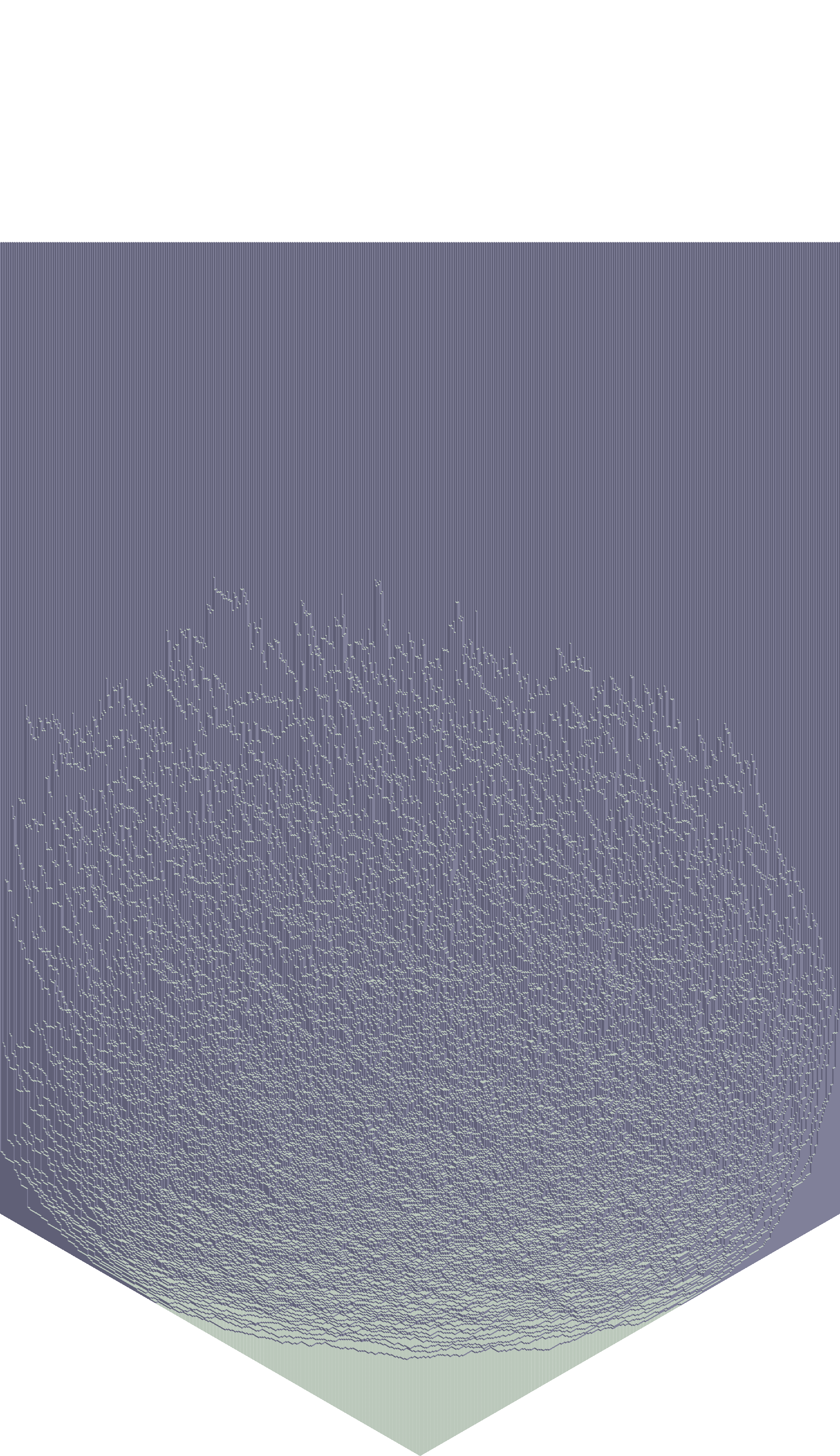}
  \caption{Limit shape for $6$-periodic model}
\end{subfigure}
\caption{}
\label{fig:limshape}
\end{figure}

In another direction of generality, the Macdonald plane partitions exhibit semilocal interactions of varying strengths. By semilocal interactions, we mean that the Macdonald plane partitions are a family of \emph{interacting dimer models} on the honeycomb lattice where semilocality refers to the interaction being longer-range in one of the coordinate directions. For our models, a deformation parameter pair $(q,t)$ modulates this interaction with $q = t$ corresponding to the non-interacting models and the interaction parameter $\frac{\log q}{\log t}$ exaggerating the strength of the interaction as it deviates away from $1$. In this direction, there is the related work of Giuliani, Mastropietro and Toninelli on global fluctuations for interacting dimers on the square lattice in \cite{GMT}. A common feature of our results is that the fluctuations depend on the interaction parameter only by a scaling factor. Let us also note that our model is \emph{non-determinantal} when $q \ne t$, and in particular the method of Macdonald processes is the only approach available presently to access KO conjecture for general Macdonald plane partitions.

Apart from their generality and variety, Macdonald plane partitions are also of interest due to their deep connection with random matrix theory. Let $\beta = 2\frac{\log q}{\log t}$ be the \emph{interaction type}. By degenerating (one-periodic) Macdonald plane partitions via the Heckman-Opdam limit which fixes the interaction type, one can obtain the eigenvalue distribution of certain products of random matrices. For $\beta = 2$, this connection is explored in \cite{BGS} where the singular values of products of truncated unitary matrices have correlation kernels obtained via limits of random plane partitions with certain boundary conditions that correspond to the truncation sizes. The $\beta = 1,4$ cases correspond to products of real symmetric and quaternion Hermitian matrices respectively. Thus the Macdonald plane partitions can be viewed as discrete realizations of product matrix processes. We further explore this connection in a future publication. In another similar connection to random matrices, the interaction type $\beta$ for our skew plane partitions behaves as the $\beta$ log-gas parameter in random matrix theory. This is manifested in the usual $\beta$-dependence in global fluctuations, namely the height functions have to be renormalized in a characteristic manner depending on $\beta$ in order to converge to the (properly scaled) Gaussian free field. We note the limit shape and global fluctuations of the so-called discrete $\beta$-ensembles were studied in \cite{BGG} which are yet another discrete system exhibiting random matrix $\beta$-type interactions.

We finally note that this is not the first work which considers Macdonald deformations of the $r^{\mathrm{volume}}$ measure. By taking $q = 0$ in the $(q,t)$ parameter pair above, one obtains the Hall-Littlewood plane partitions, parametrized by $t$, which were studied by Vuleti\'c in \cite{Vul} and Dimitrov in \cite{Di}. Vuleti\'c studied the case $t = -1$ and showed that the underlying point process is given by a Pfaffian point process. Dimitrov considered general Hall-Littlewood plane partitions, and showed that the lower boundary of the limit shape was independent of the parameter $t$, along with finding Tracy-Widom and KPZ-type fluctuations. In a similar spirit, our limit shape and fluctuation results are independent of $q,t$ except for a scaling factor given by the log-ratio of $q$ and $t$.

The remainder of the article is organized as follows. Section \ref{sec:modelresults} provides a more detailed background on random skew plane partitions, introduces the Macdonald plane partitions, and states the main results of this article: limit shape theorems and the verification of KO conjecture (i.e. global fluctuations) for Macdonald plane partitions. In Section \ref{sec:obs}, we extend the difference operators of Negut to formal equalities for joint moments of general Macdonald processes, then specialize to obtain contour integral formulas for the joint moments of random skew plane partitions. In Section \ref{sec:asymp}, we perform asymptotics on the contour integral formulas for the joint moments. Section \ref{sec:limgff} concludes the article by proving the main results on the limit shape and Gaussian free field fluctuations, relying on properties of (the complex structure on) the liquid region and frozen boundary obtained in Section \ref{sec:cpx}.

\begin{notation}
Let $\i$ denote the imaginary unit, i.e. the square root of $-1$ in the upper half plane. Given an interval $[a,b] \subset \R$, we write $[[a,b]] := [a,b] \cap \Z$. Given a set $K \subset \C$, we denote the interior of $K$ by $\intr(K)$ and the closure of $K$ by $\cl(K)$. Let $\R_{>0}$ ($\R_{\ge 0}$) denote the positive (nonnegative) real numbers and $\Z_{> 0}$ ($\Z_{\ge 0}$) denote the positive (nonnegative) integers. 
\end{notation}

\section*{Acknowledgments}
I would like to thank my advisor Vadim Gorin for suggesting this project, for many useful discussions, and for carefully looking over several drafts. I thank Sevak Mkrtchyan for helpful discussions and giving access to his random tiling sampler. I also thank Alexei Borodin for helpful suggestions. The author was partially supported by NSF Grant DMS-1664619.

\section{Model and Results} \label{sec:modelresults}
We now introduce our models and results with greater detail. For clarity, we begin by introducing the non-interacting models and results, corresponding to Subsections \ref{sec:tiling} and \ref{sec:results}. In Subsection \ref{sec:MPP}, we parallel the preceding discussion for more general interacting models.

\subsection{Plane Partitions and Lozenge Tilings} \label{sec:tiling}
We interchangeably say Young diagrams and partitions. Let $\mu \subset N^M = ( \underbrace{N,\ldots,N}_{M~\footnotesize \mbox{times}} )$ be a Young diagram. By the \emph{back wall} of $N^M/\mu$ we mean the upper boundary of the skew diagram $N^M/\mu$ (see Figure \ref{fig:skewpp}). Let $\pi = (\pi_{i,j})$ be a skew plane partition with support $N^M/\mu$. For $-M < v < N$, the diagonal section $\pi^v = (\pi_{a,v+a}, \pi_{a+1,v+a+1},\ldots)$ is an ordinary partition, where $a$ is the least integer such that $(a,v+a)$ is a box in $N^M/\mu$.

\begin{figure}[ht]
    \centering
    \includegraphics[scale=0.7]{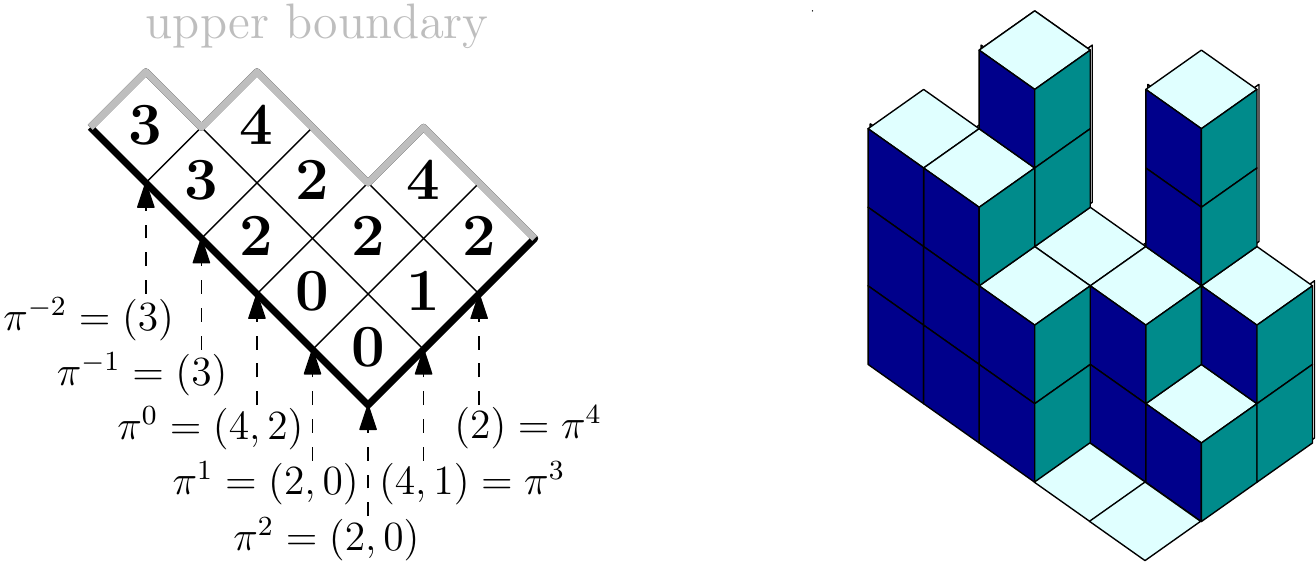}
    \caption{(Left) A skew plane partition with support $5^3 / (3,1,0)$. The grey upper boundary is the back wall. We label the partitions along the diagonal sections from $v = -2$ to $3$ from left to right. (Right) The skew plane partition as a $3$-dimensional object.}
    \label{fig:skewpp}
\end{figure}

A skew plane partition can be viewed as a $3$-dimensional object by stacking $\pi_{i,j}$ cubes above the box $(i,j)$ as in Figure \ref{fig:skewpp}. The resulting (projected) image is a tiling of lozenges $\hloz$, $\rloz$, $\lloz$. For our purposes, we transform the lozenges by the affine transformation taking $\hloz \mapsto \cpar$, $\rloz \mapsto \rpar$, $\lloz \mapsto \lpar$. Take the standard basis of $\R^2$ for the resulting image with lengths so that the transformed lozenge $\rpar$ is the unit square, see Figure \ref{fig:tiling}. This gives a unique \textit{projected coordinate} system for the tiling, up to the choice of origin. The back wall is then the graph of some function $B:I \to \R$ which is piecewise linear with slopes $0$ or $1$. The domain $I$ of $B$ is an interval of length $M+N$. For convenience, choose the origin in projected coordinates so that $I = (-M,N)$. Then the centers of the projected horizontal lozenges $\cpar$ corresponding to the diagonal section $\pi^v$ have $x$-coordinate $v$, see Figure \ref{fig:tiling}. Denote by $\cP_B$ the set of plane partitions with back wall $B$. We may also consider semi-infinite or infinite back walls by taking $M$ or $N$ to $\infty$.

\begin{figure}[ht]
    \centering
    \includegraphics[scale=0.7]{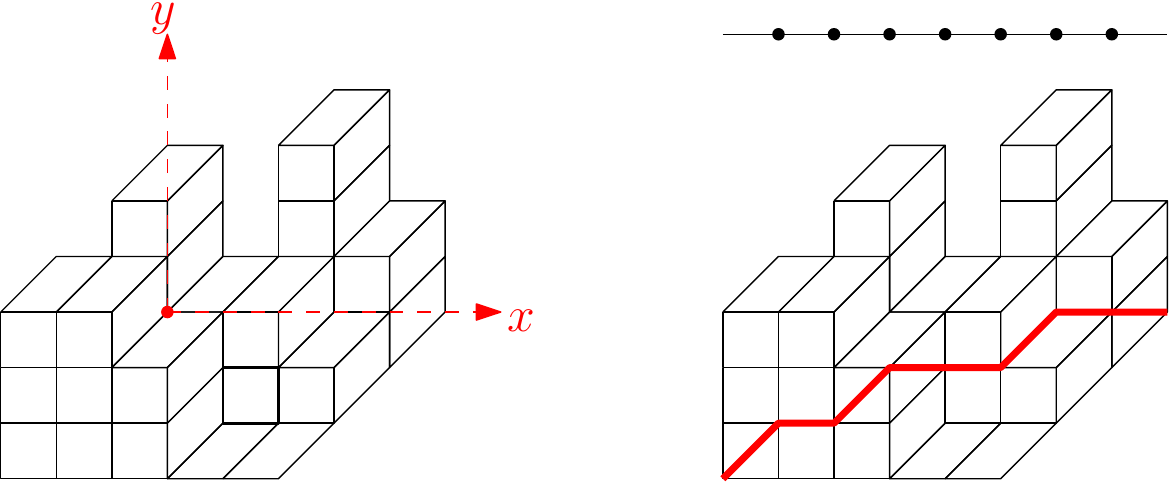}
    \caption{The skew plane partition in Figure \ref{fig:skewpp} after the affine transformation. (Left) The projected coordinate axes in red. (Right) The back wall in red, the line above represents the domain $(-3,5)$ of the back wall, and the dots correspond to coordinates of the diagonal sections.}
    \label{fig:tiling}
\end{figure}

Fix a skew plane partition $\pi \in \cP_B$. We define a height function which takes a point $(x,y)$ and gives the height at that point. More precisely, the \emph{height function} $h:(I\cap \Z) \times \R \to \R$ is the piecewise linear function which reports the total length of vertical line segments below the point $(x,y)$ in projected coordinates, see Figure \ref{fig:htfunction}.

\begin{figure}[ht]
    \centering
    \includegraphics[scale=0.7]{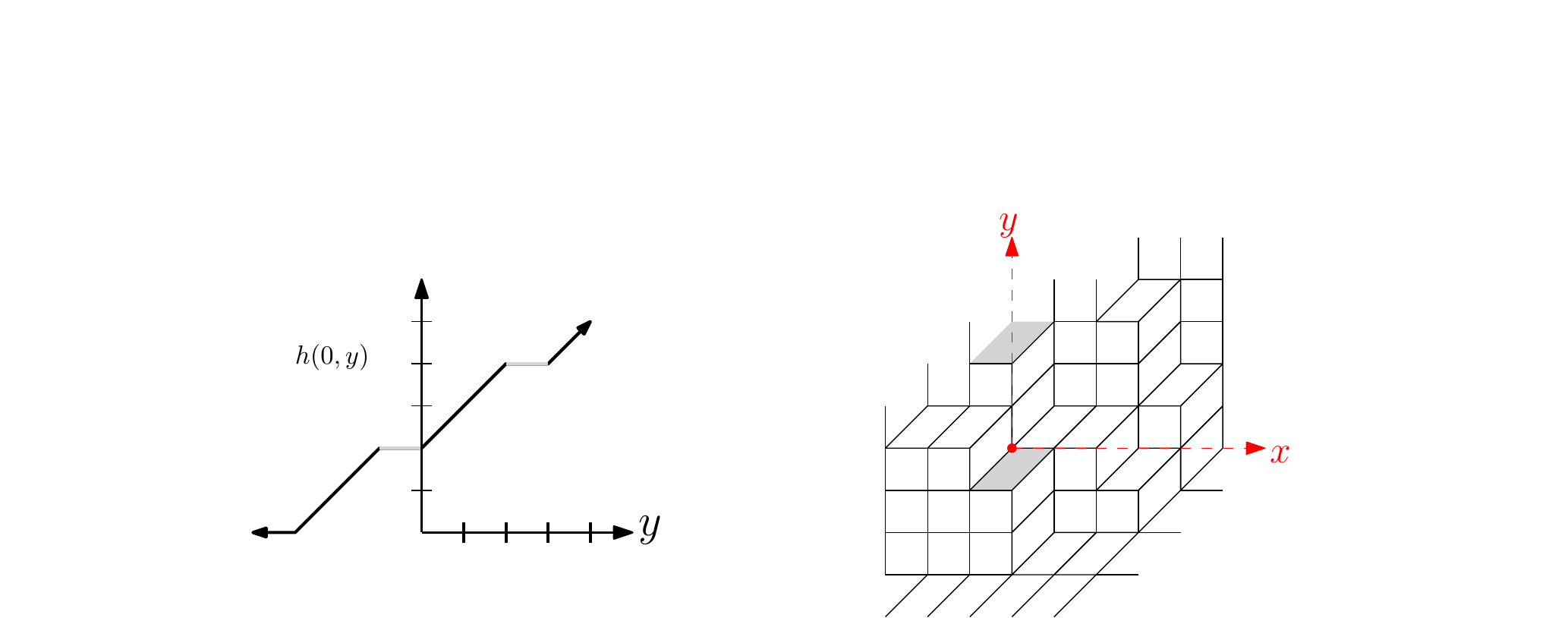}
    \caption{Graph of the height function at $x = 0$. The gray filled tiles correspond to the flat, gray parts of the graph.}
    \label{fig:htfunction}
\end{figure}

For a partition $\lambda = (\lambda_1,\lambda_2,\ldots)$ define $|\lambda| = \sum_{i\ge 1} \lambda_i$. Consider the random (skew) plane partition (RPP) with probability distribution on $\cP_B$ defined by
\begin{equation} \label{eq:rv}
\P(\pi) \propto \prod_{-M < v < N} r_v^{|\pi^v|}
\end{equation}
for a sequence of weights $r_v > 0$ such that the weights above are summable. When $r_v$ is constant in $v$, this is the $r^{\mathrm{vol}}$ measure studied in \cite{OR03}, \cite{OR07}, \cite{M}, \cite{BMRT}.

\begin{definition}
Let $\vec{s} = (\ldots,s_{-1},s_0,s_1,\ldots)$ be a $p$-periodic, bi-infinite sequence of positive numbers. Denote by $\P^{B,r,\vec{s}}$ the probability measure on $\cP_B$ defined by (\ref{eq:rv}) where
\[ r_v = s_v r, \]
given that the weights are summable.
\end{definition}

We note that local limits for a specific class of back walls $B$ are studied for $p = 2$ in \cite{M2}.

\subsection{Results} \label{sec:results}
Our main result is an explicit description of the global fluctuations, in terms of a Gaussian free field, of the measures $\P^{B,r,\vec{s}}$ as $r \to 1$ and $B$ converges to some limiting $\cB$ after rescaling. More precisely, we consider the following limit regime.

\begin{limcon*} \label{limreg}
Fix a $p$-periodic, bi-infinite sequence $\vec{s} = (\cdots,s_{-1},s_0,s_1,\cdots) \in \R_{>0}^\infty$ such that $s_0 \cdots s_{p-1} = 1$. Let $\P^{B,r,\vec{s}}$ be parametrized by a small parameter $\e > 0$ where $B:I^\e \to \R$ and $r \defeq e^{-\e}$ vary with $\e$ such that
\begin{enumerate}
    \item \label{limreg:1} there exist integers
    \[ \inf I^\e = v_0(\e) < \cdots < v_n(\e) = \sup I^\e \]
    such that for each $1 \le \ell \le n$, $B'$ is $p$-periodic on $(v_{\ell-1}, v_\ell) \cap (\Z + \frac{1}{2})$;
    \item \label{limreg:2} there exists an interval $I \subset \R$ and a piecewise linear $\cB:I \to \R$ with non-differentiable points
    \[ \inf I = V_0 < \cdots < V_n = \sup I \]
    such that
    \[ \e v_\ell(\e) \to V_\ell \quad (0 \le \ell \le n), \quad \e B^\e(x/\e) \to \cB(x) \]
    as $\e \to 0$, where the latter convergence is uniform over any compact subset of $I$.
\end{enumerate}
\end{limcon*}

\begin{remark}
The condition $s_0 \cdots s_{p-1} = 1$ is to ensure the existence of a non-trivial limit shape. If $s_0 \cdots s_{p-1} < 1$, then the limit shape becomes trivial; $0$-volume upon rescaling. If $s_0 \cdots s_{p-1} > 1$, then for $r$ close to $1$ the weights of $\P^{B,r,\vec{s}}$ are no longer summable.
\end{remark}

\begin{figure}[ht]
    \centering
    \includegraphics[scale=0.6]{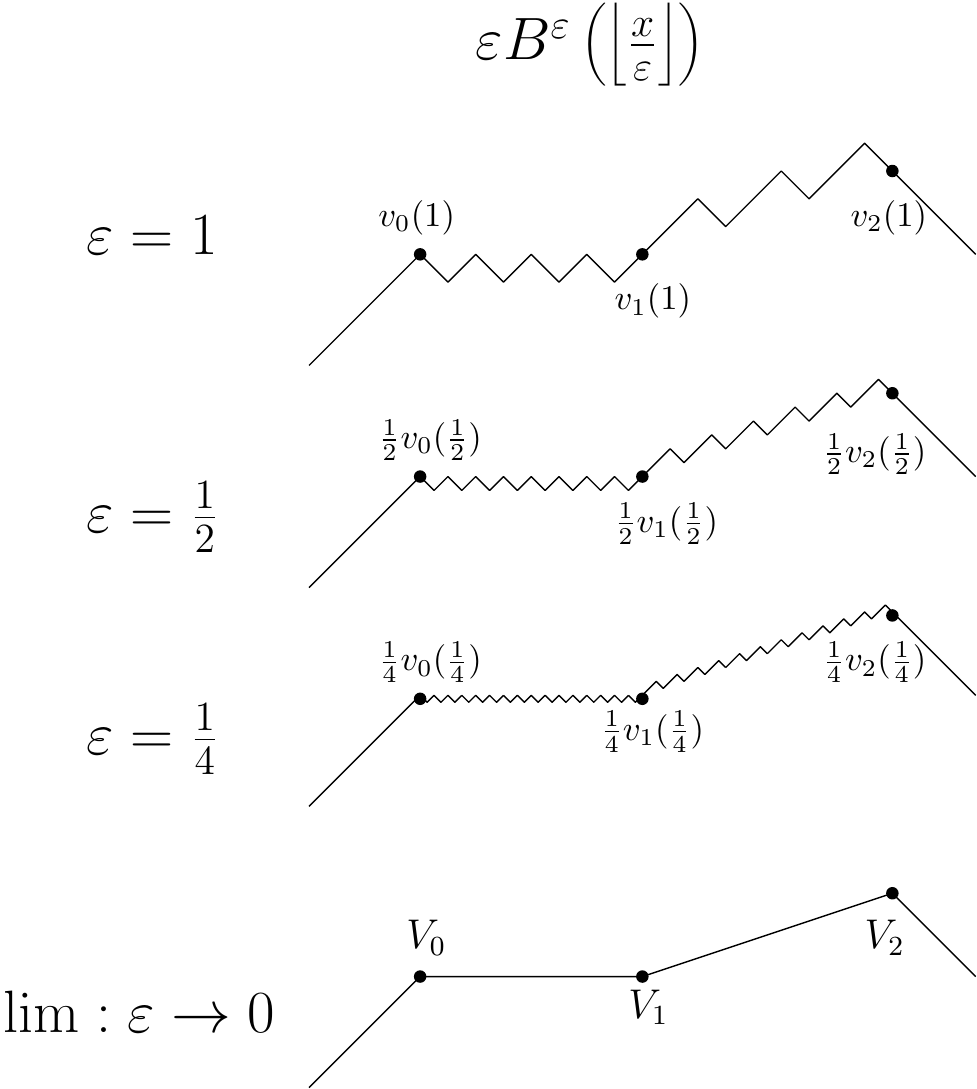}
    \caption{An example Limit Condition (\ref{limreg:2}) illustrated by the graphs of $\e B^\e(\left( \lfloor \frac{x}{\e} \rfloor \right)$ at $\e = 1,\frac{1}{2}, \frac{1}{4}$ and the limit $\e \to 0$, along with transition points $\e v_i(\e)$.}
    \label{fig:blim}
\end{figure}

In Section \ref{sec:backwall}, we classify the set $\fB(\vec{s})$ of possible limits of back walls $\cB$ attained by $\P^{B,r,\vec{s}}$ satisfying the Limit Conditions. Our law of large numbers and fluctuations results are restricted to a dense subset $\fB^\Delta(\vec{s}) \subset \fB(\vec{s})$, defined in Section \ref{sec:backwall}. The reason for this restriction is related to the presence of \emph{singular points}; a concept further explained in Section \ref{sec:backwall}. Elements $\cB \in \fB^\Delta(\vec{s})$ correspond to RPP limits with only finitely many singular points whereas $\cB \in \fB(\vec{s}) \setminus \fB^\Delta(\vec{s})$ correspond to RPP limits with a continuum's worth of singular points. Our methods in general are limited to accessing models with finitely many singular points, thus this restriction is necessary. 

Before proceeding to the main result, it is convenient to state the following limit shape result under our limit regime. Let $h$ denote the random height function of $\P^{B,r,\vec{s}}$.

\begin{theorem} \label{thm:LLN}
Suppose $\P^{B,r,\vec{s}}$ satisfies the Limit Conditions such that $\cB \in \fB^\Delta(\vec{s})$. Then there exists a deterministic Lipschitz $1$ function $\cH: I \times \R \to \R$ such that we have the convergence
\[ \e h\left( \left\lfloor \frac{x}{\e} \right\rfloor, \frac{y}{\e}\right) \to \cH(x,y) \]
of measures on $y \in \R$, weakly in probability as $\e \to 0$ for all $x\in I$. An explicit description of this height function is given in Section \ref{sec:limgff} in terms of its exponential moments.
\end{theorem}

\begin{remark}
We note that there exists an approach to Theorem \ref{thm:LLN} through the variational principle \cite{CKP}, \cite{KOS}, \cite{KO}. Our approach is different with the benefit of giving explicit formulas for exponential moments and being generalizable to the Macdonald plane partitions introduced in Section \ref{sec:MPP}.
\end{remark}

Let $p_{\hloz},p_{\lloz},p_{\rloz}$ denote the local proportions of the subscripted lozenges, if they exist. Given the deterministic limit $\cH$, the local proportions of lozenges at $(x,y) \in I \times \R$ \emph{are well-defined} and given by
\begin{align*}
\nabla \cH(x,y) & = (1 - p_{\hloz}, -p_{\lloz}) \\
p_{\hloz} + p_{\lloz} + p_{\rloz} &= 1
\end{align*}
It is convenient to encode the local proportions by a complex parameter $z \in \overline{\HH}$ so that
\begin{equation} \label{eq:locprop}
p_{\hloz} = \arg z, ~~~ p_{\lloz} = \frac{1}{p} \sum_{i=0}^{p-1} \arg (1 - s_0 \cdots s_i z)
\end{equation}
where the argument is chosen to be $0$ on the positive reals. There is a unique such choice of $z \in \overline{\HH}$ for any given triple $(p_{\hloz},p_{\lloz}, p_{\rloz})$. In the case where the period is $1$ the parameter $z$ admits a nice geometric interpretation: the triangle $(0,1,z)$ has angles $\pi (p_{\hloz}, p_{\lloz}, p_{\rloz})$, see Figure \ref{fig:tri}. For higher periods $p$, the author is unaware of a simple geometric alternative to (\ref{eq:locprop}).
\begin{figure}[ht] 
\centering
\includegraphics[scale=0.4]{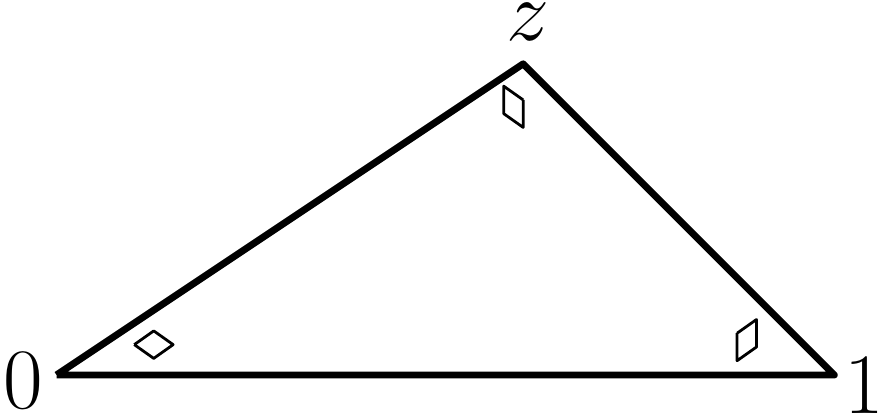} 
\caption{Geometric description for parameter $z$ in $1$-periodic case.} \label{fig:tri}
\end{figure}

We briefly recall the pullback of the Gaussian free field. Detailed discussions of the $2$-dimensional Gaussian free field can be found in \cite{S}, \cite[Section 4]{D}.
\begin{definition}
The Gaussian free field $\fH$ (with Dirichlet boundary conditions) on $\HH$ is defined to be the generalized centered Gaussian field on $\HH$ with covariance
\[ \E\, \fH(z) \fH(w) = - \frac{1}{2\pi} \log \left| \frac{z - w}{z - \bar{w}} \right| \]
Given a domain $D$ and a homeomorphism $\Omega: D \to \HH$, the $\Omega$-pullback of the Gaussian free field $\fH \circ \Omega$ is a generalized centered Gaussian field on $D$ with covariance
\[ \E\, \fH(\Omega(u))  \fH(\Omega(v)) = - \frac{1}{2\pi} \log \left| \frac{\Omega(u) - \Omega(v)}{\Omega(u) - \overline{\Omega}(v)} \right|. \]
\end{definition}

\begin{definition}
Let $\cI$ be some indexing set and some family $\{\xi_i\}_{i \in \cI}$ of random variables. Moreover, for each $\e > 0$, define a family of random variables $\{\xi_i^\e\}_{i \in \cI}$. We say that $\{\xi_i^\e\}_{i \in \cI} \to \{\xi_i\}_{i \in \cI}$ as $\e \to 0$ in distribution if for any finite collection $i_1,\ldots,i_k \in I$ the random vector $(\xi_{i_1}^\e,\ldots,\xi_{i_k}^\e)$ converges in distribution to $(\xi_{i_1},\ldots,\xi_{i_k})$.
\end{definition}

Let $\fH$ denote the Gaussian free field with Dirichlet boundary conditions on $\HH$, and denote by
\[ \overline{h}(x,y) = h(x,y) - \E h(x,y) \]
the centered height function. Let the \emph{liquid region} be defined to be the set of $(x,y)$ such that all the local proportions $p_{\lloz}, p_{\rloz}, p_{\hloz}$ are positive. We are now ready to state the main result for the ``Schur case''.

\begin{theorem} \label{thm:GFF}
Suppose $\P^{B,r,\vec{s}}$ satisfies the Limit Conditions such that $\cB \in \fB^\Delta(\vec{s})$. The map $\zeta(x,y) = e^x z(x,y)$, where $z$ is defined by (\ref{eq:locprop}), is a homeomorphism from the liquid region to $\HH$. Moreover, the centered, rescaled height function $\sqrt{\pi}\,\overline{h}\left(\lfloor \frac{x}{\e} \rfloor,\frac{y}{\e} \right)$ converges to the $\zeta$-pullback of the GFF in the sense that we have the following convergence in distribution
\[ \left\{ \sqrt{\pi} \int \overline{h}\left( \left\lfloor \frac{x}{\e} \right\rfloor,\frac{y}{\e} \right) e^{-ky} \, dy \right\}_{x\in I, k \in \Z_{>0}} \to \left\{\int \fH(\zeta(x,y)) e^{-ky} \, dy \right\}_{x \in I, k \in \Z_{>0}}.\]
\end{theorem}

In \cite{KO}, Kenyon and Okounkov conjectured that the fluctuations of the height function for $\Z \times \Z$ periodic, bipartite dimer models are given by a Gaussian free field. Theorem \ref{thm:GFF} confirms this conjecture for periodically weighted skew plane partitions.

\begin{remark}
We note that modifying the Limit Conditions so that $r = \exp(-c\e)$ for some constant $c$ amounts to scaling the coordinates of the plane partition. For this reason, we consider $c = 1$ to reduce the number of parameters. In this case $\zeta(x,y) = e^x z(x,y)$.
\end{remark}

\begin{remark}
We emphasize that in the uniform lozenge tiling models studied in previous works, the uniformization map from $\sL$ onto $\HH$ is \emph{not} given by the complex slope $\zeta(x,y)$, e.g. in \cite{BuG}, \cite{BuG2}, \cite{P}. For the uniform models, the parameter $c$ in the remark above is taken to be $0$ so that $\zeta(x,y) = z(x,y)$, and this gives a covering map from $\sL$ onto $\HH$ with degree $>1$. As an example, $z(x,y)$ gives a $2$-sheeted covering for the uniform lozenge tilings of a regular hexagon due to rotational symmetry. The reason $\zeta(x,y)$ gives a uniformization map for our models can be related to the fact that there is only one connected component for the frozen region corresponding to $\hloz$, which is a consequence of our models having no ``ceiling''. 
\end{remark}

\begin{remark}
The map $\zeta$ depends continuously on the back wall $\cB$. Thus Theorem \ref{thm:GFF} provides a continuous family of GFFs parametrized by $\cB$ corresponding to asymptotic RPPs.
\end{remark}

\subsection{Macdonald Plane Partitions} \label{sec:MPP}
We now introduce a two-parameter family of deformation for the RPP defined by (\ref{eq:rv}). These deformations correspond to the $(q,t)$-parameter family of Macdonald symmetric functions with (\ref{eq:rv}) corresponding to the Schur case $q = t$. Instead of $q$, we will take a parameter $\alpha > 0$ so that $q = t^\alpha$. Fix $0 < t < 1$, $\alpha > 0$ and define the corresponding probability distribution on $\cP_B$
\begin{equation} \label{eq:macrv}
\P(\pi) = w_{\alpha,t}(\pi) \prod_{-M < v < N} r_v^{|\pi^v|}
\end{equation}
where $w_{\alpha,t}$ is an $r$-independent $(\alpha,t)$-\emph{Macdonald weight}, and the summability of the weights (\ref{eq:macrv}) coincides with the summability of the weights (\ref{eq:rv}). The Macdonald weights $w_{\alpha,t}(\pi)$ can be described in terms of semilocal contributions of the plane partition. We explain this in greater detail below, after introducing the coordinate system.

\begin{figure}[ht]
    \centering
    \includegraphics[scale=0.7]{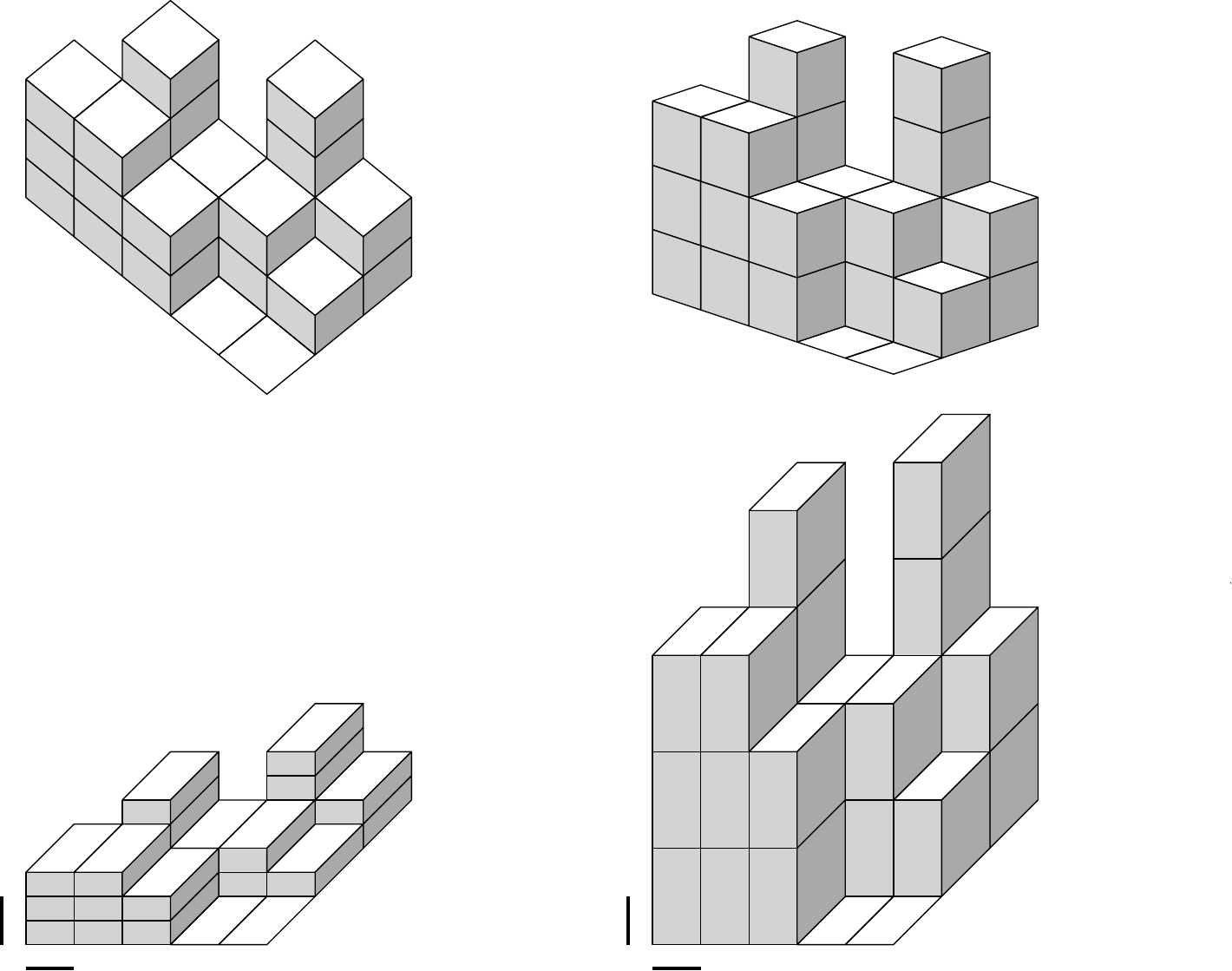}
    \caption{Left: $\alpha = 1/2$; Right: $\alpha = 2$. Top: $3$-d partition projected onto $(1,\alpha,\alpha)$-plane. Bottom: Transformed tiling with scaling so that the line segments by the lower left corner denote unit lengths.}
    \label{fig:mac}
\end{figure}

For the Macdonald plane partitions, it will be convenient to consider a different tiling and set of coordinates which we call the $\alpha$-\textit{coordinates}. We transform $\hloz, \rloz, \lloz$ to $\cpar, \rpar, \lpar$ where the widths of $\rpar,\lpar$ are $1$, the height of $\rpar$ is $\alpha$, and the height of $\cpar$ is $1$. This transformation is not affine since the height of $\cpar$ does not scale by $\alpha$ for $\alpha \neq 1$, see Figure \ref{fig:mac}. To understand where the coordinates come from, consider the $3$-dimensional plane partition. If the height corresponds to the third coordinate, then the projection in Figure \ref{fig:skewpp} is onto the $(1,1,1)$-plane. If instead we project onto the $(\alpha,\alpha,1)$-plane, then after choosing the basis parallel to the edges of the $\rloz$ lozenge we obtain the $\alpha$-coordinates (up to sign of direction), see top row of Figure \ref{fig:mac}.

Fix a plane partition $\pi \in \cP_B$. We define the height function as before which gives the height at $(x,y)$. More precisely, the \emph{height function} $h:(I \cap \Z) \times \R$ is defined as $\frac{1}{\alpha}$ times the total length of vertical line segments beneath a point $(x,y)$ in $\alpha$-coordinates, see Figure \ref{fig:macht}. The $\frac{1}{\alpha}$ term is included because the $\alpha$-coordinates \emph{contracts} the height from $\R^3$ by $\alpha$.

Consider further these vertical segments which are formed by intersections of an adjacent pair of lozenges $\lpar, \rpar$. We say that the vertical segment formed by the intersection of a such a pair of lozenges is a \emph{turn}. If the pair goes from $\lpar$ to $\rpar$ ($\rpar$ to $\lpar$) from left to right, then we call it an \emph{internal turn} (\emph{external turn}), see Figure \ref{fig:level}. The height function 

\begin{figure}[ht]
    \centering
    \includegraphics[scale=0.6]{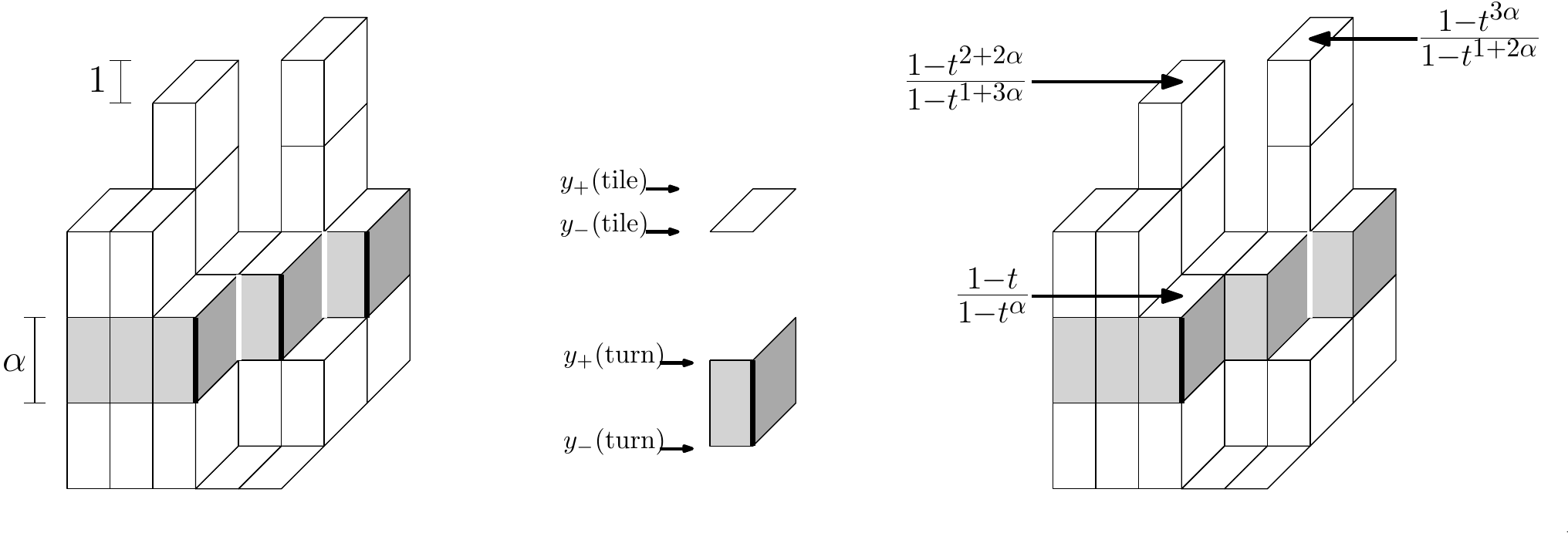}
    \caption{Left: We mark external and internal turns black and white resp. along the grey band. Middle: $y_+,y_-$ indicated for tiles and turns. Right: Weights contributed by the flat tiles above the highlighted turns.}
    \label{fig:level}
\end{figure}

\begin{figure}[ht]
    \centering
    \includegraphics[scale=0.6]{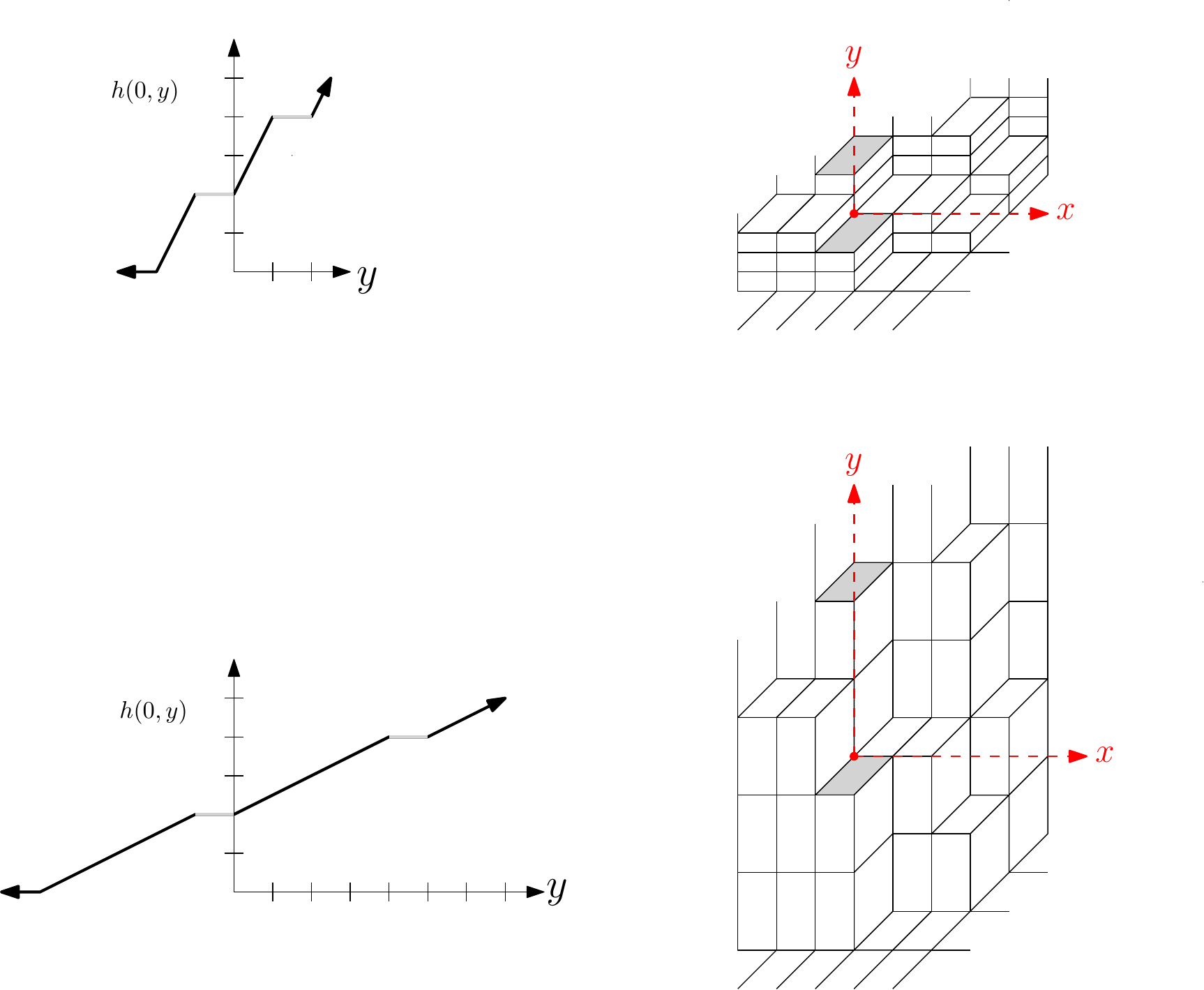}
    \caption{For $\alpha = 1/2$ (top) and $\alpha = 2$ (bottom), the graph of the height function at $x = 0$ and the associated tiling where the gray filled tiles correspond to the flat, gray parts of the graph.}
    \label{fig:macht}
\end{figure}

Denote the set of turns of $\pi$ by $\cT(\pi)$. A turn $T$ is a vertical segment $\{x_0\} \times [y_0, y_0 + \alpha]$ and we set $x(T) = x_0$, $y_-(T) = y_0$, $y_+(T) = y_0 + \alpha$. Similarly, given a lozenge $\cpar$ along the diagonal section $x = v$ it spans a set of $y$-coordinates of the form $[y_1,y_1 + 1]$ in which case we let $x(\cpar) = v$, $y_-(\cpar) = y_1$, $y_+(\cpar) = y_1 + 1$. We now introduce an interaction between a turn $T \in \cT(\pi)$ and horizontal lozenges $\cpar$ which lie directly above it, given by the weight
\begin{align} \label{eq:turnwt}
v_{\alpha,t}(\pi,T) = \prod\limits_{\substack{\cpar: x(\cpar) = x(T) \\ y_-(\cpar) \ge y_+(T)}} \frac{1 - t^{y_+(\cpar) - y_+(T)}}{1 - t^{y_-(\cpar) - y_-(T)}}.
\end{align}
Note that if $\alpha = 1$, then the weight is identically $1$. If $\alpha < 1$ ($> 1$), then each fraction in (\ref{eq:turnwt}) is $> 1$ ($< 1$). With this setup, we now define the Macdonald weight:
\[ w_{\alpha,t}(\pi) = \prod_{\substack{T \in \cT(\pi) \\ T ~\mbox{\footnotesize is external}}} v_{\alpha,t}(\pi,T) \prod_{\substack{T \in \cT(\pi) \\ T ~\mbox{\footnotesize is internal}}} v_{\alpha,t}(\pi,T)^{-1}.\]
In words, if $\alpha < 1$ then the weight $w_{\alpha,t}$ favors external turns over internal turns, and the strength of the preference is amplified by the presence of horizontal lozenges directly above the turn. Decreasing $\alpha$ further exaggerates this interaction. For $\alpha > 1$, this preference is reversed for internal and external turns, and increasing $\alpha$ exaggerates the interaction.

Analogues of Theorems \ref{thm:LLN} and \ref{thm:GFF} exist for the Macdonald RPPs.

\begin{definition}
Let $\vec{s} = (\ldots,s_{-1},s_0,s_1,\ldots)$ be a $p$-periodic, bi-infinite sequence of positive numbers such that $s_0 \cdots s_p = 1$. Denote by $\P^{B,r,\vec{s}}_{\alpha,\ft}$ the probability measure on $\cP_B$ defined by (\ref{eq:macrv}) where
\[ r_v = s_v r, \quad t = r^\ft, \quad q = t^\alpha \]
given that the weights are summable.
\end{definition}

This generalizes the family $\P^{B,r\vec{s}}$ defined earlier which corresponds to $\alpha = 1$ (in which case the value of $\ft$ is immaterial). The limit regime we consider is a generalization of the Limit Conditions for $\P^{B,r,\vec{s}}_{\alpha,\ft}$ where we fix $\vec{s},\alpha,\ft$.

\begin{theorem} \label{thm:macLLN}
Suppose $\P^{B,r,\vec{s}}_{\alpha,\ft}$ satisfy Limit Conditions with fixed $\vec{s},\alpha,\ft$ such that $\cB \in \fB^\Delta(\vec{s})$. Then there is a deterministic Lipschitz $1$ function $\cH: I \times \R \to \R$ \emph{independent of $\alpha,\ft$} such that we have the convergence
\[ \e h\left( \left\lfloor \frac{x}{\e} \right\rfloor, \frac{y}{\e}\right) \to \frac{1}{\alpha}\cH(x,y) \]
of measures on $y \in \R$, weakly in probability as $\e \to 0$ for all $x \in I \setminus \{V_\ell\}_{\ell=0}^n$ (recall these are the differentiable points of the continuous, piecewise linear limit $\cB$ of back walls). An explicit description of this height function is given in Section \ref{sec:limgff}.
\end{theorem} 

\begin{theorem} \label{thm:macGFF}
Suppose $\P^{B,r,\vec{s}}_{\alpha,\ft}$ satisfy Limit Conditions with fixed $\vec{s},\alpha,\ft$ such that $\cB \in \fB^\Delta(\vec{s})$. Then the map $\zeta(x,y) = e^x z(x,y)$, where $z$ is defined by (\ref{eq:locprop}), is a homeomorphism from the liquid region to $\HH$ \emph{independent of $\alpha,\ft$}. Moreover, the centered, rescaled height function $\alpha \ft\sqrt{\pi}\,\overline{h}\left(\lfloor \frac{x}{\e} \rfloor,\frac{y}{\e} \right)$ converges to the $\zeta$-pullback of the GFF in the sense that we have the following convergence in distribution of the random family
\[ \left\{ \alpha \ft \sqrt{\pi} \int \overline{h}\left( \left\lfloor \frac{x}{\e} \right\rfloor,\frac{y}{\e} \right) e^{-k\ft y} \, dy \right\}_{x\in I, k \in \Z_{>0}} \to \left\{\int \fH(\zeta(x,y)) e^{-k\ft y} \, dy \right\}_{x \in I, k \in \Z_{>0}}\]
for all $x \in I \setminus \{V_\ell\}_{\ell=0}^n$.
\end{theorem}

\begin{remark}
We prove stronger statements (see Theorems \ref{thm:MACLLN} and \ref{thm:MACGFF}) which remove the restriction $x \in I \setminus \{V_\ell\}_{\ell=0}^n$ in exchange for a microscopic separation condition. In these improved theorems, we replace $\lfloor x/\e \rfloor$ with some sequence $x(\e)$ such that $\e x(\e) \to x$ for \emph{any} $x \in I$ with the caveat that certain $\e x(\e)$ need to be separated by some microscopic distance from certain \emph{singular points}, see Definition \ref{def:spt}. This separation condition can be removed for the $\alpha = 1$ case, and is also unnecessary whenever $\ft = k$ for any positive integer $k > 0$. We expect that the statement of Theorem \ref{thm:macGFF} should still hold in the absence of this condition. Due to technical complications, we did not pursue this refinement.
\end{remark}

\begin{notation}
Let $B: I \to \R$ be a back wall for some RPP. Denote $I_V = I \cap \Z$ and $I_E = I \cap (\Z + \frac{1}{2})$. For back walls $B^a$ denoted by superscripts, we denote the corresponding sets with superscripts: $I^a$ (domain of $B^a)$, $I^a_V = I^a \cap \Z$, $I^a_E = I^a \cap (\Z + \frac{1}{2})$.
\end{notation}

\section{Joint Expectations of Observables} \label{sec:obs}
The main goal of this section is to obtain formulas for expectations associated to the height function. Consider
\[ \wp_k(\lambda;q,t) = (1-t^{-k}) \sum_{i=1}^{\ell(\lambda)} q^{k\lambda_i} t^{k(-i+1)} + t^{-k\ell(\lambda)} \]
where $\lambda = (\lambda_1,\lambda_2,\ldots)$ is a partition, $\ell(\lambda)$ denotes the number of indices $i$ such that $\lambda_i \ne 0$, and $k \in \Z_{\ge 0}$. The following proposition gives a connection between $\wp_k$ and height functions.

\begin{proposition} \label{prop:aht}
Consider a plane partition $\pi \in \cP_B$. Fix $\alpha > 0$, and let $h$ be the height function. Then
\begin{align*}
\int_{-\infty}^\infty \! h(x,y) t^{ky} \, dy &= \frac{t^{kB(x)}}{\alpha k^2 (\log t)^2} \cdot \wp_k(\pi^x;t^\alpha,t).
\end{align*}
\end{proposition}

The main result of this section (stated in Theorem \ref{thm:obs}) is a formula for the joint expectation
\begin{align} \label{eq:joint_moments}
\E[ \wp_{k_1}(\pi^{x_1}) \cdots \wp_{k_m}(\pi^{x_m}) ],
\end{align} 
where $k_1,\ldots,k_m \in \Z_{\ge 0}$, $x_1,\ldots,x_m \in I$, and $(\pi^x)_{x\in I}$ are the diagonals of $\pi \sim \P^{B,r,\vec{s}}_{\alpha,\ft}$. This gives us an expression whose asymptotics are accessible, and the aforementioned proposition provides the link to interpret these asymptotics in terms of the height function.

To arrive at a formula for (\ref{eq:joint_moments}), we establish a more general expression for observables of \emph{formal} Macdonald processes in Section \ref{sec:macproc} (Theorem \ref{thm:multi} and Corollary \ref{cor:multi}). Here, we combine and generalize the approaches of \cite{BC}, \cite{BCGS} and \cite{GZ}. In Section \ref{ssec:rppobs}, we specialize these formal expressions to the case of $\P^{B,r,\vec{s}}_{\alpha,\ft}$ to prove Theorem \ref{thm:obs}; our formula for (\ref{eq:joint_moments}). We note that the formal expressions obtained in Section \ref{sec:macproc} are applicable to a much more general setting than ours.

Before proceeding, we prove Proposition \ref{prop:aht}.

\begin{proof}[Proof of Proposition \ref{prop:aht}]
As defined in Section \ref{sec:modelresults}, the height function at $(x,y)$ is the total length of vertical line segments beneath a point $(x,y)$. Let $Y_i$ denote the $y_+$ coordinate of the $i$th highest $\cpar$ lozenge along the $x$-diagonal section. Then $Y_i = \alpha \pi_i^x - i + 1 + B(x)$. We claim that the height function is given by the formula
\begin{align} \label{eq:aht}
h(x,y) = \frac{1}{\alpha}\left( y - B(x) + \int_{-1}^0 \! \left| \left\{i \ge 1:  Y_i + u  \ge y \right\}\right| \, du \right).
\end{align}
To see how to obtain (\ref{eq:aht}), note that the integral counts the total number of $\cpar$ lozenges lying above $(x,y)$, counting non-integer amounts of $\cpar$ if $(x,y)$ lies on the lozenge by the vertical distance from $(x,y)$ to the top of $\cpar$. For $y$ large, the height function is just $\frac{1}{\alpha}(y - B(x))$. As we decrease $y$, the integral term in (\ref{eq:aht}) enters since no vertical segments are added when passsing through a $\cpar$ lozenge. This proves the claim.

Let $N = \ell(\pi^x)$, and note that $\partial_y h = \frac{1}{\alpha}\left( 1 - \sum_{i=1}^\infty \1[Y_i - 1, Y_i] \right)$ which is $0$ on $(-\infty, Y_{N+1}) = (-\infty,-N+B(x))$. Then
\begin{align*}
\int_{-\infty}^\infty \! h(x,y) t^{ky} \, dy &= -\frac{1}{k \log t} \int_{-\infty}^\infty \! (\partial_y h)(x,y) t^{ky} \, dy \\
&= -\frac{1}{\alpha k \log t} \int_{Y_{N+1}}^\infty \left( 1 - \sum_{i=1}^N \1[ Y_i - 1, Y_i](y) \right) t^{ky} \, dy \\
&= \frac{1}{\alpha k^2 (\log t)^2} \left( t^{k Y_{N+1}} + (1 - t^{-k}) \sum_{i=1}^N t^{k Y_i} \right) \\
&= \frac{t^{kB(x)}}{\alpha k^2 (\log t)^2}\cdot \wp_k(\pi^x;t^\alpha,t).
\end{align*}
\end{proof}

\subsection{Formal Expectations} \label{sec:macproc} 
In this subsection, we obtain formal expressions for observables of formal Macdonald processes. In Sections \ref{sec:sym}, \ref{sec:gtop}, \ref{sec:pairres}, we provide some background on symmetric functions and notions to give rigorous meaning to the formal expressions we work with. In Section \ref{sec:formmp}, we define the formal Macdonald process and associated objects. In Section \ref{sec:negut}, we give a formal expression for single cut observables of formal Macdonald processes, originally obtained in \cite{GZ}. In Section \ref{sec:multi}, we extend these formulas to multicut observables of formal Macdonald processes.

\subsubsection{Symmetric Functions} \label{sec:sym}
The following background on symmetric functions and additional details can be found in \cite[Chapters I \& VI]{Mac}.

Let $\Y$ denote the set of partitions. Recall that we represent $\lambda \in \Y$ as the nondecreasing sequence $(\l_1,\l_2,\ldots)$ of its parts and denote by $\ell(\lambda)$ the number indices $i$ such that $\lambda_i \neq 0$. Given $\mu,\l \in \Y$, we write $\mu \prec \l$ if $\ell(\mu),\ell(\l) \le N$ and
\[ \l_1 \ge \mu_1 \ge \l_2 \ge \mu_2 \ge \cdots \ge \l_N \ge \mu_N. \]
Given a countably infinite set $X = (X_1,X_2,\ldots)$ of variables, let $\Lambda_X$ denote the algebra of symmetric functions on $X$ over $\C$. For sets $X^{(1)},\ldots,X^{(n)}$ of variables, let $\Lambda_{(X^{(1)},\ldots,X^{(n)})}$ denote the algebra of symmetric functions on the disjoint union of these sets.

Recall the power symmetric functions $p_0(X) = 1$ and
\[ p_k(X) = \sum_{i \ge 1} X_i^k, ~~~k \in \Z_{>0}. \]
These symmetric functions are generators of the algebra $\Lambda_X$. For each $\lambda \in \Y$, define
\[ p_\lambda(X) = \prod_{i = 1}^{\ell(\lambda)} p_{\lambda_i}(X). \]
Then $\{p_\lambda(X)\}_{\lambda \in \Y}$ forms a linear basis of $\Lambda_X$. Fixing $0 < q,t < 1$, we have the scalar product
\[ \langle p_\lambda, p_\mu \rangle = \delta_{\lambda\mu} \prod_{i=1}^{\ell(\lambda)} \frac{1 - q^{\lambda_i}}{1 - t^{\lambda_i}} \prod_{i=1}^\infty i^{m_i(\lambda)} m_i(\lambda)!\]
where $m_i(\lambda)$ is the multiplicity of $i$ in $\lambda$.

The \textit{normalized Macdonald symmetric functions} $\{P_\lambda(X;q,t)\}_{\lambda \in \Y}$ are the unique (homogeneous) symmetric functions satisfying
\[\langle P_\lambda(X;q,t), P_\mu(X;q,t) \rangle = 0\]
for $\lambda \neq \mu$ and with leading monomial $X_1^{\lambda_1} X_2^{\lambda_2} \cdots$ with respect to lexicographical ordering of the powers $(\lambda_1,\lambda_2,\ldots)$. This implies that $\{P_\lambda(X;q,t)\}_{\l \in \Y}$ forms a linear basis for $\L_X$. Let $Q_\lambda(X;q,t)$ represent the multiple of $P_\lambda(X;q,t)$ satisfying
\[ \langle P_\lambda(X;q,t), Q_\lambda(X;q,t) \rangle = 1. \]
For $\lambda,\mu \in \Y$, the \textit{skew Macdonald symmetric functions} $P_{\lambda/\mu}(X;q,t)$, $Q_{\lambda/\mu}(X;q,t)$ are uniquely defined by
\begin{align*}
P_\lambda(X,Y) &= \sum_{\mu \in \Y} P_{\lambda/\mu}(X) P_\mu(Y), \\
Q_\lambda(X,Y) &= \sum_{\mu \in \Y} Q_{\lambda/\mu}(X) Q_\mu(Y).
\end{align*}
For a single variable $x$, we have the following expressions for skew Macdonald symmetric functions. Let $f(u) = \frac{(tu;q)_\infty}{(qu;q)_\infty}$ with $(u;q)_\infty := \prod_{i \geq 0}(1 - uq^i)$,
\begin{align} \label{eq:PQmon}
P_{\lambda/\mu}(x) = \delta_{\mu \prec \lambda} \psi_{\lambda/\mu}(q,t) x^{|\lambda| - |\mu|}
~~~~~\mbox{and}~~~~~ Q_{\lambda/\mu}(x) = \delta_{\mu \prec \lambda}  \phi_{\lambda/\mu}(q,t) x^{|\lambda| - |\mu|}
\end{align}
where the coefficients are
\begin{align} \label{eq:psi}
\psi_{\lambda/\mu}(q,t) &= \prod_{1 \leq i \leq j \leq \ell(\lambda)} \frac{f(q^{\mu_i - \mu_j} t^{j-i}) f(q^{\lambda_i - \lambda_{j+1}} t^{j-i})}{f(q^{\lambda_i - \mu_j} t^{j-i}) f(q^{\mu_i - \lambda_{j+1}} t^{j-i})}, \\ \label{eq:phi}
\phi_{\lambda/\mu}(q,t) &= \prod_{1 \leq i \leq j \leq \ell(\mu)} \frac{f(q^{\lambda_i - \lambda_j} t^{j-i}) f(q^{\mu_i - \mu_{j+1}} t^{j-i})}{f(q^{\lambda_i - \mu_j} t^{j-i}) f(q^{\mu_i - \lambda_{j+1}} t^{j-i})}.
\end{align}
The skew Macdonald symmetric functions satisfy the branching rule:
\begin{align} \label{eq:branch}
P_{\lambda/\nu}(X,Y) = \sum_{\mu \in \Y} P_{\lambda/\mu}(X) P_{\mu/\nu}(Y), ~~Q_{\lambda/\nu}(X,Y) = \sum_{\mu \in \Y} Q_{\lambda/\mu}(X) Q_{\mu/\nu}(Y)
\end{align}
for any $\lambda,\nu \in \Y$.

We say that a unital algebra homomorphism $\rho: \Lambda_X \to \C$ is a \textit{specialization}. Given a specialization $\rho$ and $f \in \L_X$, we write $f(\rho)$ instead of $\rho(f)$ in view of the special case of function evaluation. The specializations we are interested in will have the following form. Take a sequence $\{a_i\}_{i=1}^\infty$ of nonnegative real numbers such that $a_1 \ge a_2 \ge \cdots$ and $\sum_{i=1}^\infty a_i < \infty$, define $\rho$ by
\[ p_n(\rho) = \sum_{i=1}^\infty a_i^n \]
for $n > 0$. This uniquely determines the specialization $\rho$ because the power symmetric functions generate the algebra of symmetric functions. For such specializations, we may write $\rho = (a_1,a_2,a_3,\ldots)$. If the only nonzero members of the sequence are $a_1,\ldots,a_N$, we may write $\rho = (a_1,\ldots,a_N)$.

A specialization $\rho$ is $(q,t)$-\textit{Macdonald-positive} if $P_\lambda(\rho;q,t) \geq 0$ for all partitions $\lambda$. The aforementioned specialization $\rho = (a_1,a_2,\ldots)$ with $a_i \geq 0$ for all $i \geq 1$ is Macdonald positive, as follows from the nonnegativity of (\ref{eq:PQmon}), (\ref{eq:psi}), (\ref{eq:phi}).

\subsubsection{Graded Topology} \label{sec:gtop}
Let $F$ be a field and $\cA$ be a ($\Z_{\geq 0}$-)graded algebra over $F$. Let $\cA_n$ denote the $n$th homogeneous component of $\cA$. Throughout this section, let us assume that all of our graded algebras have $\dim \cA_n < \infty$ for every $n \geq 0$.

\begin{definition}
Given $a \in \cA$, define $\ldeg(a)$ to be the minimum degree among the homogeneous components of $a$.  The \emph{graded topology} is the topology on $\cA$ where a sequence $a_n \in \cA$ converges to $a \in \cA$ if and only if
\[ \ldeg (a_n - a) \to \infty \]
as $n\to\infty$. Denote the completion of $\cA$ under this topology by $\widehat{\cA}$.
\end{definition}

The completion $\widehat{\cA}$ consists of formal sums $\sum_{n=1}^\infty a_n$ where $a_n \in \cA_n$. Given two graded algebras $\cA$ and $\cA'$ over $F$, we give the following grading to $\cA \otimes_F \cA'$. If $a \in \cA_m$ and $a' \in \cA_n'$, then $a \otimes a' \in (\cA \otimes_F \cA')_{m+n}$.

For a field $F \supset \C$ and a graded algebra $\cA$ over $\C$, denote by $\cA[F]$ the graded algebra $\cA \otimes_\C F$ over $F$; i.e. the extension of scalars from $\C$ to $F$. Given graded algebras $\cA^{(1)},\ldots,\cA^{(k)}$ over $\C$, we denote the completion of $(\cA^{(1)} \otimes \cdots \otimes \cA^{(n)})[F]$ under the graded topology by
\[ \cA^{(1)} \cotimes \cdots \cotimes \cA^{(k)}[F] \quad \mbox{or} \quad \widehat{\bigotimes}_{i=1}^k \cA^{(i)}[F]. \]

Let $\L_X[F]$ denote the $F$-algebra of symmetric functions in $X = \{x_1,x_2,\ldots\}$, a set of variables, with coefficients in $F$. Take the natural grading on $\Lambda_X[F]$ in which $(\Lambda_X[F])_n$ is spanned by monomials of total degree $n$. Given disjoint ordered sets of variables $Z_1,\ldots,Z_n$ with $Z_i = (z_{i,1},\ldots,z_{i,k_i})$, let $\cL(Z_1,\ldots,Z_n)$ denote the field of formal Laurent series in the variables
\[ \bigcup_{i=1}^n \left\{ \frac{z_{i,1}}{z_{i,2}}, \ldots, \frac{z_{i,k_i-1}}{z_{i,k_i}}, z_{i,k_i} \right\}. \]

The space $\widehat{\bigotimes}_{i=1}^k \Lambda_{X^i}[F]$ consists of formal sums
\[ \sum_{\lambda^1,\ldots,\lambda^N \in \Y} c_{\lambda^1,\ldots,\lambda^N} P_{\lambda^1}(X^1) \cdots P_{\lambda^N}(X^N) \]
where $c_{\lambda^1,\ldots,\lambda^N} \in F$.

For fields $\C \subset F_1 \subset F_2$, $1 \leq k \leq N$, there is the natural inclusion map
\begin{equation} \label{eq:inclu}
\widehat{\bigotimes}_{i=1}^k \Lambda_{X^i}[F_1] \hookrightarrow \widehat{\bigotimes}_{i=1}^k \Lambda_{X^i}[F_2].
\end{equation}
We also have consistency
\begin{equation} \label{eq:consis}
\L_{X^1}[F] \otimes_F \cdots \otimes_F \L_{X^N}[F] \cong \L_{X^1} \otimes \cdots \otimes \L_{X^N} [F].
\end{equation}

\begin{definition}
The \emph{projection map} $\pi_X^n: \widehat{\Lambda_X} \to \Lambda_{\{x_1,\ldots,x_n\}}$ is defined as the continuous map sending $x_{n+1},x_{n+2},\ldots$ to $0$ and $x_i$ to $x_i$ for $i = 1,\ldots,n$.
\end{definition}

For a field $F\supset \C$ and a graded algebra $\cA$ over $\C$, we can extend the domain of the projection
\[\pi_X^n: \cA \cotimes \Lambda_X[F] \to \cA \cotimes \Lambda_{\{x_1,\ldots,x_n\}}[F]\]
by identifying with $1_{\cA} \otimes \pi_X^n$ then extending by continuity under the graded topology.
  
\begin{definition}
Let $\cA$ and $\cA'$ be graded algebras over $\C$ and $\{a_{n,j}\}_j$ be a basis for $\cA_n$ for each $n \ge 0$. We say that an element $f \in \cA \cotimes \cA'[F]$ is $\cA$-\emph{projective} if
\[ f = \sum_{n,j} a_{n,j} \otimes \alpha_{n,j}', \quad \alpha_{n,j}' \in \cA_n' \]
such that $\lim_{n\to\infty} \min_j \ldeg(\alpha_{n,j}') = \infty$. This property is independent of the choice of basis.
\end{definition}

Elements which are $\cA$-projective are closed under addition and multiplication and form a subalgebra of $\cA \cotimes \cA'[F]$. If $\cA = \Lambda_X$, denote the algebra of $\Lambda_X$-projective elements by $\mathcal{P}_X(\Lambda_X \cotimes \cA'[F])$.

\subsubsection{Macdonald Pairing and Residue} \label{sec:pairres}
Recall the Macdonald scalar product determined by
\[ \langle P_\lambda, Q_\mu \rangle = \delta_{\lambda\mu}. \]
\begin{definition}
Let $\cA, \cA'$ be graded algebras over $\C$. Fix a field $F \supset \C$, and let the \textit{Macdonald pairing} be the bilinear map $\langle \cdot, \cdot \rangle_Y: (\cA \otimes \Lambda_X)[F] \times (\Lambda_X \otimes \cA')[F] \to \cA \otimes \cA'[F]$ defined by
\[ \langle a \otimes P_\lambda, Q_\mu \otimes b \rangle_X := \langle P_\lambda, Q_\mu \rangle a \otimes b = \delta_{\lambda \mu} a \otimes b. \]
\end{definition}
This pairing does not extend by continuity to the completions of the domain. However, the pairing does extend continuously to
\[ \mathcal{P}_X(\cA \cotimes \Lambda_X[F]) \times (\Lambda_X \cotimes \cA'[F]). \]

\begin{definition}
Given an ordered set $Z = (z_1,\ldots,z_k)$ of variables, denote by $\oint \, dZ: \cL(Z) \to \C$ the \textit{residue operator} which takes an element of $\cL(Z)$ and returns the coefficient of $(z_1 \cdots z_k)^{-1}$. For $\oint \, dZ$ applied to $f \in \cL(Z)$ we write $\oint f \, dZ$ or $\oint \, dZ \cdot f$.
\end{definition}

As with the projection map, the residue operator can act on larger domains. For example, we can extend
\begin{align}
\oint dZ: \widehat{\cA}[\cL(Z,W^1,\ldots,W^k)] \to \widehat{\cA}[\cL(W^1,\ldots,W^k)]
\end{align}
by the action $1_{\cA} \otimes \oint dZ$ then extension by continuity. In this case, $\oint dZ$ preserves the degree of homogeneous elements. In particular, if we replace $\widehat{\cA}$ with $\cA \cotimes \cA'$, we have that $\oint dZ$ preserves $\cA$-projectivity.

The residue operator commutes with continuous maps under the graded topology.

\begin{lemma} \label{lem:rescom}
Let $\cA$, $\cA'$ be graded algebras over $\C$, and let $\varphi: \widehat{\cA} \to \widehat{\cA}'$ be a continuous map which extends naturally to a continuous map $\widehat{\cA}[\cL(Z,W)] \to \widehat{\cA}'[\cL(Z,W)]$. Then
\[ \varphi \circ \left( \oint \, dZ \right) = \left( \oint \, dZ \right) \circ \varphi. \]
\end{lemma}

\begin{lemma}
Let $\cA$, $\cA'$ be graded algebras over $\C$, let $f \in \cA\cotimes \L_X[\cL(Z)]$ and $g \in \L_X \cotimes \cA'[\cL(W)]$. If $f$ is $\Lambda_X$-projective, then
\begin{align} \label{eq:rpcom1}
\left \langle \oint f \,dZ, g \right \rangle_X = \oint \langle f, g \rangle_X \, dZ, \\ \label{eq:rpcom2}
\left \langle f, \oint g \,dW \right \rangle_X = \oint \langle f, g \rangle_X \, dW
\end{align}
Since the residue operator preserves projectivity, the left hand sides of the equalities above are valid expressions.
\end{lemma}
\begin{proof}
For arbitrary $g \in \L_X \cotimes \cA'[\cL(W)]$, the map $\langle \cdot, g \rangle_X$ is continuous on $\cP_X(\cA \otimes \L_X[\cL(Z)])$. By Lemma \ref{lem:rescom}, (\ref{eq:rpcom1}) follows. For $f \in \cP_X(\cA \cotimes \L_X(\cL(Z)])$, the map $\langle f, \cdot \rangle$ is continuous on $\L_X \cotimes \cA'[\cL(W)]$. By Lemma \ref{lem:rescom}, (\ref{eq:rpcom2}) follows.
\end{proof}

\subsubsection{Formal Macdonald Processes} \label{sec:formmp}
Let $X,Y$ be countable sets of variables. Fix $0 < q,t<1$ throughout this section. Define the following element of $\Lambda_X \cotimes \Lambda_Y$
\[\Pi(X,Y) := \prod_{x\in X,y \in Y} \frac{(txy;q)_\infty}{(xy;q)_\infty} \]
From \cite[Chapter VI, Sections 2 \& 4]{Mac}, we have the following equalities
\[\Pi(X,Y) = \sum_{\lambda\in\Y} P_\lambda(X) Q_\lambda(Y) = \exp\left( \sum_{n=1}^\infty \frac{1 - t^n}{1 - q^n} \frac{1}{n} p_n(X) p_n(Y) \right).\]
Define the following element of $\Lambda_X \cotimes \Lambda_Y$ obtained by taking $q = 0$ above
\begin{align} \label{eq:Hnorm}
H(X,Y;t) = \prod_{x \in X, y \in Y} \frac{1 - txy}{1 - xy} = \exp\left( \sum_{n=1}^\infty \frac{1 - t^n}{n} p_n(X) p_n(Y) \right).
\end{align}
Given countable sets of variables $X^1,X^2$, the following splitting equality holds
\begin{equation} \label{eq:split}
\Pi((X^1,X^2),Y) = \Pi(X^1,Y) \Pi(X^2,Y)
\end{equation}
and likewise for $H(\cdot,\cdot;t)$. There is also an inversion equality
\begin{equation} \label{eq:invH}
H(X,Y;t)^{-1} = H(tX,Y;t^{-1})
\end{equation}
where by $tX$ we mean the variable set $\{tx\}_{x\in X}$.
 
\begin{definition}
Fix a positive integer $N$ and let $\vec{U} = (U^1,\ldots,U^N)$ and $\vec{V} = (V^1,\ldots,V^N)$ be ordered $N$-tuples of countable sets of variables. A \textit{formal Macdonald process} is a formal probability measure on $\Y^N$ valued in $\widehat{\bigotimes}_{i=1}^N (\Lambda_{U^i} \otimes \Lambda_{V^i})$ with the assignment
\begin{multline*}
\MP_{\vec{U},\vec{V}}^{\textbf{f}}(\vec{\lambda}) = \mathscr{Z}^{-1}P_{\lambda^1}(U^1) \left( \sum_{\mu \in \Y} Q_{\lambda^1/\mu}(V^1) P_{\lambda^2/\mu}(U^2) \right) \cdots \left( \sum_{\mu \in \Y} Q_{\lambda^{N-1}/\mu}(V^{N-1}) P_{\lambda^N/\mu}(U^N) \right) Q_{\lambda^N}(V^N)
\end{multline*}
where $\vec{\l} = (\l^1,\ldots,\l^N)$ and $\mathscr{Z} \in \widehat{\bigotimes}_{i=1}^N (\Lambda_{U^i} \otimes \Lambda_{V^i})$ is the normalization constant for which the sum over $\vec{\l} \in \Y^N$ gives unity.
\end{definition}

From \cite[Section 3]{BCGS},
\begin{equation} \label{eq:Znorm}
\mathscr{Z} = \prod_{1 \leq i \leq j \leq N} \Pi(U^i,V^j).
\end{equation}
  
In terms of the pairing, the formal Macdonald process can be expressed as
\begin{align} \label{eq:mppair}
\MP_{\vec{U},\vec{V}}^{\textbf{f}}(\vec{\l}) = \mathscr{Z}^{-1} P_{\l^1}(U^1) \left(\prod_{i = 1}^{N-1} \langle Q_{\l^i}(V^i, Y^i), P_{\l^{i+1}}(Y^i,U^{i+1}) \rangle_{Y^i} \right) Q_{\l^N}(V^N).
\end{align}
This is an immediate consequence of the branching rule (\ref{eq:branch}).
  
The $\Pi$'s introduced earlier also relate well with the pairing
\[\langle \Pi(X^1,Y), \Pi(Y,X^2) \rangle = \Pi(X^1,X^2). \]
Since the power symmetric functions from an algebraic basis for $\Lambda_X[\cL(Z)]$, this relation can be further extended as follows. Take graded algebras $\cA$ and $\cA'$ over $\cL(Z)$ with $Z = (z_1,\ldots,z_k)$, and sequences $\{a_n\}, \{a_n'\}$ in $\cA$ and $\cA'$ respectively such that $\ldeg(a_n), \ldeg(a_n') \to \infty$ as $n\to\infty$. Then
\begin{align} \label{eq:exppair}
\left\langle \exp\left( \sum_{n=1}^\infty \frac{a_n}{n} p_n(Y) \right), \exp\left( \sum_{n=1}^\infty \frac{a_n'}{n} p_n(Y) \right) \right\rangle = \exp\left( \sum_{n=1}^\infty \frac{1-q^n}{1-t^n} \frac{a_n a_n'}{n} \right).
\end{align}
See \cite[Proposition 2.3]{BCGS} for further details.

\begin{lemma}
Let $X^1,X^2,X^3,X^4,Y$ be countable sets of variables. Then
\begin{align} \label{eq:HPi}
\begin{multlined}
\langle H(X^1,Y;t^{-1}) \Pi(X^2,Y), H(X^3,Y;t^{-1}) \Pi(X^4,Y) \rangle_Y \\
= H(X^1,X^4; t^{-1}) H(X^2,X^3; t^{-1}) \Pi(X^2,X^4) \prod \frac{(1 - x_1x_3)(1 - \frac{q}{t} x_1x_3)}{(1 - \frac{1}{t}x_1x_3)(1 - qx_1x_3)}
\end{multlined}
\end{align}
where the product is over $x_i \in X^i$ for $i = 1,2,3,4$. The expression $(1 - qx_1x_3)^{-1}$ is interpreted as the formal power series $\sum_{n=0}^\infty (qx_1x_3)^n$ and similarly for $(1 - t^{-1}x_1x_3)^{-1}$.
\end{lemma}
\begin{proof}
Use (\ref{eq:exppair}) with
\begin{align*}
a_n &= (1 - t^{-n}) p_n(X^1) + \frac{1 - t^n}{1 - q^n} p_n(X^2) \\
a_n' &= (1 - t^{-n}) p_n(X^3) + \frac{1 - t^n}{1 - q^n} p_n(X^4).
\end{align*}
\end{proof}

\subsubsection{Negut's Operator} \label{sec:negut}
Define the continuous linear operator $D_{-k}^X: \widehat{\Lambda_X} \to \widehat{\Lambda_X}$ by
\[ D_{-k}^X P_\lambda(X;q,t) = \wp_k(\lambda;q,t) P_\lambda(X;q,t).\]
This operator was studied in \cite{N} and an integral form for this operator was obtained in \cite{GZ}. The action of this operator can be given in terms of the residue operator. We first introduce notation to abbreviate the expression. Let $Z = (z_1,\ldots,z_k)$ be an ordered set of variables. Define
\begin{align} \label{eq:DZ}
DZ = \frac{(-1)^{|Z|-1}}{(2\pi\i)^{|Z|}} \frac{\sum_{i=1}^k \frac{z_k}{z_i} \frac{t^{k-i}}{q^{k-i}}}{(1 - \frac{tz_2}{qz_1}) \cdots (1 - \frac{tz_k}{qz_{k-1}})} \prod_{i < j} \frac{(1 - \frac{z_i}{z_j})(1 - \frac{qz_i}{tz_j})}{(1 - \frac{z_i}{tz_j})(1 - \frac{qz_i}{z_j})} \prod_{i=1}^k \frac{dz_k}{z_k}.
\end{align}
For the instances of $(1 - v)^{-1}$ in the  expression, we mean the power series expansion into $\sum_{n \geq 0} v^n$. We adopt the shorthand notation $Z^{-1} = (z_1^{-1},\ldots,z_k^{-1})$.

\begin{proposition}
Let $X$ and $Y$ be countable sets of variables. Then
\begin{align} \label{eq:negut}
D_{-k}^X \Pi(X,Y) = \Pi(X,Y) \oint DZ \cdot H(qZ^{-1},X;t^{-1}) H(Y,q^{-1}Z;t)^{-1}.
\end{align}
\end{proposition}
\begin{proof}
From \cite[Proposition 4.10]{GZ}, we have (\ref{eq:negut}) where instead of a set of variables $Y$ we have some fixed set $\{u_1,\ldots,u_n\}$ of complex numbers, and $X$ is still a countable set of variables. Here we have $\oint: \widehat{\Lambda_X}[\cL(Z)] \to \widehat{\Lambda_X}$. The goal is to extend this to a formal equality on $\Lambda_X \cotimes \Lambda_Y$ for $Y$ an arbitrary countable set of variables.

We can replace (\ref{eq:negut}) with a finite set of variables $Y^{(n)} = \{y_1,\ldots,y_n\}$ instead of fixed complex numbers. In such a setting, we must consider the residue operator as a map $\oint \, dZ:\Lambda_X \cotimes \Lambda_{y_1,\ldots,y_n}[\cL(Z)] \to \Lambda_X \cotimes \Lambda_{y_1,\ldots,y_n}$.
  
Note that if $f,g \in \Lambda_X \cotimes \Lambda_Y$ such that $\pi_n^Y f = \pi_n^Y g$ for all $n$, then $f = g$. One then sees that (\ref{eq:negut}) holds formally for arbitrary countable sets of variables $X,Y$. In this setting, the residue operator takes $\Lambda_X \cotimes \Lambda_Y[\cL(Z)]$ to $\Lambda_X \cotimes \Lambda_Y$.
\end{proof}
  
\subsubsection{Formal Multicut Expectations} \label{sec:multi}
We obtain formulas for multicut expectations of formal Macdonald processes. The idea is to repeated apply the operators $D^X_{-k}$ to $\Pi$.

\begin{theorem} \label{thm:multi}
The following formal identity holds for any nonnegative integers $k_1,\ldots,k_N$
\begin{equation*}
\begin{multlined}
\E_{\MP_{\vec{U},\vec{V}}^{\textbf{\emph{f}}}}[ \wp_{k_1}(\l^1;q,t) \cdots \wp_{k_N}(\l^N;q,t) ] = \oint DZ_1 \cdots DZ_N \\
\times \prod_{1 \leq i \leq j \leq N} H(U^i,qZ_j^{-1};t^{-1}) H(q^{-1}Z_i, V^j ;t)^{-1} \prod_{1 \leq i < j \leq N} C(Z_i,Z_j)
\end{multlined}
\end{equation*}
where $|Z_i| = k_i$ and
\begin{align} \label{eq:czw}
C(Z,W) = \prod_{i,j} \frac{(1 - t^{-1} q z_i/w_j)(1 - z_i/w_j)}{(1 - qz_i/w_j)(1 - t^{-1} z_i/w_j)}
\end{align}
for sets of variables $Z = (z_1,\ldots,z_k)$ and $W = (w_1,\ldots,w_\ell)$.
\end{theorem}
  
This theorem implies a more general result in which the $\wp_{k_i}(\l^i)$ may be taken to higher powers than $1$ in the expectation.
\begin{corollary} \label{cor:multi}
Let $1 \leq x_1 \leq \cdots \leq x_m \leq N$ and $k_1, \cdots, k_m > 0$ be integers. Then
\begin{align}
\begin{multlined} \label{eq:multi}
\E_{\MP_{\vec{U},\vec{V}}^{\textbf{\emph{f}}}}[\wp_{k_1}(\l^{x_1};q,t) \cdots \wp_{k_m}(\l^{x_m};q,t)] = \oint DZ_1 \cdots DZ_m \\
\times \prod_{a=1}^m \prod_{\substack{1 \leq i \leq x_a \\ x_a \leq j \leq N}} H(U^i,q Z_a^{-1}; t^{-1}) H(q^{-1}Z_a,V^j; t)^{-1} \prod_{a < b} C(Z_a,Z_b).
\end{multlined}
\end{align}
where $|Z_a| = k_a$.
\end{corollary}

\begin{proof}[Proof of Theorem \ref{thm:multi}]
Choose nonnegative integers $k_1,\ldots,k_N$ and let $Z_i = \{z_{ij}\}_{j=1}^{k_i}$ for $i = 1,\ldots,N$ be disjoint sets of variables.
  
  \noindent
  \textbf{1.} Consider the element
  \[ \sum_{\vec{\l}} \wp_{k_1}(\l^1) \cdots \wp_{k_N}(\l^N) \MP_{\vec{U},\vec{V}}^{\textbf{f}}(\vec{\l}) \in \widehat{\bigotimes}_{i=1}^N (\Lambda_{U^i} \otimes \Lambda_{V^i}). \]
  \\
  
  \noindent
  \textbf{2.} Multiply through by the normalizing constant. Reexpress the sums within $\MP$ in terms of Macdonald pairings as in (\ref{eq:mppair})
  \[ \sum_{\vec{\l}} \wp_{k_1}(\l^1) \cdots \wp_{k_N}(\l^N) P_{\l^1}(U^1) \left( \prod_{i=1}^{N-1} \left \langle Q_{\l^i}(V^i,Y^i), P_{\l^{i+1}}(Y^i,U^{i+1}) \right \rangle_{Y^i} \right) Q_{\l^N}(V^N). \]
  Here we note the spaces which the pairings map:
  \[ \langle \cdot, \cdot \rangle_{Y^i}: \left( \L_{V^i} \cotimes \L_{Y^i} \right) \times \left( \L_{Y^i} \cotimes \L_{U^{i+1}} \right) \to \L_{V^i} \cotimes \L_{U^{i+1}}. \]
  By natural inclusions (\ref{eq:inclu}) and consistency (\ref{eq:consis}), the domain of this pairing may be extended.
  \\
  
  \noindent
  \textbf{3.} Bring the summation inside the pairings and the pairings inside the pairings
  \[ \langle \mathbf{E}_1, \langle \mathbf{E}_2, \langle \cdots \langle \mathbf{E}_{N-1}, \mathbf{E}_N \rangle_{Y^{N-1}} \cdots \rangle_{Y^2} \rangle_{Y^1} \]
  where
  \[ \mathbf{E}_i = \sum_{\l^i} \wp_{k_i} P_{\l^1}(Y^{i-1},U^i) Q_{\l^i}(V^i,Y^i) \]
  and $Y^0, Y^N$ are empty sets of variables. It was important to use the fact that the first argument of the $Y^i$ Macdonald pairing is $\Lambda_{Y^i}$-projective which provides the continuity necessary for bringing the summations inside.
  \\
  
  \noindent
  \textbf{4.} We can reexpress the summations in terms of Negut's operator in the residue form (\ref{eq:negut})
  \begingroup \makeatletter \def \f@size{9}\check@mathfonts
  \begin{align*}
  \mathbf{E}_i &= D_{-k_i}^{Y^{i-1},U^i} \Pi((Y^{i-1},U^i),(V^i,Y^i)) \\
  &= \Pi((Y^{i-1},U^i),(V^i,Y^i)) \oint H((Y^{i-1},U^i),qZ_i^{-1};t^{-1}) H(q^{-1}Z_i,(V^i,Y^i);t)^{-1} \, DZ_i
  \end{align*}
  \endgroup
  \\
  
  \noindent
  \textbf{5.} The domain of the residue operator can be appropriately extended and consistency follows from (\ref{eq:inclu}) and (\ref{eq:consis}). Note that the integrand in $\mathbf{E}_i$ remains $\Lambda_{Y^i}$-projective. Therefore, by (\ref{eq:rpcom1}) and (\ref{eq:rpcom2}), we may commute the residue operators with the pairings. After pulling out $Y^i$ independent factors outside the residue operators, we obtain
  \[ A \oint DZ_1 \cdots \oint DZ_N \langle \mathbf{F}_1, \langle \mathbf{F}_2, \cdots \langle \mathbf{F}_{N-1}, \mathbf{F}_N \rangle_{Y^{N-1}} \cdots \rangle_{Y^2} \rangle_{Y^1} \]
  where
  \begin{align} \label{eq:fi}
  \mathbf{F}_i &=  H(Y^{i-1},qZ_i^{-1};t^{-1})\Pi(Y^{i-1},V^i)H(q^{-1}Z_i,Y^i;t)^{-1} \Pi((Y^{i-1},U^i),Y^i) \\ \label{eq:C}
  A &=  \prod_{i=1}^N H(U_i,qZ_i^{-1};t^{-1}) H(q^{-1}Z_i,V_i;t)^{-1} \Pi(U^i,V^i).
  \end{align}
  Here, (\ref{eq:split}) and (\ref{eq:invH}) were used to split $H$ and $\Pi$.
  \\
  
  \noindent
  \textbf{6.} Apply the pairings for $Y^i$ in decreasing order of $i$. At the $(N-i)$th step, we have
  \begin{align} \label{eq:6pf}
  A_i \oint DZ_1 \cdots \oint DZ_N \langle \mathbf{F}_1, \langle \mathbf{F}_2, \cdots \langle \mathbf{F}_i, \mathscr{F}_i \rangle_{Y^{N-i}} \cdots \langle_{Y^2} \rangle_{Y^1} 
  \end{align}
  where $A_i$ collects the $Y^1,\ldots,Y^i$ independent terms. We show by induction that
  \begin{align} \label{eq:sF}
  \mathscr{F}_i = H(Y^i, qZ_{[i+1,N]}^{-1};t^{-1}) \Pi(Y^i,V^{[i+1,N]})
  \end{align}
  where we used shorthand notation $Z_{[i,j]} = (Z_i,Z_{i+1},\ldots,Z_j)$ and similarly for $V$. If we suppose (\ref{eq:sF}) is true, then within the $Y^i$ bracket in (\ref{eq:6pf}), $\mathscr{F}_i$ interacts with the third and fourth terms given in (\ref{eq:fi}). By (\ref{eq:invH}) and (\ref{eq:HPi}), this interaction produces
  \begin{equation} \label{eq:outpair}
  \begin{multlined}
  C(Z_i,Z_{[i+1,N]}) H(q^{-1} Z_i,V^{[i+1,N]};t)^{-1} \\
  \times H((U^i,Y^{i-1}),qZ_{[i+1,N]}^{-1};t^{-1}) \Pi((U^i,Y^{i-1}),V^{[i+1,N]})
  \end{multlined}
  \end{equation}
  where $C(Z,W)$ is defined by (\ref{eq:czw}). As a formal expression, we expand any terms of the form $(1 - v)^{-1}$ as the geometric series. The $\{Y^j\}$ independent term of (\ref{eq:outpair}) is
  \begin{align} \label{eq:Yjind}
  C(Z_i,Z_{[i+1,N]}) H(q^{-1} Z_i,V^{[i+1,N]};t)^{-1} H(U^i, qZ_{[i+1,N]}^{-1};t^{-1}) \Pi(U^i,V^{[i+1,N]}).
  \end{align}
  After picking up the first two terms in (\ref{eq:fi}), the remaining term to interact with the $Y^{i-1}$ pairing is
  \[ H(Y^{i-1},qZ_{[i,N]}^{-1};t^{-1}) \Pi(Y^{i-1}, V^{[i,N]}) \]
  which completes the induction as the starting term and ending terms are consistent, the initial term for $i = N-1$ is exactly $\mathbf{F}_N$, and the final term is unity because $Y^0$ is empty. After collecting the $Y^j$-independent terms (\ref{eq:Yjind}) from each $i = N-1,N-2,\ldots,1$ and applying (\ref{eq:split}), we complete the proof of Theorem \ref{thm:multi}.
\end{proof}

We now illustrate the main idea of the proof of Corollary \ref{cor:multi} via a particular example. For further details, we note that the proof is essentially identical to a corresponding extension in \cite{BCGS} (Theorem 3.10 to Corollary 3.11).

\begin{proof}[Proof Idea of Corollary \ref{cor:multi}]
We consider the example of $N = 1$, $M = 2$, $m = 2$. Let $k_1,k_2 > 0$ be integers. Consider auxiliary variables $\vec{S} = (S^1,S^2),\vec{T} = (T^1,T^2)$, and the formal expectation
\begin{align} \label{eq:prephi}
\E_{\MP_{\vec{S},\vec{T}}^{\mathbf{f}}}[\wp_{k_1}(\l^1), \wp_{k_2}(\l^2)] = \sZ_0^{-1} \sum_{\l^1,\l^2 \in \Y} \wp_{k_1}(\l^1) \wp_{k_2}(\l^2) P_{\l^1}(S^1) \left( \sum_{\mu \in \Y} Q_{\l^1/\mu}(T^1) P_{\l^2/\mu}(S^2)\right) Q_{\l^1}(T^2)
\end{align}
where $\sZ_0$ is the normalizing factor $\MP_{\vec{S},\vec{T}}^{\mathbf{f}}$. Consider the map $\phi:\L_{T^1} \otimes \L_{S^2}$ which sends $f(T^1)g(S^2) \mapsto f(0)g(0)$ to the constant term for any $f,g\in \L_X$. By applying (the continuous extension of) $\phi$ to (\ref{eq:prephi}) and rewriting $S^1 = U$ and $T^2 = V$, we get
\begin{align} \label{eq:postphi}
\sZ^{-1} \sum_{\l \in \Y} \wp_{k_1}(\l) \wp_{k_2}(\l) P_{\l}(U) Q_{\l}(V) = \E_{\MP_{U,V}^{\mathbf{f}}}[\wp_{k_1}(\l^1) \wp_{k_2}(\l^2)].
\end{align}
where $\sZ$ is the normalizing factor for $\MP_{U,V}^{\mathbf{f}}$. On the other hand, by Theorem \ref{thm:multi}, we have a formal residue expression for (\ref{eq:prephi}). By applying $\phi$ to this expression, we obtain (\ref{eq:multi}) for this choice of $N,M,m$.

In the general case, we consider some formal Macdonald process in a greater number of variables, apply Theorem \ref{thm:multi}, then apply variable contractions $\phi$ to obtain the Corollary.
\end{proof}
  
\subsection{RPP Observables} \label{ssec:rppobs}
In this section, we derive a formula for (\ref{eq:joint_moments}), stated below in Theorem \ref{thm:obs}. It is convenient to do this in two steps: first apply Corollary \ref{cor:multi} to a formal version of the RPP measures, then specialize the formal RPPs to $\P^{B,r,\vec{s}}_{\alpha,\ft}$.

\subsubsection{Formal Random Plane Partition}
Fix a measure $\P^{B,r,\vec{s}}_{\alpha,\ft}$. For each $e \in I_E$, let
\begin{align} \label{eq:Fml}
F_{\mu,\lambda}(X;b,q,t) &= \left\{ \begin{array}{ll}
P_{\lambda/\mu}(X;q,t) & \mbox{if $b = 1$,} \\
Q_{\mu/\lambda}(X;q,t) & \mbox{if $b = 0$.}
\end{array} \right.
\end{align}

\begin{definition}
If the domain $I$ of the back wall $B:I \to \R$ has finite length, define the \textit{formal RPP with back wall} $B$ to be the formal probability measure $\P^{B,\mathbf{f}}$ supported on $\cP_B$ and valued in $\widehat{\bigotimes}_{e \in I_E} \Lambda_{X_e}$ so that
\begin{align} \label{eq:informal}
\P^{B,\mathbf{f}}(\pi) = \frac{1}{\sZ} \prod_{e \in I_E} F_{\pi^{e-\frac{1}{2}},\pi^{e+\frac{1}{2}}}(X_e;B'(e),q,t).
\end{align}
\end{definition}

The partition function can be computed:
\begin{align} \label{eq:frpppartition}
\sZ = \prod_{\substack{e_1,e_2 \in I_E, e_1 < e_2 \\ B'(e_1) > B'(e_2)}} \Pi(X_{e_1}, X_{e_2}).
\end{align}
We comment on how to obtain (\ref{eq:frpppartition}) after proving Proposition \ref{prop:frpp}.

\begin{proposition} \label{prop:frpp}
Suppose the domain $I$ of $B:I \to \R$ has finite length. Let $x_1 \leq \cdots \leq x_m$ be points in $I_E$ and $k_1,\ldots,k_m > 0$ be integers. Then
\begin{align*}
\begin{multlined}
\E_{\P^{B,\mathbf{f}}} \left[ \prod_{i=1}^m \wp_{k_i}(\pi^{x_i})\right] = \oint \cdots \oint \prod_{a < b} C(Z_a,Z_b) \\
\times \prod_{a=1}^m \prod_{\substack{e \in I_E, e < x_a \\ B'(e) = 1}} H(X^e, q Z_a^{-1};t^{-1}) \prod_{\substack{e \in I_E, e > x_a \\ B'(e) = 0}} H( q^{-1} Z_a,X^e;t)^{-1} \, DZ_a
\end{multlined}
\end{align*}
where $|Z_a| = k_a$.
\end{proposition}
\begin{proof}
The proof is specializing Corollary \ref{cor:multi} to the formal RPPs. We find a good way of relabeling the formal Macdonald process indices to make this specialization transparent.

Let $N = |I_E| - 1$, $e' = \min I_E$ and $v' = \min I_V$ (then $v' = e' - \frac{1}{2}$). We may reexpress (\ref{eq:informal}) as
\begin{align} \label{eq:fRPPproof}
\P_{B,\mathbf{f}}(\pi) = \frac{1}{\sZ} F_{\pi^{v'},\pi^{v'+1}}(X_{e'};B'(e'),q,t) F_{\pi^{v'+1},\pi^{v'+2}}(X_{e'+1};B'(e'+1);q,t) \cdots F_{\pi^{v'+N},\pi^{v'+N+1}}(X_{e'+N};B'(e'+N);q,t).
\end{align}
By (\ref{eq:PQmon}) and (\ref{eq:Fml}), we have that $\pi^{v'+1} = (0)$ whenever $B'(e') = 0$. Likewise $\pi^{v'+N} = (0)$ whenever $B'(e'+N) = 1$. We may therefore assume that $B'(e') = 1$ and $B'(e' + N) = 0$.

Let $\widetilde{\vec{U}} = (\widetilde{U}^1,\ldots,\widetilde{U}^N)$ and $\widetilde{\vec{V}} = (\widetilde{V}^1,\ldots,\widetilde{V}^N)$ where $N = |I_E| - 1$. Consider the formal Macdonald process $\MP_{\widetilde{\vec{U}},\widetilde{\vec{V}}}^{\mathbf{f}}(\widetilde{\lambda}^1,\ldots,\widetilde{\lambda}^N)$. It will be convenient to consider relabelings $\vec{U} = (U^{e})_{e\in I_E},\vec{V} = (V^{e})_{e\in I_E}$, $(\lambda^{e})_{e\in I_E}$ so that
\begin{align}
& \widetilde{U}^1 = U^{e'},~~~ \widetilde{U}^2 = U^{e' + 1}, ~~~\ldots,~~~ \widetilde{U}^N = U^{e' + N - 1} \\
& \widetilde{V}^1 = V^{e'+1},~~~ \widetilde{V}^2 = V^{e' + 2},~~~\ldots,~~~ \widetilde{V}^N = V^{e' + N} \\
& \widetilde{\l}^1 = \l^{v'+1},~~~ \widetilde{\l}^2 = \l^{v'+1},~~~\ldots,~~~ \widetilde{\l}^N = \l^{v'+N}.
\end{align}
Thus
\begin{align} \label{eq:MPproof}
\begin{multlined}
\MP_{\widetilde{U},\widetilde{V}}^{\mathbf{f}}(\widetilde{\l^1},\ldots,\widetilde{\l^N}) = \frac{1}{\widetilde{\sZ}} P_{\l^{v'+1}}(U^{e'}) \cdot \sum_{\mu \in \Y} Q_{\l^{v'+1}/\mu}(V^{e'+1}) P_{\l^{v'+2}/\mu}(U^{e'+1}) \\
\cdots \sum_{\mu \in \Y} Q_{\l^{v'+N-1}/\mu}(V^{e'+N-1}) P_{\l^{v'+N}/\mu}(U^{e'+N-1}) \cdot Q_{\l^{v'+N}}(V^{e'+N})
\end{multlined}
\end{align}
where $\widetilde{\sZ}$ is the normalization factor.

For $e \in I_E$, define $X^{e}$ to be $U^{e}$ if $B'(e) = 1$ and $V^{e}$ if $B'(e) = 0$, and $\vec{X} = (X^e)_{e\in I_E}$. Similarly, define $Y^{e}$ to be $V^{e}$ if $B'(e) = 1$ and $U^{e}$ if $B'(e) = 0$, and $\vec{Y} = (Y^e)_{e\in I_E}$.

Let $\rho_0^X$ denote the $0$-specialization on $\Lambda_X$, or equivalently constant term map for $\Lambda_X$. Define
\[ \rho_0^B := \bigotimes_{e \in I_E} \rho_0^{Y^e}: \bigotimes_{e \in I_E} \Lambda_{Y^e} \to \C. \]
By taking tensor products with the identity on $\Lambda_{X^e}$ for $e \in I_E$, and extending by continuity, we have a map
\[ \rho_0^B: \widehat{\bigotimes}_{e \in I_E} (\Lambda_{X^e} \widehat{\otimes} \Lambda_{Y^e}) \to \widehat{\bigotimes}_{e \in I_E} \Lambda_{X^e}. \]
We may further extend this map by extending the scalars from $\C$ to $\cL(Z^1,\ldots,Z^n)$.

Applying $\rho_0^B$ to (\ref{eq:MPproof}) gives (\ref{eq:fRPPproof}) for $\l^v = \pi^v$, $v \in I_{v}$. This continues to hold true for expectations, so we have
\[ \rho_0^B\left( \E_{\MP_{\vec{U},\vec{V}}^{\mathbf{f}}} [ \wp_{k_1}(\pi^{x_1}) \cdots \wp_{k_m}(\pi^{x_m})] \right) = \E_{\P^{B,\mathbf{f}}}[ \wp_{k_1}(\l^{x_1}) \cdots \wp_{k_m}(\l^{x_m})].\]
By Corollary \ref{cor:multi} the left hand side is exactly
\begin{align*}
\begin{multlined}
\rho_0^B \left( \oint \cdots \oint \prod_{a < b} C(Z_a,Z_b) \prod_{a=1}^m \left( \prod_{e < x_a, e\in I_E} H(U^e,q Z_a^{-1}; t^{-1}) \prod_{e > x_a, e \in I_E} H(q^{-1}Z_a,V^e; t)^{-1} \,DZ_a  \right) \right) \\
= \oint \cdots \oint \prod_{a < b} C(Z_a,Z_b) \prod_{a=1}^m \prod_{\substack{e \in I_E, e < x_a \\ B'(e) = 1}} H(X^e,  qZ_a^{-1};t^{-1}) \prod_{\substack{e \in I_E, e > x_a \\ B'(e) = 0}} H( q^{-1}Z_a,X^e;t)^{-1} \, DZ_a.
\end{multlined}
\end{align*}
where we have used the fact that the residue operator commutes with continuous maps. This proves the proposition.
\end{proof}

\begin{remark}
The formula (\ref{eq:frpppartition}) is then a consequence of applying the specializations in the proof of Proposition \ref{prop:frpp} to the partition function for the formal Macdonald process (\ref{eq:Znorm}).
\end{remark}

\subsubsection{Specialization to RPP}

Consider the distribution in (\ref{eq:macrv}) where we have a sequence of weights $(r_v)_{v\in I_V}$. Fix an arbitrary $\xi > 0$, define
\begin{align} \label{eq:ae}
a_e := \xi \prod_{\substack{v \in I_V \\ v < e}} r_v.
\end{align}
By (\ref{eq:psi}) and (\ref{eq:phi}), the distribution defined by
\begin{align} \label{eq:rppae}
\P(\pi) = \frac{1}{\sZ((a_e)_{I_E})} \prod_{e \in I_E} F_{\lambda^{e-\frac{1}{2}},\lambda^{e+\frac{1}{2}}}(a_e^{1 - 2B'(e)};B'(e),q,t)
\end{align}
coincides with (\ref{eq:macrv}) if and only if
\[ \sZ((a_e)_{I_E}) = \prod_{\substack{e_1,e_2 \in I_E, e_1 < e_2 \\ B'(e_1) > B'(e_2)}} \Pi(a_{e_1}^{-1}, a_{e_2}) < \infty. \]

This implies the following lemma.
\begin{lemma} \label{lem:ae<1}
If the weights in (\ref{eq:macrv}) are summable, then for any $e_1,e_2 \in I_E$ such that $e_1 < e_2$ and $B'(e_1) > B'(e_2)$, we have
\[ a_{e_1}^{-1} a_{e_2} < 1 \]
where $(a_e)_{e \in I_E}$ is defined in (\ref{eq:ae}). If $I$ is finite, the inequality is also sufficient to determine the weights are summable.
\end{lemma}

We note that for $I$ of finite length, each diagonal partition $\pi^v$ of $\pi \in \cP_B$ has bounded length $\ell(\pi^v)$ depending only $B$ and $v$. Thus for finite $I_E$, the finiteness of $\sZ((a_e)_{I_E})$ implies the existence of the multicut expectations of (\ref{eq:rppae}).

For the measure $\P^{B,r,\vec{s}}_{\alpha,\ft}$, a suitable choice for $a_e$ is given by
\begin{align} \label{eq:perspec}
a_e = r^v s_0 \cdots s_{\lfloor e \rfloor} = r^v s_0 \cdots s_{\lfloor e \rfloor}.
\end{align}

The main formula for the observables $\wp$ can be obtained by specializing the formal RPPs, taking $X_e \mapsto a_e^{1 - 2B'(e)}$. For $x \in I_V$, define the function
\begin{align} \label{eq:GB}
\begin{split}
G^B_{<x}(z;\e,\ft) &= \prod_{\substack{e<x,e \in I_E \\ B'(e) = 1}} \frac{1 - (t a_e z)^{-1}}{1 - (a_e z)^{-1}} = \prod_{i=0}^{p-1} \prod_{\substack{e < x, e \in I_E \\ e \in p\Z + i + \frac{1}{2}}} \left(\frac{1 - (tr^e(s_0 \cdots s_i) z)^{-1}}{1 - (r^e(s_0\cdots s_i) z)^{-1}} \right)^{B'(e)} \\
G^B_{>x}(z;\e,\ft) &= \prod_{\substack{e>x,e \in I_E \\ B'(e) = 0}} \frac{1 - a_e z}{1 - t a_e z} = \prod_{i=0}^{p-1} \prod_{\substack{e > x, e \in I_E \\ e \in p\Z + i + \frac{1}{2}}} \left( \frac{1 - r^e (s_0\cdots s_i) z}{1 - t r^e (s_0\cdots s_i) z} \right)^{1-B'(e)}
\end{split}
\end{align}

Given some function $g(z)$ in one-variable and $Z = (z_1,\ldots,z_k)$ an ordered collection of variables, we write
\[ g(Z) := \prod_{i=1}^k g(z_i). \]

\begin{theorem} \label{thm:obs}
Consider the measure $\P^{B,r,\vec{s}}_{\alpha,\ft}$ and let $(\pi^x)_{x\in I}$ denote the (random) diagonal partitions. Let $x_1 \le \cdots \le x_m$ be in $I_V$ and $k_1,\ldots,k_m > 0$  be integers. Suppose there exist positively oriented contours $\{\cC_{i,j}\}_{\substack{1 \leq j \leq k_i \\ 1 \leq i \leq m}}$ such that
\begin{itemize}
    \item the contour $\cC_{i,j}$ is contained in the domain bounded by $t \cC_{i',j'}$ whenever $(i,j) < (i',j')$ in lexicographical ordering;
    \item each domain bounded by $\cC_{i,j}$ contains $0$ and the poles of $G_{<x}^B(z;\e,\ft)$ but not the poles of $G_{>x}^B(z;\e,\ft)$.
\end{itemize}
Then
\begin{align*}
\E[\wp_{k_1}(\pi^{x_1};q,t) \cdots \wp_{k_m}(\pi^{x_m};q,t)] = \oint \cdots \oint \prod_{a < b} C(Z_a,Z_b) \prod_{a=1}^m G^B_{<x_a}(Z_a;\e,\ft) G^B_{>x_a}(Z_a;\e,\ft) \,DZ_a
\end{align*}
where $|Z_i| = k_i$, $Z_i = (z_{i,1},\ldots,z_{i,k_i})$, the contour of $z_{i,j}$ is given by $\cC_{i,j}$.
\end{theorem}

\begin{remark} \label{rem:ae<1}
By Lemma \ref{lem:ae<1}, given any poles $\fp_1,\fp_2$ of $G_{<x}^B(z;\e,\ft), G_{>x}^B(z;\e,\ft)$ respectively, we have $\fp_1 < \fp_2$. The existence of the contours $\cC_{i,j}$ is then dependent on whether there is enough distance between these two sets of poles. 
\end{remark}

\begin{proof}[Proof of Theorem \ref{thm:obs}]
If the domain $I$ has finite length, then the theorem follows from Proposition \ref{prop:frpp}. To see how to obtain the contour conditions, we recall the formal definition of (\ref{eq:Hnorm}) and (\ref{eq:DZ}). The formal expansion of (\ref{eq:Hnorm}) that we desire amounts to taking contours which contain the poles of $G_{<x}^B(z;\e,\ft)$ but not the poles of $G_{>x}^B(z;\e,\ft)$. The expressions $(1 - \frac{z_{i,j}}{tz_{i',j'}})^{-1}$ which appear in $DZ$ and $C(Z,W)$ are expanded as $\sum_{n\ge 0} \frac{z_{i,j}^n}{(tz_{i',j'})^n}$ which requires the condition that $\cC_{i,j}$ is contained in the domain bounded by $t \cC_{i',j'}$ whenever $(i,j) < (i',j')$ in lexicographical order. Note the change of variables rewriting $q^{-1}Z_a$ as $Z_a$.

If $I$ has infinite length, define $\P_N \defeq \P^{B_N,r,\vec{s}}_{\alpha,\ft}$ where the back wall $B_N:I^N \to \R$ is defined to be the restriction of $B$ to $I^N := I \cap [-N,N]$, for $N \in \mathbb{N}$. The summability of the weights of $\P \defeq \P^{B,r,\vec{s}}_{\alpha,\ft}$ implies the summability of the weights of $\P_N$.

Let $(a_e)_{e\in I_E}$ be the sequence of specializations for $\P$ as in (\ref{eq:perspec}). Then $(a_e)_{e \in I_{N,e}}$ where $I^N_E = I_E \cap [N,N]$ is the sequence of specializations for $\P_N$. Let $\sZ$ and $\sZ_N$ denote the partition functions for $\P$ and $\P_N$ respectively. Choose $x_1 \leq \cdots \leq x_m \in I_E$ and integers $k_1,\ldots,k_m > 0$. Consider $N$ large enough so that $x_1,\ldots,x_m \in [-N,N]$. We have
\begin{align*}
\begin{multlined}
\sZ_N \E_{\P_N}\left[ \prod_{i=1}^m \wp_{k_i}(\pi^{x_i}) \right]
= \sum_{(\pi^v) \in \Y^{I^N_V}} \left[ \prod_{i=1}^m \wp_{k_i}(\pi^{x_i}) \right] \prod_{e \in I^N_E} F_{\pi^{e-\frac{1}{2}},\pi^{e+\frac{1}{2}}}(a_e^{1 - 2B'(e)};B'(e),q,t) \\
\to \sum_{(\pi^v) \in \Y^{I_V}} \left[ \prod_{i=1}^m \wp_{k_i}(\pi^{x_i}) \right] \prod_{e \in I_E} F_{\pi^{e-\frac{1}{2}},\pi^{e+\frac{1}{2}}}(a_e^{1 - 2B'(e)};B'(e),q,t) = \sZ \E_\P \left[ \prod_{i=1}^m \wp_{k_i}(\pi^{x_i}) \right] 
\end{multlined}
\end{align*}
as $N\to\infty$ since the sequence is monotonically increasing. Since $\sZ_N \to \sZ$, we have as $N\to\infty$
\[ \E_{\P_N}\left[ \prod_{i=1}^m \wp_{k_i}(\pi^{x_i}) \right] \to \E_\P \left[ \prod_{i=1}^m \wp_{k_i}(\pi^{x_i}) \right]. \]
On the other hand, for any $x \in I_V$, we have $G_{<x}^{B_N}(z;\e,\ft) \to G_{<x}^B(z;\e,\ft)$ as $N\to\infty$ uniformly away from the poles of $G_{<x}^B(z;\e,\ft)$, and likewise for $G_{>x}^B(z;\e,\ft)$. By applying the theorem for the known case of $B_N$ and taking $N\to\infty$, the general theorem follows.
\end{proof}

\section{Limit Conditions and Back Walls} \label{sec:backwall}
In this section, we identify the class of functions which can be realized as limits of back walls. As mentioned in Section \ref{sec:results}, our limit theorems restrict to a dense subset of this class. We motivate this restriction through the concept of singular points and some examples from the literature. Our study of singular points is also used in Section \ref{sec:asymp} for asymptotics.

We recall the Limit Conditions.

\begin{limcon*}
Fix a $p$-periodic, bi-infinite sequence $\vec{s} = (\cdots,s_{-1},s_0,s_1,\cdots) \in \R_{>0}^\infty$ such that $s_0 \cdots s_{p-1} = 1$. Let $\P^{B,r,\vec{s}}_{\alpha,\ft}$ be a family parametrized by a small parameter $\e > 0$ where $B:I^\e \to \R$ and $r \defeq e^{-\e}$ vary with $\e$ so that
\begin{enumerate}
    \item there exist integers
    \[ \inf I^\e = v_0(\e) < \cdots < v_n(\e) = \sup I^\e \]
    such that for each $1 \le \ell \le n$, $B'(x)$ is $p$-periodic on $(v_{\ell-1}, v_\ell) \cap (\Z + \frac{1}{2})$;
    \item there exists an interval $I \subset \R$ and a piecewise linear $\cB:I \to \R$ with non-differentiable points
    \[ \inf I = V_0 < \cdots < V_n = \sup I \]
    such that
    \[ \e v_\ell(\e) \to V_\ell \quad (0 \le \ell \le n), \quad \e B^\e(x/\e) \to \cB(x) \]
    as $\e \to 0$, where the latter convergence is uniform over any compact subset of $I$.
\end{enumerate}
\end{limcon*}

In the setting of global limits under this limit regime, some of the information encoded by $\vec{s}$ is washed away. The dependence on $\vec{s}$ is only through the values and corresponding multiplicities of the sequence $\{s_0 \cdots s_i\}_{i=0}^{p-1}$. This motivates the following definition.

\begin{definition}
We associate to $\vec{s}$ the multiset $\cS = \{s_0 \cdots s_i\}_{i=0}^{p-1}$. Given $\sigma \in \cS$, let
\[ S_\sigma = \{i \in [[0,p-1]]: \sigma = s_0 \cdots s_i \}. \]
\end{definition}

In particular, we remember the multiplicity of each member $\sigma \in \cS$ in $(s_0\cdots s_i)_{i=0}^{p-1}$. The multiplicity of $\sigma$ is given by $|S_\sigma|$. In replacing $\vec{s}$ with $\cS$, we forget about the particular order of the sequence $(s_0 \cdots s_i)_{i=0}^{p-1}$.

For fixed $\cS$, it is not the case that \emph{any} piecewise linear $\cB$ may be realized as the limiting back wall of some $\P^{B,r,\vec{s}}_{\alpha,\ft}$ satisfying the Limit Conditions. This is due to the fact that $\P^{B,r,\vec{s}}_{\alpha,\ft}$ is not a probability measure for arbitrary $B$; the conditions required for the summability of the weights, summarized by Lemma \ref{lem:ae<1}, severely restricts the class of $B$ which give rise to probability measures. We now characterize the set of $\cB$ which can be achieved by the Limit Conditions.

Given a real-valued function $f$ on an interval $I$, let $f(x^\pm) = \lim_{u \to x^\pm} f(u)$.

\begin{definition} \label{def:B(S)}
Let $\tau_1 \ge \cdots \ge \tau_p$ be a labeling of the elements of $\cS$ and set $\tau_0 = \infty, \tau_{p+1} = 0$. Let $\mathfrak{B}(\cS)$ denote the set of continuous piecewise linear functions $\cB:I \to \R$ on some interval domain such that
\begin{enumerate}
    \item \label{def:B(S)_1} the non-differentiable points of $\cB$ are given by
    \[ \inf I = V_0 < \cdots < V_n = \sup I; \]
    \item \label{def:B(S)_2} for each $x \in I \setminus \{V_\ell\}_{\ell=0}^n$, $\cB'(x) \in \{ \frac{i}{p} \}_{i=0}^p$;
    \item \label{def:B(S)_3} if $V \le W$ are non-differentiable points of $\cB$ with $\cB'(V^-) = \frac{i}{p}$, $\cB'(W^+) = \frac{j}{p}$, then
    \begin{align} \label{lc}
    \tau_i^{-1} \tau_{j+1} e^{-(W - V)} \le 1.
    \end{align}
\end{enumerate}
\end{definition}

\begin{remark}
In the definition above, we take the convention that $\cB'(V_0^-) = 0$ and $\cB'(V_n^+) = 1$.
\end{remark}

\begin{theorem} \label{thm:nec_suff}
A function $\cB: I \to \R$ is a limiting back wall of some $\P^{B,r,\vec{s}}_{\alpha,\ft}$ satisfying the Limit Conditions if and only if $\cB \in \fB(\cS)$.
\end{theorem}

Before proving this theorem, we provide an important link between the slopes of $\cB'$ and that of the prelimit $B'(x)$ on $p\Z + \frac{1}{2}$.

\begin{lemma} \label{lem:sigma_order}
Suppose $\P^{B,r,\vec{s}}_{\alpha,\ft}$ satisfies the Limit Conditions. Fix $\ell \in [[1,n]]$ and let $\cB'(V_{\ell-1},V_\ell) = \frac{i}{p}$. Then for sufficiently small $\e > 0$, there exists a set $A \subset[[0,p-1]]$ (potentially varying in $\e$) of size $i$ such that
\[ s_0 \cdots s_a \le s_0 \cdots s_b, \quad \mbox{for all~} a \in A,~ b \in [[0,p-1]] \setminus A \]
and
\[ B' = \1\left[A + p\Z + \frac{1}{2} \right] \]
on $I_E^\e \cap (v_{\ell-1},v_\ell)$.
\end{lemma}
\begin{proof}
By the Limit Conditions, we know that for sufficiently small $\e > 0$ there exists a (potentially varying in $\e$) subset $A \subset [[0,p-1]]$ of size $i$ such that
\[ B' = \1\left[A + p\Z + \frac{1}{2} \right]. \]
on $I_E^\e \cap (v_{\ell-1},v_\ell)$.

Assume for contradiction that for arbitrarily small $\e > 0$, there exist (potentially varying in $\e$) pairs $\sigma = s_0 \cdots s_a$ and $\tau = s_0 \cdots s_b$ for some $a \in A$ and $b \in [[0,p-1]] \setminus A$ such that $\sigma > \tau$. For $\e$ small enough so that $(v_{\ell-1},v_\ell)$ has more than $p$ points, we may choose $e_1 \in (v_{\ell-1},v_\ell) \cap (a + p\Z + \frac{1}{2})$, $e_2 \in (v_{\ell-1},v_\ell) \cap (b + p\Z + \frac{1}{2})$ such that $0 < e_2 - e_1 < p$. By (\ref{eq:perspec}), we have
\[ a_{e_1}^{-1} a_{e_2} = \sigma^{-1} \tau r^{e_2 - e_1} > \sigma^{-1} \tau e^{-\e p} \]
where $(a_e)_{e\in I^\e_E}$ is the specialization sequence for $\P^{B,r,\vec{s}}_{\alpha,\ft}$ as defined in (\ref{eq:ae}). For $\e$ sufficiently small, the latter is $> 1$ which violates Lemma \ref{lem:ae<1}.
\end{proof}

\begin{proof}[Proof of Theorem \ref{thm:nec_suff}]
Suppose $\cB: I \to \R$ is a limiting back wall of $\P^{B,r,\vec{s}}_{\alpha,\ft}$ satisfying the Limit Conditions. It is clear that $\cB$ satisfies properties (\ref{def:B(S)_1}) and (\ref{def:B(S)_2}) of Definition \ref{def:B(S)} as a direct consequence of the $p$-periodicity and convergence in the Limit Conditions. It remains to check property (\ref{def:B(S)_3}) of Definition \ref{def:B(S)}. Let $1 \le \ell \le m < n$ and suppose $\cB'(V_\ell^-) = \frac{i}{p}$ and $\cB'(V_m^+)$; note that checking the cases $\ell=0$ and $m  = n$ is trivial. By Lemma \ref{lem:sigma_order}, for small enough $\e > 0$ there exist subsets $A_1,A_2 \in [[0,p-1]]$ of sizes $i,j$ respectively such that
\[ \left. B' \right|_{I_E^\e \cap (v_{\ell-1},v_\ell)} = \1[A_1], \quad \left. B' \right|_{I_E^\e \cap (v_m, v_{m+1})} = \1[A_2]. \]
Moreover,
\[ s_0 \cdots s_a \le s_0 \cdots s_b, \quad \mbox{for } a \in A_1,~ b \in [[0,p-1]] \setminus A_1 \]
so that the multiset $\{s_0 \cdots s_a\}_{a\in A_1}$ coincides with $\{\tau_1,\ldots,\tau_i\}$. Likewise, $\{s_0 \cdots s_a\}_{a \in A_2}$ coincides with $\{\tau_1,\ldots,\tau_j\}$. In particular, there exist
\[ e_1 \in (A_1 + p\Z + \frac{1}{2}) \cap (v_{\ell-1},v_\ell), \quad e_2 \in ([[0,p-1]] \setminus A_2 + p\Z + \frac{1}{2}) \cap (v_m,v_{m+1}) \]
where $s_0 \cdots s_a = \tau_i$, $s_0 \cdots s_b = \tau_{j+1}$, such that
\[ B'(e_1) = 1, \quad B'(e_2) = 0 \]
for small enough $\e > 0$. By $p$-periodicity, we may add that
\[ 0 \le v_\ell - e_1 < p,\quad \quad 0 \le e_2 - v_m < p. \]
By \eqref{eq:perspec} and Lemma \ref{lem:ae<1}, we have
\[ a_{e_1}^{-1} a_{e_2} = \tau_i^{-1} \tau_{j+1} r^{e_2 - e_1} < 1 \]
where $(a_e)_{e\in I^\e_E}$ is the specialization sequence for $\P^{B,r,\vec{s}}_{\alpha,\ft}$ as defined in (\ref{eq:ae}). Taking $\e \to 0$, we obtain
\[ \tau_i^{-1} \tau_{j+1} e^{-(V_m - V_\ell)} \le 1. \]

Conversely, suppose $\cB \in \fB(\cS)$. For small $\e > 0$, set
\[ v_\ell = \lfloor V_\ell/\e \rfloor + \ell, \quad 0 \le \ell \le n; \]
in the case $V_0 = -\infty$ ($V_n = +\infty$) let $v_0 = -\lfloor 1/\e^2 \rfloor$, $v_n = \lfloor 1/\e^2 \rfloor$. Let $B$ be a back wall of a skew diagram such that $B'(x)$ is $p$-periodic in $x \in (v_{\ell-1},v_\ell)$. For each $1 \le \ell \le n$, if we have $\cB'(V_{\ell-1},V_\ell) = \frac{i}{p}$, choose some subset $A_\ell \subset [[0,p-1]]$ of size $i$ so that
\[ s_0 \cdots s_a \le s_0 \cdots s_b, \quad a \in A_\ell, b\in [[0,p-1]] \setminus A_\ell. \]
By fixing $\e B(\lfloor x/\e \rfloor) = \cB(x)$ at some point $x \in I$, we have the convergence
\[ \e B(\lfloor x/\e \rfloor) \to \cB(x). \]
It remains to check that $\P^{B,r,\vec{s}}_{\alpha,\ft}$ defines a probability measure, at least for $\e$ sufficiently small. By Lemma \ref{lem:ae<1}, it suffices to check that $a_{e_1}^{-1} a_{e_2} < 1$ over all edges $e_1 < e_2$ where $B(e_1) = 1$ and $B(e_2) = 0$. We divide this into two cases.

\emph{Case 1: $e_1,e_2 \in (v_{\ell-1},v_\ell)$.} By construction of $B$, we have $B' = \1[A_\ell]$ on $(v_{\ell-1},v_\ell)$. Thus $e_1 \in a + p\Z +\frac{1}{2}$, $e_2 \in b + p\Z + \frac{1}{2}$ for some $a \in A_\ell, b\in [[0,p-1]] \setminus A_\ell$. Thus
\[ a_{e_1}^{-1} a_{e_2} = (s_0 \cdots s_a)^{-1} (s_0 \cdots s_b) r^{e_2 - e_1} < 1. \]
Note that this only relies on $p$-periodicity and did not require property (\ref{def:B(S)_3}) in Definition \ref{def:B(S)}.

\emph{Case 2: $e_1 \in (v_{\ell-1},v_\ell)$, $e_2 \in (v_m,v_{m+1})$ where $0 < \ell \le m < n$.} Again, by construction of $B$ we have $B = \1[A_\ell]$ on $(v_{\ell-1},v_\ell)$ and $B = \1[A_{m+1}]$ on $(v_m,v_{m+1})$. If $\cB'(V_\ell^-) = \frac{i}{p}$ and $\cB'(V_m^+) = \frac{j}{p}$, then we may argue as before to establish that $A_\ell$ coincides with $\{\tau_1,\ldots,\tau_i\}$ and $A_{m+1}$ coincides with $\{\tau_1,\ldots,\tau_j\}$. Then
\[ a_{e_1}^{-1} a_{e_2} = \sigma^{-1} \tau r^{e_2 - e_1} \]
where $\sigma \ge \tau_i$ and $\tau \le \tau_{j+1}$. Since $e_1 < v_\ell \le v_m < e_2$, we have
\[ a_{e_1}^{-1} a_{e_2} \le \tau_i^{-1} \tau_{j+1} r^{v_m - v_\ell + 1} = \tau_i^{-1} \tau_{j+1} e^{-\e(m-\ell+1)} e^{-\e ( \lfloor V_m/\e \rfloor - \lfloor V_\ell/\e \rfloor)} < \tau_i^{-1} \tau_{j+1} e^{-(V_m - V_\ell)} \le 1 \]
where the latter inequality follows from property (\ref{def:B(S)_3}) for $\fB(\cS)$.
\end{proof}

\subsection{Well-Behaved Back Walls}

In this subsection, we introduce a subset $\fB^\Delta(\cS)$ of $\fB(\cS)$ which corresponds to well-behaved back walls. This good behavior is characterized by the presence of only \emph{finitely many} singular points which we describe further below. We avoid a fully rigorous treatment to maintain the focus of the article on limit shape and fluctuation results. However, we provide key ideas and examples which may be further elaborated for rigorous statements.

Before providing the definition of the subset $\fB^\Delta(\cS)$, we begin by introducing and motivating the notion of a singular point. Fix $\cB \in \fB(\cS)$, and define
\begin{align} \label{eq:rho<>}
\begin{split}
\rho_<(\x) &= \max \{ \tau_i^{-1} e^\xi: \xi \le \x, \cB'(\xi^-) = \frac{i}{p}\}, \\
\rho_>(\x) &= \min \{ \tau_{j+1}^{-1} e^\xi: \xi \ge \x, \cB'(\xi^+) = \frac{j}{p} \}.
\end{split}
\end{align}

\begin{lemma}
If $\cB \in \fB(\cS)$ and $\x \in I$, then
\[ \rho_<(\x) \le \rho_>(\x). \]
\end{lemma}
\begin{proof}
Suppose $\xi_1 \le \xi_2$ and $\cB'(\xi_1^-) = \frac{i}{p}, \cB'(\xi_2^+) = \frac{j}{p}$. It is enough to show that
\begin{align} \label{eq:xi1xi2}
\tau_i^{-1} e^{\xi_1} \le \tau_{j+1}^{-1} e^{\xi_2}.
\end{align}
If $\xi_1,\xi_2 \in (V_{\ell-1},V_\ell)$ for some $\ell \in [[1,n]]$, then $i = j$ so that (\ref{eq:xi1xi2}) holds. Otherwise, there exists a non-differentiable points points $V,W$ such that
\[ \xi_1 \le V \le W \le \xi_2. \]
Assume that $V$ is the minimal such point and $W$ is the maximal such point. Then
\[ \cB'(V^-) = \cB'(\xi_1^-) = \frac{i}{p}, \quad \cB'(W^+) = \cB'(\xi_2^+) = \frac{j}{p}. \]
Since
\[ \tau_i^{-1} \tau_{j+1} e^{-(W-V)} \le 1 \]
it follows that
\[ \tau_i^{-1} e^{\xi_1} \le \tau_i^{-1} e^V \le \tau_{j+1}^{-1} e^W \le \tau_{j+1} e^{\xi_2}. \]
\end{proof}

We define the singular points to be those points which achieve equality.

\begin{definition} \label{def:spt}
Given $\cB \in \fB(\cS)$, we say that $\x \in I$ is a \emph{singular point of $\cB$} if $\rho_<(\x) = \rho_>(\x)$.
\end{definition}

\begin{definition} \label{def:BD(S)}
Denote by $\fB^\Delta(\cS)$ the subset of $\fB(\cS)$ consisting of $\cB$ with finitely many singular points.
\end{definition}

The concept of singular points is significant due to the following connection with the limit shape. Recall the content of Theorem \ref{thm:macLLN}: if $\P^{B,r,\vec{s}}_{\alpha,\ft}$ satisfies the Limit Condition with $\cB \in \fB^\Delta(\cS)$, then the random rescaled height function $\e h(x/\e,y/\e)$ converges to a deterministic limit $\cH(x,y)$. Then the local proportions $p_{\lloz},p_{\rloz},p_{\hloz}$ converge at each point $(x,y)$. Recall the liquid region is the set of $(x,y)$ where all the proportions are nonzero. We may view the closure of the liquid region as the set of points $(x,y)$ where the local picture is random.

In Section \ref{sec:cpx}, we characterize the liquid region as the set of $(x,y)$ for which some equation $\cG_x^{\cB}(\zeta) = e^{-y}$ determined by $\cB$ has a pair of nonreal complex roots, see (\ref{eq:compeq}). The map $\cB \mapsto \cG_x^{\cB}$ is continuous with respect to the topology on $\fB(\cS)$ introduced above and convergence in compactum on $\HH$ in the image. Thus one can \emph{formally} extend the definition of the liquid region associated to some $\cB \in \fB(\cS) \setminus \fB^\Delta(\cS)$ to be the set of $(x,y)$ such that $\cG_x^{\cB}(\zeta) = e^{-y}$ has a pair of nonreal roots.

Under this alternative definition of the liquid region, one can determine that the singular points of $\cB$ are exactly the points $x$ such that $(x,y)$ is in the liquid region for arbitrarily large $y$. In other words, the singular points of $\cB$ correspond to the horizontal coordinates along which the liquid region is vertically unbounded.

For certain examples of $\cS$ and $\cB \in \fB(\cS) \setminus \fB^\Delta(\cS)$, one can prove the limit shape phenomenon. In these cases, the alternative definition of the liquid region coincides with the original definition of the liquid region in terms of the local proportions of lozenges. Although it is not present in the literature, we believe that one may use the method of correlation kernels to verify the limit shape phenomenon for arbitrary $\cS$ and $\cB \in \fB(\cS) \setminus \fB^\Delta(\cS)$ in the \emph{non-interacting} ($q = t$) case.

We now provide several places in the literature which illustrate the connection between singular points and unboundedness of the limit shape, then give some references to later sections which give suggestions for generalizing this connection to arbitrary $\cB$.

\begin{example}
$\quad$
\begin{enumerate}
    \item \emph{$p$ arbitrary, $\cS = \{1,\ldots,1\}$.} The set $\fB(\cS)$ consists of $\cB$ such that $\cB'(V_{\ell-1},V_\ell) \in [0,1] \cap \frac{1}{|\cS|}\Z$ since \eqref{lc} trivially holds. The singular points in $I \setminus \{V_\ell\}_{\ell=0}^n$ are precisely $x \in (V_{\ell-1},V_\ell)$ where $\cB'(V_{\ell-1},V_\ell) \notin\{0,1\}$. It was demonstrated in \cite{BMRT}, that the horizontal coordinate $x \in (V_{\ell-1},V_\ell)$ of the limit shape is vertically unbounded if and only if $\cB'(V_{\ell-1},V_\ell) \notin \{0,1\}$, see also \cite[Section 1.2]{M}. The set $\fB^\Delta(\cS)$ consists of $\cB$ such that $\cB'(V_{\ell-1},V_\ell) = 1$ or $0$ for every $\ell = 1,\ldots,n$. In this case, the singular points are precisely those $V_\ell$ where
    \[ 0 = \cB'(V_\ell^+) < \cB'(V_\ell^-) = 1,\]
    and these are exactly the horizontal coordinates where the limit shape is vertically unbounded.
    
    \item \emph{$p = 2$, $\cS = \{1,\alpha\}, \alpha > 1$}. Consider the case where $n = 3$ and we take
    \begin{gather*}
    V_0 = -b, \quad V_1 = -a, \quad V_2 = a, \quad V_3 = b, \quad \quad a < b, \\
    \cB'(-b,-a) = 1, \quad \cB'(-a,a) = 1/2, \quad \cB'(a,b) = 0.
    \end{gather*}
    There exists a threshold value $a_0$ such that if $a < a_0$ then $\cB \notin \fB(\cS)$ (thus this does not correspond to a limit of a plane partition), and if $a \ge a_0$ then $\cB \in \fB(\cS)$. If $a > a_0$, then $\cB \in \fB^\Delta(\cS)$ and the singular points occur at $-a,a$. If $a = a_0$, then $\cB \in \fB(\cS) \setminus \fB^\Delta(\cS)$ and $[-a,a]$ is the set of singular points. In both of these cases, the singular points correspond to the set of horizontal coordinates where the limit shape is vertically unbounded, see \cite[Sections 1.1.1 and 4]{M2}.
    
    \item In general, we show in Section \ref{ssec:frozen} that the singular points for $\cB \in \fB^\Delta(\cS)$ are exactly the horizontal coordinates where the limit shape is vertically unbounded. The method for computing the frozen boundary in Section \ref{ssec:frozen} can also be used to see that $\cB \in \fB(\cS) \setminus \fB^\Delta(\cS)$ give rise to limit shapes (as defined above) that are vertically unbounded over an entire nonempty open interval, and these unbounded parts correspond to components of singular points. 
\end{enumerate}
\end{example}

Although \Cref{def:BD(S)} has the advantage of simplicity, it is not as useful for application. We have the following equivalent definition and characterization of singular points for $\cB \in \fB^\Delta(\cS)$.

\begin{definition}
Let $\sigma_1 > \cdots > \sigma_d$ be the \emph{distinct} elements of $\cS$ in decreasing order, and set $\sigma_0 = \infty$, $\sigma_{d+1} = 0$. Let $\varsigma_a = \sum_{j=1}^a \frac{|S_{\sigma_j}|}{p}$ for $0 \le a \le d$.
\end{definition}

\begin{proposition} \label{prop:BD(S)}
We have that $\cB \in \fB^\Delta(\cS)$ if and only if $\cB \in \fB(\cS)$ such that
\begin{enumerate}
    \item for each $x \in I \setminus \{V_\ell\}_{\ell=0}^n$, we have $\cB'(x) \in \{ \varsigma_i \}_{i=0}^d$;
    \item if $V < W$ are non-differentiable points of $\cB$, then $\rho_<(V) < \rho_>(W)$.
\end{enumerate}
If $\cB \in \fB^\Delta(\cS)$, then the singular points of $\cB$ are exactly the non-differentiable points $V$ of $\cB$ such that $\cB'(V^+) < \cB'(V^-)$. In this case, if $\cB'(V^-) = \varsigma_i$ then $\cB'(V^+) = \varsigma_{i-1}$ and
\[ \rho_<(V) = \rho_>(V) = \sigma_i^{-1} e^V. \]
\end{proposition}

Before providing the proof, we highlight a few features. The first condition in Proposition \ref{prop:BD(S)} restricts the possible values of the slopes of $\cB$ whereas the second condition is a refinement of property (\ref{def:B(S)_3}) in Definition \ref{def:B(S)}. This is transparent when we rewrite the second condition as the following equivalent statement:

\emph{If $V < W$ are non-differentiable points of $\cB$ with $\cB'(V^-) = \varsigma_i,$ $\cB'(W^+) = \varsigma_j$ for some $0 \le i,j \le d$, then}
\begin{align} \label{lc4}
\sigma_i^{-1} \sigma_{j+1} e^{-(W-V)} < 1.
\end{align}

Although this refinement requires the inequality to be strict for pairs of non-differentiable points $V < W$, it does not require the same for $V = W$; namely if $V$ is a non-differentiable point of $\cB$ such that $\cB'(V^-) = \varsigma_i$ and $\cB'(V^+) = \varsigma_j$ then we still have the \emph{weak inequality}
\[ \sigma_i^{-1} \sigma_{j+1} e^{-(W - V)} \le 1. \]

This way of viewing $\fB^\Delta(\cS)$ also has the advantage of realizing the subset as dense in $\fB(\cS)$.

\begin{corollary}
If we endow $\fB(\cS)$ with the topology induced from the disjoint union $\bigsqcup_{n=1}^\infty(\R \cup \{\infty\})^{2n+1}$ by identifying $\cB \in \fB(\cS)$ with the point
\[ (V_0,\ldots,V_n, \cB'(V_0,V_1),\ldots,\cB'(V_{n-1},V_n)). \]
Then $\fB^\Delta(\cS)$ is a dense subset of $\fB(\cS)$ by Proposition \ref{prop:BD(S)}. Note that the parametrization identifies $\cB$ which differ up to translation.
\end{corollary}

To prove Proposition \ref{prop:BD(S)}, we require a lemma which describes $\rho_<(\x), \rho_>(\x)$ in terms of maximizing, minimizing over finite sets.

\begin{lemma} \label{lem:rho_conditions}
Let $\cB \in \fB(\cS)$. Then
\begin{align*}
\rho_<(\x) &= \max_{\ell \in [[1,n]]} \{ \tau_i^{-1} e^{\min(\x, V_\ell)} : V_{\ell-1} < \x, \cB'((V_{\ell-1},V_\ell)) = \frac{i}{p} \} \\
\rho_>(\x) &= \min_{\ell \in [[1,n]]} \{\tau_{i+1}^{-1} e^{\max(V_{\ell-1},\x)}: V_\ell > \x, \cB'((V_{\ell-1},V_\ell)) = \frac{i}{p} \}
\end{align*}
Moreover,
\begin{enumerate}
    \item if $\x \in (V_{\ell-1},V_\ell]$ and $\cB'(\x^-) = \frac{i}{p}$, then
    \[ \rho_<(\x) = \max(\rho_<(V_{\ell-1}), \tau_i^{-1} e^{\x}), \]
    \item and if $\x \in [V_{\ell-1},V_\ell)$ and $\cB'(\x^+) = \frac{i}{p}$, then
    \[ \rho_>(\x) = \min(\rho_>(V_\ell), \tau_{i+1}^{-1} e^{\x}). \]
\end{enumerate}
\end{lemma}
\begin{proof}
Suppose $\xi \le \x$ with $\xi \in (V_{\ell-1},V_\ell]$ and $\cB'(\xi^-) = \frac{i}{p}$. Observe that
\[ \cB'(\xi^-) = \cB'((V_{\ell-1},V_\ell)) = \cB'(\min(\x,V_\ell)^-) \]
which is $\frac{i}{p}$. Then
\[ \tau_i^{-1} e^\xi \le \tau_i^{-1} e^{\min(\x,V_\ell)^-} \le \rho_<(\x). \]
Maximizing over all $\ell \in[[1,n]]$ proves the statement for $\rho_<(\x)$. A similar argument yields the expression for $\rho_>(\x)$. The rest of the lemma follows from the ascertained form for $\rho_<(\x)$ and $\rho_>(\x)$.
\end{proof}

\begin{proof}[Proof of Proposition \ref{prop:BD(S)}]
Suppose $\x \in (V_{\ell-1},V_\ell)$ is a singular point of $\cB$. By Lemma \ref{lem:rho_conditions}, one of the following equalities must hold
\[ \rho_<(V_{\ell-1}) = \rho_>(V_\ell), \quad \tau_i^{-1} e^{\x} = \tau_{i+1}^{-1} e^{\x}, \quad \rho_<(V_{\ell-1}) = \tau_{i+1}^{-1} e^{\x}, \quad \tau_i^{-1} e^{\x} = \rho_>(V_\ell). \]
Observe that the latter two equalities cannot hold, we have
\[ \rho_<(V_{\ell-1}) \le \rho_>(V_{\ell-1}) \le \tau_{i+1}^{-1} e^{V_{\ell-1}} < \tau_{i+1}^{-1} e^{\x} \]
and similarly $\tau_i^{-1} e^{\x} < \rho_>(V_\ell)$. Thus if $\x \in (V_{\ell-1},V_\ell)$ is singular, then
\[ \rho_<(V_{\ell-1}) = \rho_>(V_\ell), \quad \mbox{or} \quad \tau_i^{-1} = \tau_{i+1}^{-1}. \]
However note that these equalities are independent of $\x \in (V_{\ell-1},V_\ell)$. In particular, this means that $\cB$ has a singular point in $(V_{\ell-1},V_\ell)$ if and only if every point is singular in $(V_{\ell-1},V_\ell)$.

Therefore, $\cB$ has finitely many singular points if and only if there are no singular points in $I \setminus \{V_\ell\}_{\ell=0}^n$. By our discussion above, this is true if and only if
\[ \rho_<(V_{\ell-1}) < \rho_>(V_\ell) \]
for every $\ell \in [[1,n]]$ and
\[ \tau_i^{-1} < \tau_{i+1}^{-1} \]
for any $0 \le i \le p$ such that $\cB'((V_{\ell-1},V_\ell)) = \frac{i}{p}$. The latter condition is equivalent to $i = \sum_{j=1}^a |S_{\sigma_j}|$ for some $0 \le a \le d$ in which case $\cB'(x) \in \{\varsigma_i\}_{i=0}^d$. This proves the desired equivalent description of $\fB^\Delta(\cS)$.

We now characterize the singular points. Since the singular points of $\cB \in \fB^\Delta(\cS)$ are necessarily non-differentiable points of $\cB$, we may our singular point is $V_\ell$ for some $1 \le \ell < n$ with $\cB'(V_\ell) = \varsigma_i$ and $\cB'(V_\ell^+) = \varsigma_j$. Since
\[ \max(\rho_<(V_{\ell-1}), \sigma_i^{-1} e^{V_\ell}) = \rho_<(V_\ell) = \rho_>(V_\ell) = \min(\rho_<(V_{\ell+1}), \sigma_{j+1}^{-1} e^{V_\ell}) \]
Since $\rho_<(V_{\ell-1}) < \rho_>(V_\ell)$ and $\rho_<(V_\ell) < \rho_>(V_{\ell+1})$, we must have
\[ \sigma_i^{-1} e^{V_\ell} = \rho_<(V_\ell) = \rho_>(V_\ell) = \sigma_{j+1}^{-1} e^{V_\ell}. \]
This is the case when $j = i - 1$.
\end{proof}

\section{Asymptotics} \label{sec:asymp}

In this section, we obtain asymptotics of the observables from Section \ref{sec:obs} under our limit regime. We then give relevant definitions prior to the statement of the main theorems for this section (Theorems \ref{thm:mom} and \ref{thm:gauss}).

Define
\begin{align} \label{eq:G<>}
\begin{split}
\cG_{<\x}(z) &= \prod_{\sigma \in \cS} \prod_{\substack{1 \le \ell \le n \\ V_{\ell-1} < \x}} \left( \frac{1 - e^{\min(V_\ell,\mathbf{x})} (\sigma z)^{-1}}{1 - e^{V_{\ell-1}} (\sigma z)^{-1}} \right)^{\frac{1}{p} \1[\cB'(V_\ell^-) \ge \varsigma_\sigma]}, \\
\cG_{>\x}(z) &= \prod_{\sigma \in \cS} \prod_{\substack{1 \le \ell \le n \\ V_\ell > \x}} \left( \frac{1 - e^{-\max(V_{\ell-1},\mathbf{x})} \sigma z}{1 - e^{-V_\ell} \sigma z} \right)^{\frac{1}{p} \1[\cB'(V_\ell^-) < \varsigma_\sigma]}.
\end{split}
\end{align}
where we set $\varsigma_\sigma = \varsigma_{\sigma_i}$ if $\sigma = \sigma_i$. Here the branches are chosen so that the argument is $0$ for $|z|$ large and real. Recall that $\cS$ is a multiset of $p$ elements up to multiplicity with $d$ distinct values. In particular, the product over $\sigma \in \cS$ is a product over $p$ terms.

\begin{definition} \label{def:singpt}
Let $\P^{B,r,\vec{s}}_{\alpha,\ft}$ satisfy the Limit Conditions and let $V_{\ell_1},\ldots,V_{\ell_\nu}$ be the singular points of $\cB$. Given $d > 0$, we say that $x(\e)$ is $d$-\emph{separated from singular points} if
\[ |x(\e) - v_{\ell_i}(\e)| \ge d, \quad 1 \le i \le \nu \]
for all sufficiently small $\e > 0$.
\end{definition}

We now present the main results for this section. Let $(\pi^x)_{x\in I^\e}$ denote the diagonal sections of $\P^{B,r,\vec{s}}_{\alpha,\ft}$.

\begin{theorem} \label{thm:mom}
Suppose $\P^{B,r,\vec{s}}_{\alpha,\ft}$ satisfies the Limit Conditions with $\cB \in \fB^\Delta(\cS)$. Fix $\mathbf{x} \in I, k \in \Z_{\ge 0}$. Let $x \in I_V^\e$ be $k\ft$-separated from singular points and satisfy $\e x \to \mathbf{x}$. Then
\begin{align} \label{eq:MOM}
\E \wp_k(\pi^x;r^{\ft \alpha},r^\ft) \to \frac{1}{2\pi\i} \oint \! [ \cG^{\cB}_{<\mathbf{x}}(z) \cG^{\cB}_{>\mathbf{x}}(z) ]^{k\ft} \, \frac{dz}{z}
\end{align}
as $\e \to 0$, where the contour is positively oriented around $[0,\rho_<(\x))$ and does not contain $(\rho_>(\x),\infty)$; recall $\rho_<(\x)$ and $\rho_>(\x)$ are defined in (\ref{eq:rho<>}).
\end{theorem}

\begin{theorem} \label{thm:gauss}
Suppose $\P^{B,r,\vec{s}}_{\alpha,\ft}$ satisfies the Limit Conditions with $\cB \in \fB^\Delta(\cS)$. Fix $\mathbf{x}_1 \le \cdots \le \mathbf{x}_m$, $k_1,\ldots,k_m \in \Z_{\ge 0}$. Let $x_1 \le \cdots \le x_m$ in $I^\e_V$ be such that $x_i$ is $2k_i\ft$-separated from singular points and satisfies $\e x_i \to \mathbf{x}_i$ for each $i \in [[1,m]]$. Then the vector
\[ \left( \frac{1}{\e}(\wp_{k_1}(\pi^{x_1};r^{\ft \alpha},r^\ft) - \E \wp_{k_1}(\pi^{x_1};r^{\ft \alpha},r^\ft)), \ldots, \frac{1}{\e}(\wp_{k_m}(\pi^{x_m};r^{\ft \alpha},r^\ft) - \E \wp_{k_m}(\pi^{x_m};r^{\ft \alpha},r^\ft)) \right) \]
converges in distribution as $\e \to 0$ to the centered gaussian vector $(\fD_{k_1}(\mathbf{x}_1),\ldots \fD_{k_m}(\mathbf{x}_m))$ with covariance defined by
\begin{align} \label{eq:thmcov}
\cov\left( \fD_{k_i}(\mathbf{x}_i), \fD_{k_j}(\mathbf{x}_j) \right) = \frac{k_i k_j}{(2\pi\i)^2} \oint \! \oint \! \frac{ [\cG^{\cB}_{<\mathbf{x}_i}(z) \cG^{\cB}_{>\mathbf{x}_i}(z)]^{k_i\ft} [ \cG^{\cB}_{<\mathbf{x}_j}(w) \cG^{\cB}_{>\mathbf{x}_j}(w)]^{k_j\ft}}{(z-w)^2} dz \, dw, \quad 1 \le i < j \le m
\end{align}
where
\begin{itemize}
    \item the $z$-contour is positively oriented around $[0,\rho_<(\x_i))$ but does not contain $(\rho_>(\x_i),\infty)$,
    \item the $w$-contour is positively oriented around $[0,\rho_<(\x_j))$ but does not contain $(\rho_>(\x_j),\infty)$,
    \item and the $z$-contour is enclosed by the $w$-contour.
\end{itemize}
If $\rho_<(\x_i) = \rho_>(\x_j) =: \rho$, then the $z$- and $w$-contours intersect at $\rho$.
\end{theorem}

The remainder of this section is devoted to the proofs of Theorems \ref{thm:mom} and \ref{thm:gauss}. We begin by collecting some asymptotic preliminaries. We then give an outline of the proofs to illustrate the key ideas, followed by the rigorous proofs.

\subsection{Preliminary Asymptotics}

In preparation for the asymptotics of the formula from Theorem \ref{thm:obs}, we study the asymptotics $G^B_{<x}, G^B_{>x}$ and some asymptotic properties of their poles, formulated in two propositions. The latter is important for understanding the placement of contours in the analysis of moments. Before presenting the propositions, we introduce some notation to work with the poles of $G^B_{<x}, G^B_{>x}$.

Given $\P^{B,r,\vec{s}}_{\alpha,\ft}$ satisfying the Limit Conditions such that $\cB \in \fB^\Delta(\cS)$, we define the ($\e$-dependent) sets
\begin{align} \label{eq:Ril}
R_{<x}^{a,\ell} &= \{ (s_0 \cdots s_a r^e)^{-1}: e \in (v_{\ell-1},v_\ell) \cap (a + p\Z + \frac{1}{2}), e < x, \cB'(V_{\ell-1},V_\ell) \ge \varsigma_{s_0 \cdots s_a} \}, \\ \nonumber
R_{>x}^{a,\ell} &= \{ (t s_0 \cdots s_a r^e)^{-1}: e \in (v_\ell,v_{\ell+1}) \cap (a + p\Z + \frac{1}{2}), e > x, \cB'(V_\ell,V_{\ell+1}) < \varsigma_{s_0 \cdots s_a} \}.
\end{align}

Recall the $q$-Pochhammer symbol
\[ (a;q)_N = \prod_{i=0}^{N-1} (1 - aq^i), \quad (a;q)_\infty = \prod_{i=0}^\infty (1 - a q^i) \]

\begin{lemma} \label{lem:rest}
Suppose $\P^{B,r,\vec{s}}_{\alpha,\ft}$ satisfies the Limit Conditions such that $\cB \in \fB^\Delta(\cS)$. For $\e > 0$ sufficiently small,
\begin{align} \label{eq:Gpoch}
\begin{split}
G_{<x}^B(z;\e,\ft) &= \prod_{R_{<x}^{a,\ell} \ne \emptyset} \frac{(t^{-1}(\max R_{<x}^{a,\ell}) z^{-1};r^p)_\infty}{((\max R_{<x}^{a,\ell}) z^{-1} ;r^p)_\infty} \cdot \frac{(r^p(\min R_{<x}^{a,\ell}) z^{-1} ;r^p)_\infty}{(t^{-1}r^p(\min R_{<x}^{a,\ell}) z^{-1};r^p)_\infty} \\
G_{>x}^B(z;\e,\ft) &= \prod_{R_{>x}^{a,\ell} \ne \emptyset}  \frac{(t^{-1}(\min R_{>x}^{a,\ell})^{-1} z ;r^p)_\infty}{((\min R_{>x}^{a,\ell})^{-1} z;r^p)_\infty} \cdot \frac{(r^p(\max R_{>x}^{a,\ell})^{-1}z;r^p)_\infty}{(t^{-1} r^p(\max R_{>x}^{a,\ell})^{-1}z ;r^p)_\infty} .
\end{split}
\end{align}
\end{lemma}
\begin{proof}
If $\cB'((V_{\ell-1},V_\ell)) = \varsigma_i$, then Lemma \ref{lem:sigma_order} implies
\[ B' = \1\left[ \bigcup_{1 \le j \le i} (S_{\sigma_j} + p\Z + \frac{1}{2}) \right] = \1[ \cB'((V_{\ell-1},V_\ell)) \ge \varsigma_i] \]
on $I_E^\e \cap (v_{\ell-1},v_\ell)$. By the definitions of $G_{<x}^B$ and $R_{<x}^{a,\ell}$, we have
\[ G_{<x}^B(z;\e,\ft) = \prod_{R_{<x}^{a,\ell} \ne \emptyset} \prod_{\rho \in R_{<x}^{a,\ell}} \frac{1 - t^{-1} \rho z^{-1}}{1 - \rho z^{-1}}. \]
Since the elements of $R_{<x}^{a,\ell}$ are
\[ \max R_{<x}^{a,\ell}, \quad r^p \max R_{<x}^{a,\ell}, \quad r^{2p} \max R_{<x}^{a,\ell}, \quad \ldots, \quad \min R_{<x}^{a,\ell}, \]
in decreasing order, we can rewrite the product over $\rho$ as the ratios of Pochhammer symbols in (\ref{eq:Gpoch}). The argument is similar for $G_{>x}^B$.
\end{proof}

The first proposition in this section describes the behavior of the largest pole of $G_{<x}^B$ and the smallest pole of $G_{>x}^B$.

\begin{proposition} \label{prop:dist<>}
Suppose $\P^{B,r,\vec{s}}_{\alpha,\ft}$ satisfies the Limit Conditions such that $\cB \in \fB^\Delta(\cS)$. For $x\in I^\e$, let $\rho_<^\e(x)$ and $\rho_>^\e(x)$ denote the maximal pole of $G_{<x}^B(z;\e,\ft)$ and minimal pole of $G_{>x}^B(z;\e,\ft)$ respectively. Suppose $\x \in I$ and $x \in I^\e$ satisfy $\e x \to \x$ as $\e \to 0$.
\begin{enumerate}[(a)]
    \item \label{dist<>:1} If $\x \in I \setminus \{V_\ell\}_{\ell=0}^n$, then
    \[ \lim_{\e \to 0} \rho_<^\e(x) = \rho_<(\x), \quad \quad \lim_{\e \to 0} \rho_>^\e(x) = \rho_>(\x). \]
    
    \item \label{dist<>:2} If $\x = V_\ell$ is not a singular point and
    \begin{enumerate}[(i)]
        \item \label{dist<>:2a} $x > v_\ell$ for all $\e > 0$, then
        \[ \lim_{\e \to 0} \rho_<^\e(x) = \rho_<(V_\ell^+), \quad \quad \lim_{\e \to 0} \rho_>^\e(x) = \rho_>(V_\ell^+) = \rho_>(V_\ell). \]
        
        \item \label{dist<>:2b} $x < v_\ell$ for all $\e > 0$, then
        \[ \lim_{\e \to 0} \rho_<^\e(x) = \rho_<(V_\ell^-)  = \rho_<(V_\ell), \quad \quad \lim_{\e \to 0} \rho_>^\e(x) = \rho_>(V_\ell^-). \]
        
        \item \label{dist<>:2c} $x = v_\ell$ for all $\e > 0$, then
        \[ \lim_{\e \to 0} \rho_<^\e(x) = \rho_<(V_\ell), \quad \quad \lim_{\e \to 0} \rho_>^\e(x) = \rho_>(V_\ell). \] 
    \end{enumerate}
    
    \item \label{dist<>:3} If $\x = V_\ell$ is a singular point, then
    \[ \lim_{\e \to 0} \rho_<^\e(x) = \rho_<(\x) = \rho_>(\x) = \lim_{\e \to 0} \rho_>^\e(x). \]
    Furthermore, if $\cB'(V_\ell) = \varsigma_\sigma$, then
    \begin{align} \label{eq:rho<_union}
    \lim_{\e \to 0}  \max \left( \bigcup_{a \in S_\sigma} R_{<x}^{a,\ell} \right) = \rho_<(V_\ell) > &  \limsup_{\e \to 0}  \max \{\mbox{poles of $G_{<x}^B(z;\e,\ft)$}\} \setminus \left( \bigcup_{a \in S_\sigma} R_{<x}^{a,\ell} \right), \\ \label{eq:rho>_union}
    \lim_{\e \to 0} \min \left( \bigcup_{a \in S_\sigma} R_{>x}^{a,\ell} \right) = \rho_>(V_\ell) < & \liminf_{\e \to 0} \min \{\mbox{poles of $G_{>x}^B(z;\e,\ft)$}\} \setminus \left( \bigcup_{a \in S_\sigma} R_{>x}^{a,\ell} \right), \\ \label{eq:rho_squeeze}
    t^{-1} r^{-|x-v_\ell|-1} & \le \frac{\rho_>^\e(x)}{\rho_<^\e(x)} \le t^{-1} r^{-|x-v_\ell|-2p+1},
    \end{align}
    for $\e > 0$ sufficiently small.
\end{enumerate}
\end{proposition}

The next proposition gives asymptotics of $G_{<x}^B, G_{>x}^B$ with special precision given to points near maximal and minimal poles respectively.

Let $\theta \in (0,\pi)$, $\delta > 0$. Define
\[ D^{\e,\theta,\delta} = \{w \in \C: \mathrm{dist}(w,[1,\infty)) \le \delta\} \cap \{ w \in \C: |\arg(w - (1-\e))| \le \theta \}, \]
that is the neighborhood formed by the $\delta$-neighborhood of $[1,\infty)$ clipping away the parts separated by the rays of arguments $\theta$ and $-\theta$ started from $1-\e$.

\begin{proposition} \label{prop:Gconv}
Suppose $\P^{B,r,\vec{s}}_{\alpha,\ft}$ satisfies the Limit Conditions such that $\cB \in \fB^\Delta(\cS)$. For $x\in I^\e$, let $\rho_<^\e(x)$ and $\rho_>^\e(x)$ denote the maximal pole of $G_{<x}^B(z;\e,\ft)$ and minimal pole of $G_{>x}^B(z;\e,\ft)$ respectively. Suppose $x \in I^\e_V$ such that $\e x \to \x \in I$. Further assume that if $\x = V_\ell$ with $V_\ell$ not a singular point, then either $x > v_\ell$, $x = v_\ell$, or $x < v_\ell$ independent of $\e > 0$. Then
\begin{align*}
G_{<x}^B(t^{-1} z^{-1} \cdot \rho_<^\e(x);\e,\ft) &=  \cG_{<\x}(z^{-1} \cdot \rho_<(\x))^\ft \exp\left( O\left( \e \frac{|z|\vee |z|^2}{|1 - z|} \right) \right) \\
G_{>x}^B(tz \cdot\rho_>^\e(x);\e,\ft) &= \cG_{>\x}(z \cdot\rho_>(\x))^\ft \exp\left( O\left( \e \frac{|z|\vee |z|^2}{|1 - z|} \right) \right) 
\end{align*}
uniformly for $z \in D^{\e,\theta,\delta}$ and $\e > 0$ sufficiently small.
\end{proposition}

The remainder of this subsection is devoted to the proofs of Propositions \ref{prop:dist<>} and \ref{prop:Gconv}.

\begin{proof}[Proof of Proposition \ref{prop:dist<>}]
Suppose throughout the proof, $\e$ is small enough so that the conclusion of Lemma \ref{lem:rest} is true and so that $v_\ell - v_{\ell-1} > p$ for each $\ell \in [[1,n]]$ (that is we have at least one period in $(v_{\ell-1},v_\ell)$). Then
\begin{align} \label{eq:rho<_max}
\rho_<^\e(x) = \max \left( \bigcup_{a \in [[0,p-1]], m \in [[1,n]]} R_{<x}^{a,m} \right)
\end{align}
where the above set being maximized over is the set of poles of $G_{<x}^B(z;\e,\ft)$. Observe that if $R_{<x}^{a,m}$ is nonempty, then
\[ \max R_{<x}^{a,m} = (s_0 \cdots s_a)^{-1} e^{\min(\x,V_m)} + o(1) \]
as $\e \to 0$. The set $R_{<x}^{a,m}$ is nonempty if $v_{m-1} < x$ and $\cB'((V_{m-1},V_m)) \ge \varsigma_{s_0 \cdots s_a}$; here we used the fact that $v_m - v_{m-1} > p$, otherwise it is possible that $\Z + a + \frac{1}{2}$ and $(v_{m-1},v_m)$ do not intersect.

We first prove the convergence statements in (\ref{dist<>:1}), (\ref{dist<>:2}), (\ref{dist<>:3}) for $\rho_<^\e(x)$. The argument for $\rho_>^\e(x)$ is similar.

\underline{(\ref{dist<>:1}).} If $\x \in I \setminus \{V_\ell\}_{\ell=0}^n$, then the set $R_{<x}^{a,m}$ is nonempty if and only if $V_{m-1} < \x$ and $\cB'((V_{m-1},V_m)) \ge \varsigma_{s_0 \cdots s_a}$. Then for $\e x \to \x$, we have
\begin{align} \label{eq:rho<_conv}
\begin{split}
\rho_<^\e(x) & \to \max \{ (s_0 \cdots s_a)^{-1} e^{\min(\x,V_m)}: V_{m-1} < \x, \cB'((V_{m-1},V_m)) \ge \varsigma_{s_0 \cdots s_a} \} \\
& = \max \{ (s_0 \cdots s_a)^{-1} e^{\min(\x,V_m)}: V_{m-1} < \x, \cB'((V_{m-1},V_m)) = \varsigma_{s_0 \cdots s_a} \} = \rho_<(\x)
\end{split}
\end{align}
as $\e \to 0$, where the first equality follows from the fact that $\varsigma_{s_0 \cdots s_a}$ is an increasing function of $(s_0 \cdots s_a)^{-1}$ and the second equality follows from Lemma \ref{lem:rho_conditions}.

\underline{(\ref{dist<>:2b}), (\ref{dist<>:2c}) and (\ref{dist<>:3}).} Suppose $\x = V_\ell$ for some $0 < \ell \le n$ and $x \le v_\ell$ for all $\e > 0$. In this case, we again have that the set $R_{<x}^{a,m}$ is nonempty if and only if $V_{m-1} < \x = V_\ell$ and $\cB'((V_{m-1},V_m)) \ge \varsigma_{s_0 \cdots s_a}$. Then (\ref{eq:rho<_conv}) holds for this case as well.

\underline{(\ref{dist<>:2a}) and (\ref{dist<>:3}).} Suppose $\x = V_\ell$ for some $0 \le \ell < n$ and $x > v_\ell$ for all $\e > 0$. Then $R_{<x}^{a,m}$ is nonempty if and only if $V_{m-1} \le V_\ell = \x$ and $\cB'((V_{m-1},V_m)) \ge \varsigma_{s_0 \cdots s_a}$. Then
\begin{align*}
\rho_<^\e(x) & \to \max \{ (s_0 \cdots s_a)^{-1} e^{\min(\x,V_m)}: V_{m-1} \le \x = V_\ell, \cB'((V_{m-1},V_m)) \ge \varsigma_{s_0 \cdots s_a} \} \\
& = \max \{ (s_0 \cdots s_a)^{-1} e^{\min(\x,V_m)}: V_{m-1} \le \x = V_\ell, \cB'((V_{m-1},V_m)) = \varsigma_{s_0 \cdots s_a} \} = \rho_<(V_\ell^+)
\end{align*}
as $\e \to 0$, by the same reasoning as before.

Note that combining the latter two cases gives us the complete case of the convergence of $\rho_<^\e(x)$ for (\ref{dist<>:3}). It remains to show (\ref{eq:rho<_union}), (\ref{eq:rho>_union}), (\ref{eq:rho_squeeze}). For the remainder of the proof, assume $\x = V_\ell$ is a singular point.

\underline{(\ref{dist<>:3}): (\ref{eq:rho<_union}).} By the argument preceding the case analysis above, $R_{<x}^{a,m}$ is empty if $m > \ell+1$. We rewrite the union
\begin{align} \label{eq:three_union}
\begin{multlined}
\left( \bigcup_{a \in [[0,p-1]], m \in [[1,n]]} R_{<x}^{a,m} \right) =  \left( \bigcup_{a \in [[0,p-1]], m < \ell} R_{<x}^{a,m} \right) \cup \left( \bigcup_{a \in [[0,p-1]]} R_{<x}^{a,\ell+1} \right) \\
\cup \left( \bigcup_{a \in [[0,p-1]]\setminus S_\sigma} R_{<x}^{a,\ell} \right) \cup \left( \bigcup_{a \in S_\sigma} R_{<x}^{a,\ell} \right)
\end{multlined}
\end{align}
as four smaller unions. We want to show that the $\limsup$ of the maximum of the first three unions converges to $\rho < \rho_<(V_\ell)$, then by (\ref{eq:rho<_max}) the maximum of the fourth union must be equal to $\rho_<^\e(x)$. Since $\rho_<^\e(x) \to \rho_<(V_\ell)$ as $\e \to 0$, this establishes (\ref{eq:three_union}). Note that the maximum of the first union is exactly $\rho_<^\e(v_{\ell-1})$ and therefore converges to $\rho_<(V_{\ell-1})$ as $\e \to 0$. By Proposition \ref{prop:BD(S)}, 
\[ \rho_<(V_{\ell-1}) < \rho_>(V_\ell) = \rho_<(V_\ell) \]
because $V_\ell$ is a singular point. We have thus shown
\[ \lim_{\e \to 0} \max \left( \bigcup_{a \in [[0,p-1]], m < \ell} R_{<x}^{a,m} \right) < \rho_<(V_\ell). \]

For the second union in (\ref{eq:three_union}), observe that by Proposition \ref{prop:BD(S)}, if $\sigma = \sigma_i$ then
\[ \cB'((V_\ell,V_{\ell+1})) = \cB'(V_\ell^+) = \varsigma_{i-1}. \]
Thus, arguing as in the case analysis above,
\begin{align*}
\limsup_{\e \to 0} \max \left( \bigcup_{a \in [[0,p-1]]} R_{<x}^{a,\ell+1} \right) &\le \max\{(s_0 \cdots s_a)^{-1} e^{V_\ell}: \varsigma_{i-1} = \cB'((V_\ell,V_{\ell+1})) \ge \varsigma_{s_0 \cdots s_a} \} \\
&\le \sigma_{i-1}^{-1} e^{V_\ell} < \sigma_i^{-1} e^{V_\ell} = \rho_<(V_\ell)
\end{align*}
where we recall $\varsigma_{i-1} = \varsigma_{\sigma_{i-1}}$ and the final equality follows from Proposition \ref{prop:BD(S)}. It is necessary to take the $\limsup$ since the set above may be empty depending on whether $x < v_\ell$ or $x > v_\ell$. Also, note that the second inequality is equality if the preceding set is nonempty (this is when $i > 0$). Otherwise it is strict since $\max \emptyset = -\infty < 0$; recall that $\sigma_0 = \infty$ so we take $\sigma_0^{-1} = 0$. The third union is similar, we have $\cB'((V_{\ell-1},V_\ell)) = \varsigma_\sigma = \varsigma_i$ by assumption. Then
\begin{align*}
\lim_{\e\to 0} \max \left( \bigcup_{a \in [[0,p-1]] \setminus S_\sigma} R_{<x}^{a,\ell} \right) &= \max\{(s_0 \cdots s_a)^{-1} e^{V_\ell}: \varsigma_i = \cB'((V_\ell,V_{\ell+1})) \ge \varsigma_{s_0 \cdots s_a}, s_0 \cdots s_a \ne \sigma_i \} \\
&\le \sigma_{i-1}^{-1} e^{V_\ell} < \sigma_i^{-1} e^{V_\ell} = \rho_<(V_\ell).
\end{align*}
This proves (\ref{eq:rho<_union}).

\underline{(c): (\ref{eq:rho>_union}).} Uses an analogous argument as for (\ref{eq:rho<_union}).

\underline{(c): (\ref{eq:rho_squeeze}).} Using (\ref{eq:rho<_union}) and (\ref{eq:rho>_union}), notice that
\[ \rho_<^\e(x) = \max \left( \bigcup_{a \in S_\sigma} R_{<x}^{a,\ell} \right) = \sigma^{-1} r^{-e_1}, \quad \rho_>^\e(x) = \min \left( \bigcup_{a \in S_\sigma} R_{>x}^{a,\ell} \right) = t^{-1} \sigma^{-1} r^{-e_2} \]
where
\[ e_1 = \max\{ e \in (v_{\ell-1},v_\ell) \cap (S_\sigma + p\Z + \frac{1}{2}): e < x\}, \quad e_2 = \min\{e \in (v_\ell,v_{\ell+1}) \cap (S_\sigma + p\Z + \frac{1}{2}): e > x\}. \]
We have the inequalities
\begin{align*}
\min(v_\ell,x) - p + \frac{1}{2} & \le e_1 \le \min(v_\ell,x) - \frac{1}{2} \\
\max(v_\ell,x) + \frac{1}{2} & \le e_2 \le \max(v_\ell,x) + p - \frac{1}{2}.
\end{align*}
Thus we conclude (\ref{eq:rho_squeeze}).
\end{proof}

Before going into the proof of Proposition \ref{prop:Gconv}, we first require a lemma on the asymptotics of Pochhammer symbols.

\begin{lemma} \label{lem:pochbd}
Let $\fa > 0$ and suppose $N(\e) \in \Z_{>0}$ such that $\limsup_{\e \to 0} \e N(\e) > 0$ as $\e \to 0$. Then for any fixed $\theta \in (0,\pi), \delta > 0$, we have
\[ \frac{(z;e^{-\e})_{N(\e)}}{(e^{-\e \fa}z;e^{-\e})_{N(\e)}} = \left(\frac{1 - z}{1-r^N z} \right)^\fa \exp\left( O\left(\e \frac{|z| \vee |z|^2}{|1 - z|} \right) \right)\]
uniformly for $z$ in $\C \setminus D^{\e,\theta,\delta}$ and $\e$ arbitrarily small.
\end{lemma}
\begin{proof}
Throughout the proof, the constant symbol $C$ is independent of $\e > 0$ (though it may depend on $\fa, \theta$), and may change from line to line.

Set $r = e^{-\e}$, $N = N(\e)$, $D^\e = D^{\e,\theta,\delta}$. Define
\begin{align*}
E^\e_1(z) &= \log \frac{(z;r)_N}{(r^\fa z;r)_N} - \frac{1 - r^\fa}{1 - r} \log\left( \frac{1 - z}{1 - r^Nz} \right), \\
E^\e_2(z) &= \frac{1 - r^\fa}{1 - r} \log\left( \frac{1 - z}{1 - r^Nz} \right) - \fa \log\left( \frac{1 - z}{1 - r^N z} \right) \\
E^\e(z) &= E^\e_1(z) + E^\e_2(z) = \log \frac{(z;r)_N}{(r^\fa z;r)_N} - \fa \log\left( \frac{1 - z}{1 - r^N z} \right).
\end{align*}

Using the fact that
\[ 1 - r^\fa = 1 - (1 - (1 - r))^\fa = \fa (1 - r) - \binom{\fa}{2} (1 - r)^2 + O(\e^3), \]
we have
\[ E_2^\e(z) = \left( \frac{1 - r^\fa}{1 - r} - \fa \right) \log \left( \frac{1 - z}{1 - r^N z} \right) = \left( - \binom{\fa}{2}\e + O(\e^2) \right) \log \left( \frac{1 - z}{1 - r^N z} \right). \]
The $\log$ term is bounded for large $|z|$, behaves as $C|z|$ for small $|z|$, and has logarithmic singularities at $z = 1$ and $z = r^{-N}$. Since $\C \setminus D^\e$ is separated by a constant distance from the $r^N $ singularity by the assumption $\limsup_{\e \to 0} \e N(\e) > 0$, we may disregard the singularity at $z = r^{-N}$ and have
\[ |E_2^\e(z)| \le C \e \frac{|z|}{|1 - z|} \]
for $z \in \C \setminus D^\e$. Note that we accounted for the $\log$ singularity at $1$ with the ``wasteful'' $|1 - z|$ term in the denominator and the boundedness for $|z|$ large by balancing the right hand side.

Observe that
\begin{align*}
\log \frac{(z;r)_N}{(r^\fa z;r)_N} = - \sum_{j=0}^{N-1} \int_{r^{\fa + j}}^{r^j} \frac{z}{1 - uz} \, du = - \frac{1 - r^\fa}{1 - r} \int_{r^N}^1 \frac{z}{1 - u(v) z} \, dv
\end{align*}
by the change of variables
\begin{align}
\begin{split} \label{eq:CoV}
(r^{j+1},r^j] \ni v \mapsto u(v) & = \frac{1 - r^\fa}{1 - r} (v - r^j) + r^j \in (r^{j+\fa},r^j].
\end{split}
\end{align}
We may similarly write
\[ \frac{1 - r^\fa}{1 - r} \log \left( \frac{1 - z}{1 - r^Nz} \right) = -\frac{1 - r^\fa}{1 - r} \int_{r^N}^1 \frac{z}{1 - vz} \, dv. \]
Thus
\begin{align*}
E_1^\e(z) = \frac{1 - r^\fa}{1 - r} \int_{r^N}^1 \left( \frac{z}{1 - u(v) z} - \frac{z}{1 - vz} \right) \, dv = \frac{1 - r^\fa}{1- r} \int_{r^N}^1 \frac{(u(v) - v) z^2}{(1 - u(v) z)(1 - vz)} \, dv.
\end{align*}

For $v \in [r^{j+1},r^j]$,
\[ u(v) - v = \frac{r - r^\fa}{1 - r}(v - r^j), \]
so that
\[ |u(v) - v| < C \e \]
where the constant is also uniform over $v \in [r^N,1]$. Also for $v \in [r^{j+1},r^j]$,
\[ \left| \frac{1 - vz}{1 - u(v)z} \right| = \left| 1 + \frac{z(u(v) - v)}{1 - u(v)z} \right| \le C\]
for $z \in \C \setminus D^\e$. Indeed, $u(v)$ ranges from $r^{N-1 + \fa}$ to $1$ so that the buffer provided by $D^\e$ bounds $|1 - u(v)z|$ from below by some constant times $\e$. Thus
\[ |E_1^\e(z)| \le \frac{1 - r^\fa}{1 - r}\left|\int_{r^N}^1 \! \frac{(u(v)-v) z^2}{(1 - u(v) z)(1 - vz)} \, dv \right| \leq \int_{r^N}^1 \! \frac{C \, \e\, |z|^2}{|1 - vz|^2} \, dv \]
By writing
\[ \frac{1}{|1 - vz|^2} = \frac{1}{z - \bar{z}} \left( \frac{z}{1 - vz} - \frac{\bar{z}}{1 - v\bar{z}} \right), \quad v \in [r^N,1], \]
we obtain
\[ |E_1^\e(z)| \le \frac{C \, \e |z|^2}{z - \bar{z}}  \left[ \log \left( \frac{1 - \bar{z}}{1 - r^N \bar{z}} \right) - \log \left( \frac{1 - z}{1 - r^N z} \right) \right] = C \, \e \cdot \frac{|z|^2}{\Im z} \arg\left( \frac{1 - \bar{z}}{1 - r^N \bar{z}} \right) \]
Observe that
\[ 0 \le \frac{|z|}{\Im z} \arg\left( \frac{1 - \bar{z}}{1 - r^N \bar{z}} \right) = - \frac{|z|}{|1 - z|} \frac{1}{\sin \arg(1 - z)} \arg\left( \frac{1 - \bar{z}}{1 - r^N z} \right) \le C \frac{|z|}{|1 - z|}\]
for $z \in \C \setminus D^\e$. The latter bound follows from the fact that the $\sin \arg(1 - z)$ denominator term is balanced by $\arg \left( \frac{1 - \bar{z}}{1 - r^N z} \right)$ away from $[1,\infty)$; note that as $z$ approaches $\R$ we only need the fact that $\arg\left( \frac{1 - \bar{z}}{1 - r^Nz} \right)$ approaches $0$, however near $z = 1$ we really need $\arg(1 - \bar{z})$ to balance $\sin \arg(1 - z)$. Thus
\[ |E_1^\e(z)| \le C \e \frac{|z|^2}{|1 - z|} \]
uniformly for $z \in \C \setminus D^\e$. Combining our bounds, we obtain
\[ |E^\e(z)| \le C_1 \e \frac{|z| \vee |z|^2}{|1 - z|} \]
uniformly for $z \in \C \setminus D^\e$.
Our lemma now follows from
\[ \frac{(z;r)_N}{(r^\fa z;r)_N} = \exp(E^\e(z)). \]
\end{proof}

\begin{proof}[Proof of Proposition \ref{prop:Gconv}]
By Lemma \ref{lem:rest}, for sufficiently small $\e > 0$, we have
\[ G_{>x}^B(t z\cdot\rho_>^\e(x);\e,\ft) = \prod_{R_{>x}^{a,\ell} \ne \emptyset}  \frac{((\min R_{>x}^{a,\ell})^{-1} z \cdot \rho_>^\e(x) ;r^p)_\infty}{(t(\min R_{>x}^{a,\ell})^{-1} z \cdot \rho_>^\e(x);r^p)_\infty} \cdot \frac{(t r^p(\max R_{>x}^{a,\ell})^{-1}z \cdot \rho_>^\e(x);r^p)_\infty}{(r^p(\max R_{>x}^{a,\ell})^{-1}z \cdot \rho_>^\e(x) ;r^p)_\infty} \]
Fix $R_{>x}^{a,\ell} \ne \emptyset$. As established in the various cases in the proof of Proposition \ref{prop:dist<>}, the nonemptiness of $R_{>x}^{a,\ell}$ for $a,\ell$ fixed is independent of $\e > 0$ when $\e$ is sufficiently small, under our assumptions on $x$. We have
\[ \frac{((\min R_{>x}^{a,\ell})^{-1} z \cdot \rho_>^\e(x) ;r^p)_\infty}{(t(\min R_{>x}^{a,\ell})^{-1} z \cdot \rho_>^\e(x);r^p)_\infty} \cdot \frac{(t r^p(\max R_{>x}^{a,\ell})^{-1}z \cdot \rho_>^\e(x);r^p)_\infty}{(r^p(\max R_{>x}^{a,\ell})^{-1}z \cdot \rho_>^\e(x) ;r^p)_\infty} = \frac{((\min R_{>x}^{a,\ell})^{-1} z \cdot \rho_>^\e(x) ;r^p)_N}{(t(\min R_{>x}^{a,\ell})^{-1} z \cdot \rho_>^\e(x);r^p)_N} \]
where $N = N(\e)$ satisfies
\[ r^{p(N-1)} = \frac{\min R_{>x}^{a,\ell}}{\max R_{>x}^{a,\ell}}. \]
By definition of $\rho_>^\e(x)$,
\[ \frac{\rho_>^\e(x)}{\min R_{>x}^{a,\ell}} = r^\nu \]
for some $\nu = \nu(\e) \ge 0$. Note that $\e \nu \to \hat{\nu}$ as $\e \to 0$ for some $\hat{\nu}$; this follows from the fact that $\min R_{>x}^{a,\ell}$ converges to $(s_0 \cdots s_a)^{-1} e^{\max(\x,V_\ell)}$ and that $\rho_>^\e(x)$ converges by Proposition \ref{prop:dist<>}.

There are two main cases to consider: $\hat{\nu} > 0$ and $\hat{\nu} = 0$. First suppose $\hat{\nu} > 0$. Note that if $a_1 = a_1(\e), a_2 = a_2(\e)$ converge to some $\hat{a} > 0$ as $\e \to 0$, then
\[ \frac{1 - r^a z}{1 - r^b z} = \exp\left( \log\left( 1 + \frac{z(r^b - r^a)}{1 - r^b z} \right) \right) = e^{O(\e \cdot (|z| \wedge 1))}\]
uniformly over $z\in D^{0,\theta,\delta}$. Thus Lemma \ref{lem:pochbd} implies
\[ \frac{((\min R_{>x}^{a,\ell})^{-1} z \cdot \rho_>^\e(x) ;r^p)_N}{(t(\min R_{>x}^{a,\ell})^{-1} z \cdot \rho_>^\e(x);r^p)_N} = \left( \frac{1 - (s_0 \cdots s_a)^{-1} e^{\max(\x,V_\ell)} z \cdot \rho_>(\x)}{1 - (s_0 \cdots s_a)^{-1} e^{V_{\ell+1}} z \cdot \rho_>(\x)} \right)^{\frac{\ft}{p}} e^{O(\e(|z|\wedge 1))} \]
uniformly over $z \in D^{0,\theta,\delta}$.

If $\hat{\nu} = 0$, then write
\[ \frac{((\min R_{>x}^{a,\ell})^{-1} z/\rho_>^\e(x) ;r^p)_N}{(t(\min R_{>x}^{a,\ell})^{-1} z/\rho_>^\e(x);r^p)_N} = \frac{(r^\nu z ;r^p)_N}{(t r^\nu z;r^p)_N} \]
If, in addition, $\lim_{\e \to 0} \e N > 0$, then by Lemma \ref{lem:pochbd}, we have
\[ \frac{(r^\nu z ;r^p)_N}{(t r^\nu z;r^p)_N} = \left( \frac{1 - r^\nu z}{1 - (s_0 \cdots s_a)^{-1} e^{V_\ell+1} z \cdot \rho_>(\x)} \right)^{\frac{\ft}{p}} \exp\left( O\left( \e \frac{|z|\vee |z|^2}{|1 -z|} \right) \right) \]
uniformly over $z \in D^{\e,\theta,\delta}$. We may further replace the $1 - r^\nu z$ with $1 - z$ without changing the right hand side since
\begin{align} \label{eq:nu_small}
\frac{1 - r^\nu z}{1 - z} = \exp\left( \log \left( 1 + \frac{z(1 - r^\nu)}{1 - z} \right) \right) = \exp\left( O\left( \e \frac{|z|}{|1 -z|} \right) \right)
\end{align}
uniformly over $z \in D^{\e,\theta,\delta}$.

Otherwise, if $\lim_{\e \to 0} \e N = 0$, take $M = M(\e) \in \Z_{>0}$ large so that $\lim_{\e \to 0} \e M > 0$. Then by Lemma \ref{lem:pochbd},
\begin{align*}
\frac{(r^\nu z ;r^p)_{N+M}}{(t r^\nu z;r^p)_{N+M}}  &= \frac{(r^\nu z ;r^p)_{N+M}}{(t r^{\nu+N} z;r^p)_{N+M}} \cdot \frac{(t r^{\nu + N} z;r^p)_M}{(r^\nu z ;r^p)_M} \\
& = \left( \frac{1 - r^\nu z}{1 - r^{\nu + N}z} \right)^{\frac{\ft}{p}} \exp\left( O\left( \e \frac{|z| \vee |z|^2}{|1 - z|} \right) \right) = \exp\left( O \left( \e \frac{|z| \vee |z|^2}{|1 - z|} \right) \right)
\end{align*}
uniformly over $z \in D^{\e,\theta,\delta}$ where the last equality follows from applying (\ref{eq:nu_small}) to $\frac{1 - r^\nu z}{1 - z}$ and $\frac{1 - r^{N+\nu}z}{1 - z}$, then taking their quotient.

Multiplying over all $R_{>x}^{a,\ell} \ne \emptyset$ and comparing with $\cG_{>\x}$ as defined in (\ref{eq:G<>}), we obtain the desired asymptotics for $G_{>x}^B$. The argument for $G_{>x}^B$ is similar.
\end{proof}

\subsection{Outline of proofs}
We first outline the proof of Theorem \ref{thm:mom} to indicate the main obstacles. Part of the proof of Theorem \ref{thm:gauss} will mirror the outlined proof of Theorem \ref{thm:mom}.

By Theorem \ref{thm:obs} and (\ref{eq:DZ}), we can express $\E\wp_k(\pi^x;r)$ as
\begin{align} \label{eq:premom}
\frac{1}{(2\pi\i)^k} \oint_{\cC_1'} \cdots \oint_{\cC_k'} \frac{\sum_{i=1}^k \frac{1}{z_i} \frac{q^{i-1}}{t^{i-1}}}{(z_2 - \frac{q}{t}z_1) \cdots (z_k - \frac{q}{t}z_{k-1})} \prod_{i < j} \frac{(z_j - z_i)(z_j - \frac{q}{t}z_i)}{(z_j - \frac{1}{t}z_i)(z_j - qz_i)} \prod_{i=1}^k G_{<x}^{B^\e}(z_i;\e,\ft) G_{>x}^{B^\e}(z_i;\e,\ft) \, dz_i
\end{align}
where $\cC_i'$ is the $z_i$-contour and satisfies the conditions in Theorem \ref{thm:obs}. If the $\cC_i'$ converge to $\cC_i$ and are separated from one another so that the contours do not pass through any singularities of the integrand, then the integrand converges on the contour and the integral (\ref{eq:premom}) converges to
\begin{align} \label{eq:mom}
\frac{1}{(2\pi\i)^k} \oint_{\cC_1} \cdots \oint_{\cC_k} \frac{\sum_{i=1}^k \frac{1}{z_i}}{(z_2 - z_1) \cdots (z_k - z_{k-1})} \prod_{i=1}^k [\cG_{<\x}^{\cB}(z) \cG_{>\x}^{\cB}(z)]^\ft \, dz_i.
\end{align}
The dimension of this contour is in general higher than what we desire. However, we can obtain the desired form by applying the following Theorem from \cite[Corollary A.2]{GZ}.

\begin{theorem}[\cite{GZ}] \label{thm:dimred}
Let $s$ be a positive integer. Let $f,g_1,\ldots,g_s$ be meromorphic functions with possible poles at $\{\fp_1,\ldots,\fp_m\}$. Then for $k \ge 2$,
\begin{align*}
\begin{multlined}
\frac{1}{(2\pi\i)^k} \oint \cdots \oint \frac{1}{(v_2 - v_1)\cdots(v_k - v_{k-1})} \prod_{j=1}^s \left( \sum_{i=1}^k g_j(v_i) \right) \prod_{i=1}^k f(v_i)\, dv_i \\
= \frac{k^{s-1}}{2\pi\i} \oint f(v)^k \prod_{j=1}^s g_j(v)\,dv.
\end{multlined}
\end{align*}
where the contours contain $\{\fp_1,\ldots,\fp_m\}$ and on the left side we require that the $v_i$-contour is contained in the $v_j$-contour whenever $i < j$. 
\end{theorem}
We note that this was how the asymptotics of moments were carried out in \cite{GZ}.

In the case that $\mathbf{x} = V_\ell$ is a singular point, there is some additional difficulty due to the order of $\mathrm{dist}(\rho_<^\e(x),\rho_>^\e(x))$ being $O(\e)$ where $\rho_<^\e(x),\rho_>^\e(x)$ are defined as in Proposition \ref{prop:dist<>}. By the conditions on the contours in Theorem \ref{thm:obs}, our $k$th moment formula only makes sense if there is some separation between $v_\ell$ and $x$. This is where the $k\ft$-separation condition is needed. Even with this separation condition, the contours $\cC_i'$ in (\ref{eq:premom}) are $O(\e)$ from one another on the positive real axis; this is the main technical complication and is a byproduct of $\mathrm{dist}(\rho_<^\e(x),\rho_>^\e(x))$ being $O(\e)$. For points $(z_1,\ldots,z_n)$ where $|z_i - z_j| \sim O(\e)$, $i \ne j$, we take care to show that the integrand does not diverge. We still have convergence to (\ref{eq:mom}) but with the limiting contours sharing a common point; this means that the integration goes over some singularities of the integrand. However, we are still able to obtain (\ref{eq:MOM}) by finding a sequence of integrals which converge to (\ref{eq:mom}) for which we can apply the dimension reduction formula to complete the proof of Theorem \ref{thm:mom}.

The proof of Theorem \ref{thm:gauss} requires the asymptotics of higher cumulants, see the Appendix for a review of the definition and some facts about cumulants. Modulo a reduction step, the proof of Theorem \ref{thm:gauss} follows a similar line of argument as Theorem \ref{thm:mom}. Therefore some of the more repetitive points will be done in less detail.

\subsection{Proof of Theorem \ref{thm:mom}}

As outlined above, we want to show that (\ref{eq:premom}) converges to (\ref{eq:mom}). We suppress the dependence on $B$, $\alpha$ and $\ft$ in the indexing. Throughout the proof, constants $C$ are uniform in $\e$ and may vary from line to line. We let $\rho_<^\e(x), \rho_>^\e(x)$ be as in Proposition \ref{prop:dist<>}. For the proof, we assume that if $x = V_\ell$, then either $x > v_\ell$, $x = v_\ell$ or $x < v_\ell$ for all sufficiently small $\e > 0$. The purpose of this assumption is to ensure $\rho^\e_{<x}$ and $\rho^\e_{>x}$ both converge as in Proposition \ref{prop:dist<>}. Note that this condition is not restrictive, if we show (\ref{eq:premom}) converges to (\ref{eq:mom}) under each separate regime $x > v_\ell$, $x = v_\ell$, and $x < v_\ell$, then certain we have that (\ref{eq:premom}) converges to (\ref{eq:mom}) under a general limit $\e x \to \x$ without restriction on the inequality between $x$ and $v_\ell$, since the limit is independent of this ordering.

The first step is find contours such that the conditions in Theorem \ref{thm:obs} are satisfied. As we will see, the $k\ft$-separation condition gives us existence of such contours. In order for the conditions in Theorem \ref{thm:obs} to be met, we require the contours $\cC_1',\ldots,\cC_k'$ in (\ref{eq:premom}) to satisfy the following: $\cC_i'$ contains $[0,\rho_<^\e(x)]$ but does not intersect $[\rho_>^\e(x),\infty)$, and $\cC_j'$ encircles $t \cC_i'$ for any $i < j$. In particular, we require that $\cC_i'$ intersects $(\rho_<^\e(x), \rho_>^\e(x))$ at some point $a_i$, and these points must satisfy
\begin{align} \label{eq:kpt}
\rho_<^\e(x) < a_1 < ta_2 < \cdots < t^{k-1} a_k < t^{k-1} \rho_>^\e(x).
\end{align}

Let $0 < \theta_k,\ldots, \theta_1 < \pi$, and set $\cC_1'',\ldots,\cC_k''$ to be contours in $\C$ such that $\cC_i''$ is the contour consisting of line segments and circular arc
\[ \{1 + u e^{\pm \i \theta_j}: u \in [0,1]\}, \quad \{ z \in \C: |z| = |1 + e^{\i \theta_j}|, \arg z > \arg(1 + e^{\i \theta_j})\} \]
where the $\arg$ branches are in $(-\pi,\pi]$, positively oriented around $0$. Then $\cC_1'',\ldots,\cC_k''$ intersect pairwise at $1$ and $\cC_j''$ encircles $\cC_i'' \setminus \{1\}$ whenever $i < j$.

The $k\ft$-separation condition guarantees the existence of points $a_1:= a_1(\e),\ldots,a_k:= a_k(\e)$ satisfying (\ref{eq:kpt}). Indeed, by Proposition \ref{prop:dist<>}, if
\[ t^{-M} r^{-1} := \frac{\rho_>^\e(x)}{\rho_<^\e(x)} \]
then $M \ge k+1$. Indeed, if $\mathbf{x}$ is not a singular point then $\lim_{\e \to 0} \rho_<^\e(x) < \lim_{\e \to 0} \rho_>^\e(x)$ so that there is enough space to guarantee this inequality for $\e$ sufficiently small, and if $\mathbf{x}$ is a singular point then the $k\ft$-separation condition implies the right hand side is $\ge t^{-1} r^{-k\ft-1} =  t^{-(k+1)} r^{-1}$. In particular, we may set
\[ a_i = (t^{-M} r^{-1})^{\frac{i}{k+1}} \rho_<^\e(x) = (t^{-M} r^{-1})^{-\frac{k+1-i}{k+1}} \rho_>^\e(x). \]
Since $a_i/a_{i-1} > t^{-1}$, by setting $\cC_i' = a_i \cC_i''$, we obtain contours satisfying the conditions in Theorem \ref{thm:obs}.

If $\x$ is not a singular point, then $\lim_{\e \to 0} \rho_<^\e(x) < \lim_{\e \to 0} \rho_>^\e(x)$ so that $\lim_{\e \to 0} t^{-M} > 1$. Then the integrand in (\ref{eq:premom}) is bounded, having no singularities. Thus by Proposition \ref{prop:Gconv}, (\ref{eq:premom}) converges to (\ref{eq:mom}) with $\cC_i = (\lim_{\e \to 0} a_i) \cC_i''$.

In the case that $\x$ is a singular point, then $t^{-M} \to 1$, so there are singularities in (\ref{eq:premom}) to take care of. Let $F_\e(z_1,\ldots,z_k)$ denote the integrand in (\ref{eq:premom}). We may change variables in (\ref{eq:premom}) to replace the $\cC_i'$ contours with $\cC_i''$ via $z_i = a_i w_i$. By dominated convergence, to prove that (\ref{eq:premom}) converges to (\ref{eq:mom}), we seek a function $g$ so that
\[ \left| F_\e\left( a_1 w_1,\ldots, a_k w_k \right) \right| \le g(w_1,\ldots,w_k)\]
for $(w_1,\ldots,w_k) \in \cC_1''\times \cdots \times\cC_k''$ with $g$ integrable on $\cC_1''\times\cdots\times\cC_k''$ with respect to $d|w_1| \cdots d|w_k|$.

We may write
\begin{align*}
G_{<x}(a_i w_i;\e,\ft) &= G_{<x}( t^{-1} \cdot t (t^{-M} r^{-1})^{\frac{i}{k+1}} w_i \cdot \rho_<^\e(x)) \\
G_{>x}(a_i w_i;\e \ft) &= G_{>x}( t \cdot t^{-1} (t^{-M} r^{-1})^{\frac{k+1-i}{k+1}} w_i \cdot \rho_>^\e(x))
\end{align*}
If we denote by $\cC_i''^{-1} = \{w: w^{-1} \in \cC_i''\}$, then there exists $\theta \in (0,\pi)$ such that $\cC_i'', \cC_i''^{-1} \subset D^{0,\theta,\delta}$ for $1 \le i \le k$ and some fixed $\delta > 0$. This is clear for $\cC_i''$. For $\cC_i''^{-1}$, this follows from the map $w \mapsto w^{-1}$ being conformal and an involution taking $(-\infty,0)$ onto intself and $(0,1)$ to $(1,\infty)$. This implies that given $A > 0$, then $e^{-\e \fa} \cC_i''^{-1}, e^{-\e \fa} \cC_i'' \subset D^{c\e, \theta,\delta}$ for $1 \le i \le k$, $\fa \ge A$, $\delta > 0$ fixed, and $c$ depending on $A$. In particular, since
\[ t (t^{-M} r^{-1})^{\frac{i}{k+1}} > r^{-\frac{1}{k+1}}, \]
we have
\begin{align*}
\left( t \frac{a_i}{\rho_<^\e(x)} w_i \right)^{-1} &= (t(t^{-M} r^{-1})^{\frac{i}{k+1}} w_i )^{-1} \in D^{c\e,\theta,\delta} \\
t^{-1} \frac{a_i}{\rho_>^\e(x)} w_i &= t^{-1} (t^{-M} r^{-1})^{-\frac{k+1-i}{k+1}} w_i \in D^{c\e,\theta,\delta}
\end{align*}
for $1 \le i \le k$ and some fixed $\delta, c > 0$, $\theta \in (0,\pi)$. Thus by Proposition \ref{prop:Gconv}, we have
\begin{align} \label{eq:contour_conv}
\begin{split}
G_{<x}(a_i w_i; \e, \ft) &= \cG_{<\x}(t(t^{-M}r^{-1})^{\frac{i}{k+1}} w_i \cdot \rho_<(\x))^\ft \exp\left( O\left( \frac{\e}{|1 - \frac{t a_i w_i}{\rho_<^\e(x)}|} \right) \right), \\
G_{>x}(a_i w_i; \e, \ft) &= \cG_{>\x}(t^{-1} (t^{-M} r^{-1})^{\frac{k+1-i}{k+1}} w_i \cdot \rho_>(\x))^\ft \exp\left( O\left( \frac{\e}{|1 - \frac{ta_i w_i}{\rho_>^\e(x)}|} \right) \right)
\end{split}
\end{align}
for $w_i \in \cC_i'', 1 \le i \le k$; note that we dropped the $|z| \vee |z|^2$ term since $\cC_i''$ is bounded and separated from $0$. Observe that away from $[0,1]$,
\[ \cG_{<\x}(z)^p = R(z) (1 - z) \]
where $R(z)$ is a rational functions with poles $< 1$, we have a similar statement for $\cG_{<\x}(z)$ but we will only need the fact that it is bounded (we could alternatively flip the roles of $\cG_{<\x}$ and $\cG_{>\x}$ here). Using this and the fact that we may replace the exponential term in (\ref{eq:contour_conv}) with a crude constant order term, we have
\begin{align} \label{eq:Gbd}
|G_{<x}(a_i w_i;\e,\ft) \cdot G_{>x}(a_i w_i; \e,\ft)| \le C |1 - t^\delta w_i|^c
\end{align}
for $w_i \in \cC_i''$ and some fixed constant $c > 0$, $\delta > 0$. Next, we need the following lemma.

\begin{lemma} \label{lem:raybd}
Fix $\nu > 0$. For $\e > 0$ arbitrarily small, there exists a constant $C$ independent of $\e$ such that
\begin{align*}
\frac{|1 - e^{-\nu \e}z|}{|w - e^{-\e}z|} \le C
\end{align*}
on $(z,w) \in \cC_i'' \times \cC_j''$ whenever $1 \le i < j \le k$.
\end{lemma}
\begin{proof}
It suffices to check the inequality for $(z,w)$ near $(1,1)$ where $\cC_i'' \times \cC_j''$ looks like
\[ \{1 + u e^{\pm \i \theta_i}: u \in [0,\delta]\} \cup \{1 + u e^{\pm \i \theta_j}: u \in [0,\delta]\} \]
for some small $\delta > 0$. Since $\frac{1 - e^{-\nu \e} z}{1 - e^{-\e}z}$ is bounded over this region, independent of $\e$, it suffices to bound
\[ \left| \frac{1 - e^{-\e} z}{w - e^{-\e}z} \right| = \left|1 + \frac{1 - w}{(w - 1) + (1 - e^{-\e}z)} \right| = 1 + \left| \frac{w - 1}{(w - 1) + (1 - e^{-\e} z)} \right| = 1 + \left| \frac{1}{1 + \frac{1 - e^{-\e}z}{w - 1}} \right|. \]
Without loss of generality, we may suppose $w = 1 + u e^{\i \theta_j}$ for $u \in [0,\delta]$. Also note that if $z\in \HH$ then $|w - e^{-\e}z| < |w - e^{-\e}\bar{z}|$, so we may suppose that $z \in \HH$. Then $1 - e^{-\e} z = u' e^{\i \theta}$ for some $0 < \theta < \theta_i < \theta_j$ and $u' \in [0,\delta]$. Thus
\[ \left| \frac{1}{1 + \frac{1 - e^{-\e}z}{w - 1}} \right| = \left| \frac{1}{1 + ve^{\i (\theta - \theta_j)}} \right| \]
for some $v \in [0,\infty)$. We can maximize the right hand side over $v\in [0,\infty)$ for each fixed $\theta < \theta_j$; the maximizing value is given by $\frac{1}{\sin(\theta - \theta_j)}$. Thus
\[ \left| \frac{1 - e^{-\e}z}{w - e^{-\e}z} \right| \le 1 + \frac{1}{\sin(\theta_i - \theta_j)} \]
which proves the lemma.
\end{proof}

Using (\ref{eq:Gbd}) and the fact that $\left| \frac{(z_j - z_i)(z_j - \frac{q}{t} z_i)}{(z_j - \frac{1}{t} z_i)(z_j - q z_i)} \right|$ is bounded on $\cC_1' \times \cdots \times \cC_k'$ uniformly in $\e$, we can dominate $F_\e(z_1,\ldots,z_k)$ as follows

\begin{align} \label{eq:preFineq}
|F_\e(z_1,\ldots, z_k)| & \le C \cdot \left| \frac{1}{(z_2 - \frac{q}{t} z_1) \cdots (z_k - \frac{q}{t} z_{k-1})} \right| \prod_{i=1}^k |1 - t^\delta a_i^{-1} z_i|^c \\ \label{eq:Fineq}
& \le C \left| \frac{1}{(z_2 - \frac{q}{t} z_1) \cdots (z_k - \frac{q}{t} z_{k-1})} \right|^{1 - c}
\end{align}
on $\cC_1' \times \cdots \times \cC_k'$, where the second inequality follows from Lemma \ref{lem:raybd}. With $z_i = a_i w_i$, (\ref{eq:Fineq}) becomes
\begin{align} \label{eq:zwithw}
\begin{multlined}
|F_\e(z_1,\ldots,z_k)| \le C \left| \frac{1}{(a_2 w_2 - \frac{q}{t} a_1 w_1) \cdots (a_k w_k - \frac{q}{t} a_{k-1} w_{k-1})} \right|^{1 - c} \\
\le C \left| \frac{1}{(w_2 - w_1) \cdots (w_k - w_{k-1})} \right|^{1 - c} =: g(w_1,\ldots,w_k)
\end{multlined}
\end{align}
for $(w_1,\ldots,w_k) \in \cC_1'' \times \cdots \times \cC_k''$. The dominating function $g$ is integrable, since all its singularities are integrable; e.g. integrate $w_1$, then $w_2$, etc. We may thus apply dominated convergence and see that (\ref{eq:premom}) converges to (\ref{eq:mom}) up to a change of variables where we need to take $\cC_i = \rho_<(V_\ell) \cC_i''$.

Now that we have shown (\ref{eq:premom}) converges to (\ref{eq:mom}), we want to show that (\ref{eq:mom}) is the right hand side of (\ref{eq:MOM}). By a similar argument, we have the integral
\begin{align} \label{eq:altint}
\frac{1}{(2\pi\i)^k} \oint_{\cC_1'} \cdots \oint_{\cC_k'} \frac{\sum_{i=1}^k \frac{1}{z_i}}{(z_2 - z_1) \cdots (z_k - z_{k-1})} \prod_{i=1}^k G_{<x}(z_i;\e,\ft) \cdot G_{>x}(z_i;\e,\ft) \, dz_i,
\end{align}
with $\cC_i'$ and $G_x$ depending on $\e$, converges to (\ref{eq:mom}) as well. On the other hand, Theorem \ref{thm:dimred} says that (\ref{eq:altint}) may be reexpressed as
\[ \frac{1}{2\pi\i} \oint_{\cC'} \! [G_{<x}(z;\e,\ft) \cdot G_{>x}(z;\e,\ft)]^k \frac{dz}{z} \]
where $\cC'$ is some contour containing the poles of $G^B_{<x}(z;\e,\ft)$ and $0$ and containing no pole of $G^B_{>x}(z;\e,\ft)$; for example it suffices to pick $\cC_i'$ for some $i \in [[1,k]]$. This converges to the right hand side of (\ref{eq:MOM}). Therefore, (\ref{eq:mom}) coincides with the right hand side of (\ref{eq:MOM}), completing the proof of Theorem \ref{thm:mom}. \qed

\subsection{Proof of Theorem \ref{thm:gauss}}
We begin by defining
\begin{align} \label{eq:defD}
D^\e_k(x) = \frac{1}{\e} (\wp_k(\pi^x;q,t) - \E \wp_k(\pi^x;q,t)).
\end{align}
As alluded to in proof outline, the idea of the proof is to compute the cumulants of $(D^\e_{k_1}(x_1),\ldots,D^\e_{k_m}(x_m))$ and check that they are asymptotically Gaussian as $\e \to 0$; that is the order $\ge 3$ cumulants vanish and the order $2$ cumulants (i.e. the covariances) have the structure asserted by (\ref{eq:thmcov}). The asymptotics of the cumulants have many elements similar to the asymptotics from the proof of Theorem \ref{thm:mom}. However, the higher order cumulants require greater separation in order for the formulas from Theorem \ref{thm:obs} to apply which poses a problem yet again for the singular points. To ameliorate this, we prove Theorem \ref{thm:gauss} by first reducing to a seemingly weaker claim.

\begin{claim}
Fix integers $\nu \ge 2$, $k_1,\ldots,k_\nu \in \Z_{> 0}$ and let $x_1 \le \cdots \le x_\nu$ be points in $I^\e$ (depending on $\e$) such that
\begin{align} \label{eq:clmcond}
\begin{multlined}
\e x_i \to \mathbf{x}_i \in I~~\mbox{for $i \in [[1,\nu]]$, and} \\
x_i ~~\mbox{is $(k_i + \cdots + k_\nu)\ft$-separated from singular points, for $i \in [[1,\nu]]$.}
\end{multlined}
\end{align}
Then
\begin{align}
\kappa(D^\e_{k_1}(x_1), \ldots, D^\e_{k_\nu}(x_\nu)) \to \left\{ \begin{array}{cl}
\kappa(\fD_{k_1}(x_1),\fD_{k_2}(x_2)) & \mbox{if $\nu = 2$,} \\
0 & \mbox{if $\nu > 2$}.
\end{array} \right.
\end{align}
\end{claim}

\begin{lemma}
The Claim above implies Theorem \ref{thm:gauss}
\end{lemma}
\begin{proof}
We first show that for any integer $k > 0$ and any $\vec{x} \in I$, we have 
\begin{align} \label{eq:L2van}
\begin{multlined}
\E(D^\e_k(x) - D^\e_k(\tilde{x}))^2 \to 0 \\
\mbox{for any $x,\tilde{x} \in I^\e$ that are $(2k)\ft$-separated from singular points with $\e x, \e \tilde{x} \to \vec{x}$.}
\end{multlined}
\end{align}
To see (\ref{eq:L2van}), rewrite
\begin{align} \label{eq:L2cum}
\begin{multlined}
\E(D^\e_k(x) - D^\e_k(\tilde{x}))^2 = \E (D^\e_k(x))^2 - 2\E D^\e_k(x) D^\e_k(\tilde{x}) + \E (D^\e_k(\tilde{x}))^2 \\
= \kappa(D^\e_k(x),D^\e_k(x)) - 2 \kappa(D^\e_k(x),D^\e_k(\tilde{x})) + \kappa(D^\e_k(\tilde{x}),D^\e_k(\tilde{x})).
\end{multlined}
\end{align}
By the Claim, each cumulant appearing in (\ref{eq:L2cum}) converges as $\e \to 0$ to the same limit:
\[ \kappa(D^\e_k(x),D^\e_k(x)),~ \kappa(D^\e_k(x),D^\e_k(\tilde{x})),~ \kappa(D^\e_k(\tilde{x}),D^\e_k(\tilde{x})) \to  \kappa(\fD_{k_1}(\x),\fD_{k_2}(\x)). \]
This implies (\ref{eq:L2van}).

Let $x_1,\ldots,x_m \in I^\e$ and integers $k_1,\ldots,k_m \in \Z_{> 0}$ be as in the statement of Theorem \ref{thm:gauss}. For any $\chi > 0$, we can find $\tilde{x}_1,\ldots,\tilde{x}_m \in I^\e$ such that
\begin{align} \label{eq:xtilde}
\begin{multlined}
\e \tilde{x}_i \to \mathbf{x}_i \in I~~\mbox{for $i \in [[1,\nu]]$, and} \\
\tilde{x}_i ~~\mbox{is $\chi \ft$-separated from singular points, for $i \in [[1,\nu]]$.}
\end{multlined}
\end{align}
Fix an arbitrarily large integer $\eta > 0$. By the Claim, we can choose $\chi$ large enough so that
\begin{align} \label{eq:xtildecum}
\kappa(D^\e_{k_{i_1}}(\tilde{x}_{i_1}),\ldots,D^\e_{k_{i_\nu}}(\tilde{x}_{i_\nu})) \to \left\{ \begin{array}{cl}
\kappa(\fD_{k_1}(x_{i_1}),\fD_{k_2}(x_{i_2})) & \mbox{if $\nu = 2$,} \\
0 & \mbox{if $\nu > 2$}.
\end{array} \right.
\end{align}
for any $i_1,\ldots,i_\nu \in [[1,m]]$ and $\nu \in [[1,2\eta]]$. By (\ref{eq:L2van}), we have
\[ \E( D^\e_{k_i}(x_i) - D^\e_{k_i}(\tilde{x}_i))^2 \to 0 \]
for each $i \in [[1,m]]$. Thus (\ref{eq:xtildecum}) and Lemma \ref{lem:replace} from the Appendix imply that
\begin{align} \label{eq:xcum}
\kappa(D^\e_{k_{i_1}}(x_{i_1}),\ldots,D^\e_{k_{i_\nu}}(x_{i_\nu})) \to \left\{ \begin{array}{cl}
\kappa(\fD_{k_1}(x_{i_1}),\fD_{k_2}(x_{i_2})) & \mbox{if $\nu = 2$,} \\
0 & \mbox{if $\nu > 2$}.
\end{array} \right.
\end{align}
for any $i_1,\ldots,i_\nu \in [[1,m]]$ and $\nu \in [[1,\eta]]$. Since $\eta > 0$ was arbitrary, (\ref{eq:xcum}) holds for any integer $\nu \ge 2$. By Lemma \ref{lem:gausscum}, this implies that $(D^\e_{k_1}(x_1),\ldots,D^\e_{k_m}(x_m))$ converges in distribution to the Gaussian vector $(\fD_{k_1}(x_1),\ldots,\fD_{k_m}(x_m))$.
\end{proof}

We are left to prove the Claim. Adding a constant vector to a random vector adds only a constant order term to the logarithm of the characteristic function. By Definitions \ref{def:cum1} and \ref{def:cum2}, we have for $\nu > 1$
\begin{align}\label{eq:cum}
\begin{multlined}
\kappa( D_{k_1}^\e(x_1),\ldots, D_{k_\nu}^\e(x_\nu)) = \e^{-\nu} \kappa\left( \wp_{k_1}(\pi^{x_1};q,t),\ldots, \wp_{k_\nu}(\pi^{x_\nu};q,t)\right)\\
= \sum_{\substack{d \in \Z_{> 0} \\ \{ U_1,\ldots,U_d\} \in \Theta_\nu}} \e^{-\nu} (-1)^{d-1} (d-1)! \prod_{\ell=1}^d \E \left[ \prod_{i \in U_\ell} \wp_{k_i}(\pi^{x_i};q,t) \right]
\end{multlined}
\end{align}
where we use the notation from the Appendix with $\Theta_\nu$ the collection of all set partitions of $[[1,\nu]]$.

If $x_1 \le \cdots \le x_\nu$, then given that the conditions of Theorem \ref{thm:obs} are met, (\ref{eq:cum}) can be expressed as
\begin{align}\label{eq:intcum} 
\e^{-\nu}\oint \cdots \oint \fC(Z_1,\ldots,Z_\nu) \prod_{i=1}^\nu G_{x_i}(Z_i) \,DZ_i
\end{align}
where $Z_i = (z_{i,1},\ldots,z_{i,k_i})$,
\begin{align} \label{eq:fC}
\fC(Z_1,\ldots,Z_\nu) = \sum_{\substack{d \in \Z_{> 0} \\ \{U_1,\ldots,U_d\} \in \Theta_\nu}} (-1)^{d-1} (d-1)! \prod_{\ell=1}^d \prod_{\substack{(i,j) \in U_\ell \\ i < j}} \fC(Z_i,Z_j).
\end{align}
Let $\cC_{i,j}'$ denote the $z_{i,j}$-contour in (\ref{eq:intcum}).

As before, the separation condition (\ref{eq:clmcond}) ensures that the conditions of Theorem \ref{thm:obs} are met. We verify that this is the case. We require the existence of contours $\cC_{i,j}'$, $j \in [[1,k_i]]$ and $i \in [[1,\nu]]$, satisfying the following: $\cC_{i,j}'$ contains $[0, \rho_<^\e(x_i)]$ but no elements of $[\rho_>^\e,\infty)$, and $t \cC_{i',j'}'$ encircles $\cC_{i,j}'$ whenever $(i',j') > (i,j)$ in lexicographical order. In particular, we require $\cC_{i,j}'$ intersects $(\rho_<^\e(x_i),\rho_>^\e(x_i))$ at some point $a_{i,j}$, and these points must satisfy 
\begin{align} \label{eq:multikpt}
a_{i,j} < t a_{i',j'}, \quad \quad (i,j) < (i',j')
\end{align}
in lexicographical order. We can construct such contours as follows.

Let $0 < \theta_{i,j} < \pi$ such that $\theta_{i',j'} < \theta_{i,j}$ whenever $(i,j) < (i',j')$ in lexicographical order and set $\cC_{i,j}$ to be the contour in $\C$ consisting of line segments and a circular arc
\[ \{1 + u e^{\pm \i \theta_{i,j}}: u \in [0,1]\}, \quad \{z \in \C: |z| = | 1 + e^{\i \theta_{i,j}} |, \arg z > \arg(1 + e^{\i \theta_{i,j}}) \} \]
where the $\arg$ brsnches are in $(-\pi,\pi]$, positively oriented around $0$. Then $\cC_{i,j}''$, $1 \le j \le k_i, 1 \le i \le \nu$, intersect pairwise at $1$ and $\cC_{i',j'}''$ encircles $\cC_{i,j}'' \setminus \{1\}$ whenver $(i,j) < (i',j')$ in lexicographical order.

From $\rho_<^\e(x_i) < \rho_>^\e(x_i)$, we have
\[ \lim_{\e \to 0} \rho_<^\e(x_i) \le \lim_{\e \to 0} \rho_>^\e(x_i). \]
By monotonicity of $\rho_<^\e$ and $\rho_>^\e$, we also have that the left hand side and right hand side grow as $i$ increases; recall that $x_1 \le \cdots \le x_\nu$. By definition of singular points and Proposition \ref{prop:dist<>}, we have equality if and only if $\x_i$ is a singular point. Let $\Psi \subset [[1,\nu]]$ be defined so that $i \in \Psi$ if and only if $\x_i$ is singular. Then, we can find $a_{i,j} := a_{i,j}(\e)$, for each $(i,j), (i',j')$ where $j \in [[1,k_i]], j' \in [[1,k_{i'}]]$ and $i,i' \in [[1,\nu]] \setminus \Psi$, such that
\begin{align} \label{eq:aij1}
\begin{split}
& \lim_{\e \to 0} \rho_<^\e(x_i) < \lim_{\e \to 0} a_{i,j} < \lim_{\e \to 0} \rho_>^\e(x_i), \\
& \lim_{\e \to 0} a_{i,j} < \lim_{\e \to 0} a_{i',j'}, \quad \quad (i,j) < (i',j')
\end{split}
\end{align}
where we order lexicographically. For $i \in \Psi$ we define $a_{i,j}$ using the separation condition. Suppose we have,
\[ \x_{\eta-1} < \x_\eta = \cdots = \x_{\eta+\xi-1} < \x_{\eta+\xi} \]
with $\x_i = V_\ell$ a singular point, in particular $\eta,\ldots,\eta+\xi-1 \in \Psi$. By Proposition \ref{prop:dist<>} the separation condition implies that
\[ \frac{\rho_>^\e(x_i)}{\rho_<^\e(x_i)} \ge t^{-1} r^{-(k_i + \cdots + k_\nu)\ft-1} = t^{-(k_i + \cdots + k_\nu +1)} r^{-1}, \quad i = \eta,\ldots,\eta+\xi-1. \]
Thus, we may define $a_{i,j}$ for $i \in \Psi$ such that for any pair $a_{i,j}, a_{i',j'}$ with $j \in [[1,k_i]], j' \in [[1,k_{i'}]]$ and $i,i' \in [[\eta,\eta+\xi-1]]$, we have
\begin{align} \label{eq:aij2}
\frac{a_{i,j}}{\rho_<^\e(x_i)} > t^{-1-\delta}, \quad \frac{\rho_>^\e (x_i)}{a_{i,j}} > t^{-1-\delta}, \quad \frac{a_{i',j'}}{a_{i,j}} > t^{-1-\delta}, \quad (i,j) < (i',j')
\end{align}
for some fixed, small $\delta > 0$. Indeed, this follows from the separation condition and telescoping over $a_{i,j+1}/a_{i,j}$ with boundary cases
\[ a_{i,1}/\rho_<^\e(x_i), \quad \rho_>^\e(x_i)/a_{i,k_i}. \]
Although the separation condition is not optimal, it is nearly saturated in the case where $x_1 = \cdots = x_\nu$ all converge to the same singular point $V_\ell$ and are as close to $v_\ell$ as the separation condition allows.

From this choice of $a_{i,j}$, we construct the $z_{i,j}$-contour $\cC_{i,j}' = a_{i,j} \cC_{i,j}''$ similar to the proof of Theorem \ref{thm:mom}. By (\ref{eq:aij1}) and (\ref{eq:aij2}), the points $a_{i,j}$ satisfy (\ref{eq:multikpt}). Therefore the contours $\cC_{i,j}'$ meet the conditions of Theorem \ref{thm:obs}, and so we may express (\ref{eq:cum}) as (\ref{eq:intcum}).

Our next objective is to apply dominated convergence to (\ref{eq:intcum}). To this end, we rewrite (\ref{eq:intcum}) in a different form. Let $U \subset [[1,\nu]]$, $\cT(U)$ denote the set of undirected simple graphs with vertices labeled by $U$, and $\cL(U) \subset \cT(U)$ the subset of connected graphs. Given a graph $\Omega$, we denote by $E(\Omega)$ the edge set of $\Omega$. We show
\begin{align} \label{eq:Ccum}
\fC(Z_1,\ldots,Z_m) = \sum_{\Omega \in \cL([[1,\nu]])} \prod_{\substack{i < j \\ (i,j) \in E(\Omega)}} (C(Z_i,Z_j) - 1).
\end{align}
Define
\begin{align}
\cK(U) = \sum_{\Omega \in \cL([[1,\nu]])} \prod_{\substack{i < j \\ (i,j) \in E(\Omega)}} (C(Z_i,Z_j) - 1), ~~~ \cE(U) = \sum_{\Omega \in \cT([[1,\nu]])} \prod_{\substack{i < j \\ (i,j) \in E(\Omega)}} (C(Z_i,Z_j) - 1).
\end{align}
Then
\begin{align}
\cE(U) = \sum_{\substack{d > 0 \\ \{U_1,\ldots,U_d\} \in \Theta_U}} \prod_{\ell=1}^d \cK(U_\ell).
\end{align}
By Lemma \ref{lem:cuminv}, we have
\begin{align}
\cK(U) = \sum_{\substack{d > 0 \\ \{U_1,\ldots,U_d\} \in \Theta_U}} (-1)^{d-1} (d-1)! \prod_{\ell=1}^d \cE(U_\ell)
\end{align}
which agrees with the right hand side of (\ref{eq:fC}) when $U = [[1,\nu]]$. This proves (\ref{eq:Ccum}).

We also record that
\begin{align} \label{eq:Ceq}
C(Z_i,Z_j) - 1 = \sum_{\substack{\emptyset \neq S \subset [[1,k_i]] \\ \emptyset \neq T \subset [[1,k_j]]}} \prod_{(a,b) \in S \times T} \frac{(1 - q)(t^{-1} - 1) z_{i,a} z_{j,b}}{(z_{j,b} - q z_{i,a})(z_{j,b} - \frac{1}{t} z_{i,a})}
\end{align}

For each $i \in [[1,\nu]]$ and $S \subset [[1,k_i]]$, let $a_S$ be the minimal member of $S$ (we note this choice is arbitrary). For each edge $i,j \in [[1,\nu]]$ with $i < j$, and any pair of subsets $S \subset [[1,k_i]],T \subset [[1,k_j]]$, we have
\begin{align} \label{eq:Cbd}
\begin{split}
\left| \prod_{(a,b) \in S\times T} \frac{(1 - q)(t^{-1} - 1) z_{i,a} z_{j,b}}{(z_{j,b} - qz_{i,a})(z_{j,b} - \frac{1}{t} z_{i,a})} \right| & \le C \left| \frac{\e^2}{(z_{j,a_T} - qz_{i,a_S})(z_{j,a_T} - \frac{1}{t}z_{i,a_S})} \right|
\end{split}
\end{align}
for $z_{i,j} \in \cC_{i,j}'$. The inequality follows from removing all but one $(a_S,a_T) \in S\times T$ using the fact that $|z_{j,b} - qz_{i,a}|, |z_{j,b} - \frac{1}{t} z_{i,a}| \ge C \e$ by our choice of contours. For each $\Omega \in \cL([[1,\nu]])$, fix a complete subtree $\Omega'$ of $\Omega$. Then by (\ref{eq:Ccum}), (\ref{eq:Ceq}) and (\ref{eq:Cbd}), we have
\begin{align} \nonumber
|\fC(Z_1,\ldots,Z_\nu)| & \le C \sum_{\Omega \in \cL([[1,\nu]])} \prod_{\substack{i < j \\ (i,j) \in E(\Omega)}} \sum_{\substack{S \subset [[1,k_i]] \\ T \subset [[1,k_j]]}} \left| \frac{\e^2}{(z_{j,a_T} - qz_{i,a_S})(z_{j,a_T} - \frac{1}{t}z_{i,a_S})} \right| \\ \label{eq:Ccumbd}
& \le C \e^{2\nu-2} \sum_{\Omega \in \cL([[1,\nu]])} \prod_{\substack{i < j \\ (i,j) \in E(\Omega')}} \sum_{\substack{S \subset [[1,k_i]] \\ T \subset [[1,k_j]]}} \left| \frac{1}{(z_{j,a_T} - qz_{i,a_S})(z_{j,a_T} - \frac{1}{t}z_{i,a_S})} \right|
\end{align}
for $z_{i,j} \in \cC_{i,j}'$. In the last line, we used the fact that the number of edges in $\Omega'$ is $\nu - 1$ in pulling out the $\e$ factor. 

On the other hand, as with (\ref{eq:preFineq}) in the proof of Theorem \ref{thm:mom}, we have the bound
\begin{align} \label{eq:GZbd}
\left| G_{x_i}(Z_i) \frac{DZ_i}{dZ_i} \right| \le C \left|\frac{1}{(z_{i,2} - \frac{q}{t} z_{i,1}) \cdots (z_{i,k} - \frac{q}{t} z_{i,k-1})} \right| \prod_{j=1}^{k_i} |1 - t^\delta a_{ij}^{-1} z_{ij}|^{c'}
\end{align}
for some $c' > 0$ and $\delta > 0$ fixed along $z_{i,j} \in \cC_{i,j}'$ where we write $\frac{DZ_i}{dZ_i}$ to indicate that we remove the differentials from $DZ_i$. Then by (\ref{eq:Ccumbd}), (\ref{eq:GZbd}), we obtain the following bound for the integrand $F_\e(Z_1,\ldots,Z_\nu)$ of (\ref{eq:intcum}) as we had done with (\ref{eq:Fineq}) in the proof of Theorem \ref{thm:mom}:
\begin{align} \label{eq:Fbd}
\begin{multlined}
|F_\e(Z_1,\ldots,Z_\nu)| \le C \e^{\nu-2} \prod_{i=1}^\nu \left| \frac{1}{(z_{i,2} - \frac{q}{t}z_{i,1}) \cdots (z_{i,k_i} - \frac{q}{t}z_{i,k_i - 1})} \right|^{1 - c}  \\
\times \sum_{\Omega \in \cL([[1,\nu]])} \prod_{\substack{i < j \\ (i,j) \in E(\Omega')}} \sum_{\substack{S \subset [[1,k_i]] \\ T \subset [[1,k_j]]}} \left| \frac{1}{(z_{j,a_T} - qz_{i,a_S})(z_{j,a_T} - \frac{1}{t}z_{i,a_S})} \right|^{1-c} 
\end{multlined}
\end{align}
for some $c > 0$ along $z_{i,j} \in \cC_{i,j}'$ for each $j \in [[1,k_i]]$ and $i \in [[1,\nu]]$.

If we expand out the right hand side of (\ref{eq:Fbd}), we have that $|F_\e(Z_1,\ldots,Z_\nu)|$ is bounded on along $\cC_{i,j}'$ by a sum of finitely many terms of the form
\begin{align} \label{eq:sumbd}
\begin{multlined}
T_\e(Z_1,\ldots,Z_\nu) = C \e^{\nu-2} \prod_{i=1}^\nu \left| \frac{1}{(z_{i,2} - \frac{q}{t}z_{i,1}) \cdots (z_{i,k_i} - \frac{q}{t}z_{i,k_i - 1})} \right|^{1 - c}  \\
\times \prod_{\substack{i < j \\ (i,j) \in E(\Omega')}} \left| \frac{1}{(z_{j,b_{(i,j)}} - qz_{i,a_{(i,j)}})(z_{j,b_{(i,j)}} - \frac{1}{t}z_{i,a_{(i,j)}})}\right|^{1-c}
\end{multlined}
\end{align}
for some $\Omega \in \cL([[1,\nu]])$ and some $a_{(i,j)} \in [[1,k_i]]$, $b_{(i,j)} \in [[1,k_j]]$ for each $(i,j) \in E(\Omega')$. Since we are seeking an integrable function which dominates $F_\e(Z_1,\ldots,Z_\nu)$, it suffices to dominate $T_\e(Z_1,\ldots,Z_\nu)$ by an integrable function.

We may assume $c < 1$. Choose a distinguished element $(i_\Omega, j_\Omega) \in \Omega'$ where $i_\Omega < j_\Omega$. Then (\ref{eq:sumbd}) may be replaced by
\begin{align} \label{eq:Tbd1}
\begin{multlined}
T_\e(Z_1,\ldots,Z_\nu) \le C \e^{(\nu-2)c} \prod_{i=1}^\nu \left| \frac{1}{(z_{i,2} - \frac{q}{t} z_{i,1}) \cdots (z_{i,k_i} - \frac{q}{t}z_{i,k_i - 1})} \right|^{1 - c}  \\
\times \left|\frac{1}{(z_{j_\Omega,b_{(i_\Omega,j_\Omega)}} - qz_{i_\Omega,a_{(i_\Omega,j_\Omega)}})}\prod_{\substack{i < j \\ (i,j) \in E(\Omega')}} \frac{1}{(z_{j,b_{(i,j)}} - \frac{1}{t}z_{i,a_{(i,j)}})}\right|^{1-c}
\end{multlined}
\end{align}
where we used $\e^{(\nu-2)(1 - c)}$ to remove each $|z_{j,b_{(i,j)}} - qz_{i,a_{(i,j)}}|$ term except for the one corresponding to the distinguished edge $(i_\Omega,j_\Omega)$, recalling again that the number of edges of $\Omega'$ is $\nu - 1$. 

We set $z_{i,j} = \frac{a_{i,j}}{\lim_{\e to 0} a_{i,j}} w_{i,j}$ for the rest of the proof so that $w_{i,j}$ runs along $\cC_{i,j} := ( \lim_{\e to 0} a_{i,j} )\cC_{i,j}''$. Then, like (\ref{eq:zwithw}) in the proof of Theorem \ref{thm:mom}, we have
\begin{align} \label{eq:Tbd2}
\begin{multlined}
T(Z_1,\ldots,Z_\nu) \le C \e^{(\nu-2)c} \prod_{i=1}^\nu \left| \frac{1}{(w_{i,2} - w_{i,1}) \cdots (w_{i,k_i} - w_{i,k_i - 1})} \right|^{1 - c}  \\
\times \left|\frac{1}{(w_{j_\Omega,b_{(i_\Omega,j_\Omega)}} - w_{i_\Omega,a_{(i_\Omega,j_\Omega)}})}\prod_{\substack{i < j \\ (i,j) \in E(\Omega')}} \frac{1}{(w_{j,b_{(i,j)}} - w_{i,a_{(i,j)}})}\right|^{1-c} \\
=: \e^{(\nu-2)c} g(W_1,\ldots,W_\nu)
\end{multlined}
\end{align}
We show that $g(W_1,\ldots,W_\nu)$ is integrable on $w_{i,j} \in \cC_{i,j}$, $j \in [[1,k_i]]$, $i \in [[1,\nu]]$ with respect to the measure $\prod_{\substack{i \in [[1,\nu]] \\ j \in [[1,k_i]]}} d|w_{i,j}|$. We use the following lemma to this end.

\begin{lemma} \label{lem:singint}
Let $\Gamma$ be a graph with vertices labeled by some subset of $\{(i,j)\}_{\substack{i \in [[1,\nu]] \\ j \in [[1,k_i]]}}$. If $\Gamma$ is a tree and $c > 0$, then for any $(v_1',v_2') \in E(\Gamma)$ we have
\[ \oint \cdots \oint \frac{1}{|w_{v_1'} - w_{v_2'}|^{1 - c}} \prod_{(v_1,v_2) \in E(\Gamma)} \frac{1}{|w_{v_1} - w_{v_2}|^{1 - c}} \, \prod_{\substack{i \in[[1,\nu]] \\ j \in [[1,k_i]]}} |dw_{i,j}| < \infty\]
where $\cC_{i,j}$ is the $w_{i,j}$ contour.
\end{lemma}
\begin{proof}
Since the contours $\cC_{i,j}$ have finite length, we can integrate out the variables independent of the integrand, so it suffices to show
\begin{align} \label{eq:vertint}
\oint \cdots \oint \frac{1}{|w_{v_1'} - w_{v_2'}|^{1 - c}} \prod_{(v_1,v_2) \in E(\Gamma)} \frac{1}{|w_{v_1} - w_{v_2}|^{1 - c}} \, \prod_{v \in V(\Gamma)} |dw_v|
\end{align}
where $V(\Gamma)$ is the vertex set of $\Gamma$.

For the case where the number of edges of $\Gamma$ is $1$, we have
\[ \oint \oint \frac{1}{|w_{v_1'} - w_{v_2'}|^{2(1 - c)}} |dw_{v_1'}||dw_{v_2'}| \]
is integrable even if $\cC_{v_1}$ and $\cC_{v_2}$ meet at a point.

For general $|E(\Gamma)| \ge 1$, choose a leaf vertex $v_1$ of $\Gamma$ so that $v_1$ is not a vertex of the edge $(v_1',v_2')$. There is a unique edge $e_1 = (v_1,v_2) \in E(\Gamma)$ containing $v_1$. If we integrate over $w_{v_1}$, then the $w_{v_1}$-dependent part of the integrand is $|w_{v_1} - w_{v_2}|^{c - 1}$ which is integrable even if $\cC_{v_2}$ and $\cC_{v_1}$ intersect. Then repeat this procedure with the graph $\Gamma \setminus \{v_1\}$ which is still a tree. After repeating this procedure, we will eventually return to the case $|E(\Gamma)| = 1$.
\end{proof}

We check that $g(W_1,\ldots,W_\nu)$ meets the conditions of Lemma \ref{lem:singint}. Our graph $\Gamma$ is the union of graph $\Gamma_i$ for $i \in [[0,\nu]]$ defined as follows. Let $\Gamma_i$ be the graph with edge set $E(\Gamma_i) = \{((i,1),(i,2)),\ldots,((i,k_i-1),(i,k_i))\}$ for $i \in [[1,\nu]]$ and $E(\Gamma_0) = \{((i,a_{(i,j)}),(j,b_{(i,j)})):(i,j) \in \Omega'\}$. We must check that $\Gamma$ is a tree. Indeed, if we collapse each $\Gamma_i$ to a point $v_i$ for $i \in [[1,\nu]]$ (alternatively said, we project away the second coordinate for the vertex labels), then $\Gamma$ becomes the tree $\Omega'$. Since each $\Gamma_i$ for $i \in [[1,\nu]]$ is a tree, this implies that $\Gamma$ is a tree. Thus by Lemma \ref{lem:singint},
\begin{align}
\int \cdots \int g(W_1,\ldots,W_\nu) \prod_{\substack{i \in[[1,\nu]] \\ j \in [[1,k_i]]}} d|w_{i,j}| < +\infty
\end{align}
where the $w_{i,j}$-contour is $\cC_{i,j}$.

If $\nu > 2$, then since $\e^{(\nu-2)c} \to 0$, we have that (\ref{eq:Tbd2}) converges to $0$. Thus for $\nu > 2$, the cumulant (\ref{eq:cum}) converges to $0$. If $\nu = 2$, then (\ref{eq:Tbd2}) allows us to apply dominated convergence so that (\ref{eq:cum}) converges to
\begin{align} \label{eq:covlim}
\begin{multlined}
\frac{1}{(2\pi\i)^{k_1 + k_2}}\oint \cdots \oint \left( \sum_{\substack{a \in [[1,k_1]] \\ b \in [[1,k_2]]}} \frac{w_{1,a} w_{2,b}}{(w_{1,a} - w_{2,b})^2} \right) \\
\times\prod_{i=1}^2 \frac{ \sum_{j=1}^{k_i} \frac{1}{w_{i,j}}}{(w_{i,2} - w_{i,1}) \cdots (w_{i,k_i} - w_{i,k_i-1})}  \prod_{j=1}^{k_i}[\cG_{<\mathbf{x}_i}(w_{i,j}) \cG_{>\mathbf{x}_i}(w_{i,j})]^{\ft} dw_{i,j}
\end{multlined}
\end{align}
where $\cC_{i,j}$ is the $z_{i,j}$-contour. As in the proof of Theorem \ref{thm:mom}, we can consider the family of integrals replacing $\cG_{<\x_i}^\ft$ with $G_{<x_i}$, $\cG_{>\x_i}^\ft$ with $G_{>x_i}$ and $\cC_{i,j}$ with $\cC_{i,j}'$ in (\ref{eq:covlim}). By applying Theorem \ref{thm:dimred} twice to this family and taking the limit as $\e \to 0$, we have that (\ref{eq:covlim}) is equal to
\[ \frac{k_1k_2}{(2\pi\i)^2} \oint_{\cC_2} \oint_{\cC_1} \frac{1}{(z - w)^2} [\cG_{<\mathbf{x}_1}(z) \cG_{>\mathbf{x}_1}(z)]^{k_1\ft} [\cG_{<\mathbf{x}_2}(w) \cG_{>\mathbf{x}_2}(w)]^{k_2\ft} \, dz \, dw \]
where $\cC_1 = \cC_{1,i}$ for some $i \in [[1,k_1]]$ and $\cC_2 = \cC_{2,j}$ for some $j \in [[1,k_2]]$. This proves the Claim, and therefore proves Theorem \ref{thm:gauss}.

\section{Complex Structure of the Liquid Region} \label{sec:cpx}
Throughout this section, we fix a family $\P^{B,r,\vec{s}}_{\alpha,\ft}$ satisfying the Limit Conditions, as defined in the beginning of Section \ref{sec:asymp}, with limiting back wall $\cB: I \to \R$. The \emph{liquid region}, denoted by $\sL$, is the subset of $\R^2$ consisting of $(x,y)$ such that the local proportions of $\lloz,\rloz,\hloz$ lozenges at $(x,y)$ are all positive. The goal of this section is to give a natural complex structure on $\sL$ through $\cG_x$. This complex structure will be used in Section \ref{sec:limgff} to describe the GFF fluctuations of $\P^{B,r,\vec{s}}_{\alpha,\ft}$ in the limit as $\e \to 0$.

The key object of study is the function
\begin{align} \label{eq:Gx}
\cG_x^\cB(\zeta) = e^{-\cB(x)} \cG_{<x}^\cB(\zeta)\cdot \cG_{>x}^\cB(\zeta)
\end{align}
Throughout this section, we write $\cG_x := \cG^\cB_x$.

\begin{definition}
For $(x,y) \in \R^2$, we define the $(x,y)$-\emph{companion equation}
\begin{align} \label{eq:compeq}
\cG_x(\zeta) = e^{-y}.
\end{align}
\end{definition}

The importance of the companion equation is due to its connection to the liquid region.

\begin{definition}
The \textit{liquid region} denoted by $\sL$ is the set of $(x,y) \in \R^2$ for which the $(x,y)$-companion equation has a nonreal pair of roots.
\end{definition}

We provided a different definition of the liquid region earlier: the \emph{liquid region} is the set of $(x,y)$ such that all the local proportions $p_{\lloz}, p_{\rloz}, p_{\hloz}$ are positive. The equivalence of these definitions can be seen in Section \ref{sec:limgff}. For now, we use the definition of the liquid region in terms of the companion equation. The following theorem is the main result of this section.

\begin{theorem} \label{thm:homeo}
For each $(x,y) \in \sL$, there exists a unique root $\zeta(x,y)$ of the $(x,y)$-companion equation in the upper half plane $\HH$. The map $\zeta: \sL \to \HH$ is a diffeomorphism.
\end{theorem}

Theorem \ref{thm:homeo} endows $\sL$ with the complex structure of $\HH$ given by the pullback of $\zeta(x,y)$.

We prove Theorem \ref{thm:homeo} in two parts. In Section \ref{sec:exist}, we show that the map $\zeta(x,y)$ in Theorem \ref{thm:homeo} exists. In Section \ref{sec:homeo}, we prove that the map $\zeta(x,y)$ is a diffeomorphism. In Section \ref{sec:frzbdry}, we give a parametrization of the boundary of $\sL$ which we call the \emph{frozen boundary}; this involves studying the double roots of (\ref{eq:compeq}).

Before proceeding, we show a convenient expression for $\cG_x$ as a product of two functions where one is independent of $x$ and the other depends on $x$. We also introduce some useful definitions. Define
\begin{align} \label{eq:Px}
P_x(z) &= \prod_{\sigma \in \cS} (1 - e^{-x} \sigma z)^{\frac{1}{p}} \\ \label{eq:Q}
Q(z) &= e^{\cB(V_0)} \prod_{\substack{\ell \in [[0,n]] \\ \sigma \in \cS}} (1 - e^{-V_\ell} \sigma z)^{\frac{1}{p}(\1[\cB'(V_\ell^+) \ge \varsigma_\sigma] - \1[\cB'(V_\ell^-) \ge \varsigma_\sigma])} \\ \label{eq:1-s}
&= e^{\cB(V_0)} \prod_{\sigma \in \cS} \left[ (1 - e^{-V_0} \sigma z)^{\frac{1}{p}} \prod_{\ell=1}^n \left( \frac{1 - e^{-V_\ell} \sigma z}{1 - e^{-V_{\ell-1}} \sigma z} \right)^{\frac{1}{p} \1[ \cB'(V_{\ell-1},V_\ell) < \varsigma_\sigma]} \right] \\ \label{eq:1+s}
&= e^{\cB(V_0)} \prod_{\sigma \in \cS} \left[ (1 - e^{-V_n} \sigma z)^{\frac{1}{p}} \prod_{\ell=1}^n \left( \frac{1 - e^{-V_{\ell-1}} \sigma z}{1 - e^{-V_\ell} \sigma z} \right)^{\frac{1}{p} \1[ \cB'(V_{\ell-1},V_\ell) \ge \varsigma_\sigma]} \right]
\end{align}
where the branches of $(1 - vz)^s$ above are taken so that $z < v^{-1}$ give positive real values. If $V_0 = -\infty$, then the expression above is to be interpreted as the limit of the expression as $V_0 \to -\infty$.

\begin{lemma}
For $x\in \R$, we have
\[ \cG_x(z) = \frac{P_x(z)}{Q(z)}. \]
\end{lemma}
\begin{proof}
It is sufficient to prove this lemma for $V_0 >-\infty$, since we may then take $V_0 \to -\infty$. Suppose $a \in [[0,p-1]]$ is maximal so that $V_{a-1} < x$, then (\ref{eq:G<>}) implies
\begin{align} \label{eq:Gxpf}
e^{-\cB(x)} \cG_{<x}(z) = e^{-\cB(x)} \prod_{\sigma \in \cS} \prod_{\ell = 1}^a \left( \frac{1 - e^{\min(V_\ell,x)} (\sigma z)^{-1}}{1 - e^{V_{\ell-1}} (\sigma z)^{-1}} \right)^{\frac{1}{p} \1[\cB'(V_\ell^-) \ge \varsigma_\sigma]}.
\end{align}
Observe that if $x_1,x_2 \in [V_{\ell-1},V_\ell]$, then
\[ \cB(x_1) - \cB(x_2) = (x_1-x_2) \cB'(V_{\ell-1},V_\ell). \]
This implies
\begin{align} \label{eq:expbprop}
e^{-\cB(x)} = e^{-\cB'(V_a^-)(x-V_{a-1})} e^{- \cB(V_{a-1})}.
\end{align}
By Lemma \ref{lem:sigma_order} and the definition of $\varsigma_\sigma$, we have $\cB'(V_\ell^-) = \sum_{\sigma \in \cS} \frac{1}{p} \1[\cB'(V_\ell^-) \ge \varsigma_\sigma]$. Thus the first exponential term on the right hand side of (\ref{eq:expbprop}) can be used to turn
\[ \frac{1 - e^x (\sigma z)^{-1}}{1 - e^{V_{a-1}} (\sigma z)^{-1}} \quad \mbox{into} \quad \frac{1 - e^{-x} \sigma z}{1 - e^{-V_{a-1}} \sigma z} \]
for each $\sigma \in \cS$ in (\ref{eq:Gxpf}). The second factor in (\ref{eq:expbprop}) can be used to iterate this procedure where $V_{a-1}$ and $V_{a-2}$ play the role of $x$ and $V_{a-1}$ respectively; then with $V_{a-2}$ and $V_{a-3}$, etc. At the end of this procedure, we obtain
\begin{align*}
e^{-\cB(x)} \cG_{<x}(z) &= e^{-\cB(V_0)} \prod_{\sigma \in \cS} \prod_{\ell = 1}^a \left( \frac{1 - e^{-\min(V_\ell,x)} \sigma z}{1 - e^{-V_{\ell-1}} \sigma z} \right)^{\frac{1}{p} \1[\cB'(V_\ell^-) \ge \varsigma_\sigma]} \\
&= e^{-\cB(V_0)} \prod_{\sigma \in \cS} \left[ (1 - e^{-x} \sigma z)^{\frac{1}{p} \1[\cB'(V_a^-) \ge \varsigma_\sigma]} \prod_{\ell=0}^{a-1} (1 - e^{-V_a} \sigma z)^{\frac{1}{p} (\1[\cB'(V_a^-) \ge \varsigma_\sigma] -  \1[\cB'(V_a^+) \ge \varsigma_\sigma])} \right]
\end{align*}
where the second equality uses the fact that $\cB'(V_{\ell+1}^-) = \cB'(V_\ell^+)$. From (\ref{eq:G<>}), we may also write
\[ \cG_{>x}(z) = \prod_{\sigma \in \cS} \left[ (1 - e^{-x} \sigma z)^{\frac{1}{p} \1[\cB'(V_a^-) < \varsigma_\sigma]} \prod_{\ell=a}^n (1 - e^{-V_a} \sigma z)^{\frac{1}{p} (\1[\cB'(V_a^+) < \varsigma_\sigma] - \1[\cB'(V_a^-) < \varsigma_\sigma])} \right]. \]
From (\ref{eq:Gx}) and the fact that
\[ \1[ \cB'(V_a^+) < \varsigma_\sigma] - \1[ \cB'(V_a^-) < \varsigma_\sigma] = \1[ \cB'(V_a^-) \ge \varsigma_\sigma] - \1[ \cB'(V_a^+) \ge \varsigma_\sigma], \]
we obtain
\[ \cG_x(z) = e^{-\cB(V_0)} \prod_{\sigma \in \cS} \left[ (1 - e^{-x} \sigma z)^{\frac{1}{p}} \prod_{\ell=0}^n (1 - e^{-V_a} \sigma z)^{\frac{1}{p} \1[ \cB'(V_a^-) \ge \varsigma_\sigma] - \1[ \cB'(V_a^+) \ge \varsigma_\sigma] } \right] = \frac{P_x(z)}{Q(z)}. \]
\end{proof}

\begin{definition}
Suppose $f:U \to \hat{\C}$ is such that $f^p$ is meromorphic. We say $z_0$ is a \textit{branch pole} (\textit{branch zero}) 
of $f$ if $z_0$ is a pole (zero) of $f^p$. The \emph{order} of a branch pole (zero) of $f$ at $z_0$ is defined to be the order of the pole (zero) of $f^p$ at $z_0$ divided by $p$.
\end{definition}

\begin{definition}
Suppose $f$ is a product of terms of the form $a - z$ or $1 - a^{-1} z$ for some $a \in \R$. Define
\[ \arg_+ f(u) = \lim_{\substack{z \to u \\ \arg z > 0}} \arg f(z)\]
for $u \in \R$ away from a branch point of $f$, where we take $\arg$ so that $\arg(a - z), \arg(1 - a^{-1} z) = 0$ for $z < a$.
\end{definition}

\begin{definition}
Let $V$ be an non-differentiable point of $\cB$. Define
\[ J_V = \{\sigma^{-1} e^V: \1[ \cB'(V^+) \ge \varsigma_\sigma] \ne \1[\cB(V^-) \ge \varsigma_\sigma] \} \]
and for $x \in I$ define $J_x = \{ \sigma^{-1} e^x: \sigma \in \cS\}$. If $V_0 = -\infty$ (or $V_n = +\infty$), take $J_{V_0} = \{0\}$ ($J_{V_n} = \{\infty\}$). Let $\cJ = \bigcup_{\ell=0}^n J_V$ and $\cJ_x = J_x \cup \cJ$. By (\ref{eq:Q}), (\ref{eq:Px}), and (\ref{eq:Gx}), $\cJ_x$ contains the set of branch poles and zeros of $\cG_x^{\cB}$.
\end{definition}

By the definition of $J_V$ and the fact that \Cref{prop:BD(S)} classifies the situation where $\cB'(V^+) - \cB'(V^-) < 0$, we have the following lemma.

\begin{lemma} \label{lem:Vchange}
Let $\P^{B,r,\vec{s}}_{\alpha,\ft}$ satisfy the Limit Conditions. Let $V$ be a non-differentiable point of $\cB$.
\begin{itemize}
    \item If $\cB'(V^+) - \cB'(V^-) > 0$, then $J_V = \{\sigma_i^{-1} e^V: \varsigma_i \in (\cB'(V^-), \cB'(V^+)], 1 \le i \le d\}$.
    \item If $\cB'(V^+) - \cB'(V^-) < 0$, then $J_V = \{\sigma_i^{-1} e^V\}$ for $1 \le i \le d$ such that $\cB'(V^-) = \varsigma_i$.
\end{itemize}
\end{lemma}

As a consequence of Lemma \ref{lem:Vchange} and (\ref{lc4}), we have $J_V < J_W$ whenever $V < W$.

\subsection{Existence of \texorpdfstring{$\zeta(x,y)$}{zeta(x,y)}} \label{sec:exist}

It will be convenient to consider the phase and magnitude equations for (\ref{eq:compeq}):
\begin{align} \label{eq:magneq}
e^{-y} &= |\cG_x(\zeta)| \\ \label{eq:phaseq}
0 &= \arg_+ \cG_x(\zeta).
\end{align}
From (\ref{eq:magneq}) and (\ref{eq:phaseq}), we see that the zeros of (\ref{eq:compeq}) approach branch poles and branch zeros of $\cG_x$ as $y \to -\infty$ and $y \to +\infty$ respectively. The idea is to use this condition to restrict the number of nonreal roots to a single pair.

The following definitions will be useful.
\begin{definition}
Denote by $\hatR := \R \cup \{\infty\}$ the one-point compactification of $\R$. By an \textit{interval} in $\hatR$ we mean either an interval in $\R$ or a set of the form $(a,+\infty)\cup\{\infty\} \cup (-\infty,b)$. For $a > b$, we denote the $\hatR$-interval $(a,+\infty) \cup \{\infty\}\cup(-\infty,b)$ by $(a,b)$. The set $\hatR \setminus \cJ$ is a union of $|\cJ|$ disjoint intervals.

Given a set of points $A$ in $\hatR$, we say that an \emph{external (internal) component} of $\hatR \setminus A$ is a connected component of $\hatR \setminus A$ of the form $(a,b)$ for some $a > b$ ($a < b$).
\end{definition}

The following basic facts are immediate consequences of the definition (\ref{eq:Gx}) of $\cG_x$.

Recall that the branch poles and branch zeros of $\cG_x$ are contained in $\cJ_x$. When $J_x \cap \cJ$ is nonempty, there may be cancellations of poles and zeros, thus this containment may be strict for some values of $x$. Note that if $a \in J_x \setminus \cJ$, then $a$ is a branch zero of $\cG_x$. Also, note that any branch pole $a$ of $\cG_x$ must be a point in $J_V$ for some non-differentiable point $V$ of $\cB$ such that $a = \sigma^{-1} e^V$ and $\cB'(V^+) - \cB'(V^-) > 0$.

The following lemma gives a restriction on branch poles.

\begin{lemma} \label{lem:endpt}
Fix $x \in \R$. Suppose $L$ is a connected component of $\hatR \setminus \cJ_x$ such that $\arg_+ \cG_x(L) = 0$. If $L$ is the external component and $x \in (V_0,V_n)$, then both endpoints of $L$ are branch poles of $\cG_x$. Otherwise, at most one endpoint of $L$ is a branch pole of $\cG_x$.
\end{lemma}
\begin{proof}
Since $\cB'(V_0^-) = 0$ and $\cB'(V_n^+) = 1$, we have $\min \cJ = \min J_{V_0} = \sigma_1^{-1} e^{V_0}$ and $\max \cJ = \max J_{V_n} = \sigma_d^{-1} e^{V_n}$. Furthermore, $\cB'(V_0^+) - \cB'(V_0^-) > 0$ and $\cB'(V_n^+) - \cB'(V_n^-) > 0$.

We prove the lemma by examining the case of the external component and internal components separately.

\vspace{5mm}
\noindent
\underline{External.} Consider the case where $L$ is the external component of $\hatR \setminus \cJ_x$. If $x \in (V_0,V_n)$, then the endpoints of $L$ are $\min \cJ$ and $\max \cJ$, both of which are elements of $\cJ \setminus J_x$. Thus these are branch poles of $\cG_x$. If $x > V_n$ $(< V_0)$, then one endpoint of $L$ is a branch pole given by $\min \cJ$ ($\max \cJ$) and the other endpoint is a branch zero given by $\max J_x$ ($\min J_x$).

The situation for $x = V_n$ and $x = V_0$ is more complicated due to cancellation of elements in $\cJ$ and $J_x$. Suppose $x = V_n$, then one endpoint of $L$ is still a branch pole given by $\min \cJ$, and if $J_{V_n}$ is a strict subset of $J_x$ then $\max(J_x \setminus J_{V_n})$ is a branch zero and the other endpoint of $L$. However, we may have $J_{V_n} = J_x$. In this case, $\cB'(V_{n-1}^+) = \cB'(V_n^-) = 0$ which implies $\cB'(V_{n-1}^+) - \cB'(V_{n-1}^-) < 0$. Then $\max \cJ_x \setminus J_x = \max J_{V_{n-1}}$ is a branch zero of $\cG_x$ and the other endpoint of $L$. The argument for $x = V_0$ is similar. This concludes the analysis for the external components.

\vspace{5mm}
\noindent
\underline{Internal.} Now consider the case where $L$ is some internal component. Suppose for contradiction that both endpoints of $L$ are branch poles of $\cG_x$.

We claim that $L = (\sigma_i^{-1} e^V, \sigma_{i+1}^{-1} e^W)$ for non-differentiable points $V \le W$ of $\cB$ and for some $i \in [[1,d-1]]$. By Lemma \ref{lem:Vchange}, each endpoint of $L$ must be an element of $J_U$ for some non-differentiable point $U$ of positive change in slope. The endpoints of $L$ must be adjacent elements of $\cJ$. If $\sigma_i^{-1} e^V = \max J_V$, then the right endpoint of $L$ must be $\sigma_{i+1}^{-1} e^W = \min J_W$ where $W$ is the adjacent non-differentiable point $V < W$ by Lemma \ref{lem:Vchange}. Otherwise, the right endpoint of $L$ must be $\sigma_{i+1}^{-1} e^V$. This proves the claim.

Observe that
\begin{align*}
\arg_+ Q(L) & = -\frac{\pi}{p} \sum_{\ell=0}^n \sum_{\substack{\sigma \in \cS: \\ \sigma^{-1} e^{V_\ell} \in J_{V_\ell}}} (\1[\cB'(V_\ell^+) \ge \varsigma_\sigma] - \1[\cB'(V_\ell^-) \ge \varsigma_\sigma]) \1[ \sigma^{-1} e^{V_\ell} < L] \\
&= -\frac{\pi}{p} \sum_{\sigma \in \cS} \left( \1[\cB'(V^+) \ge \varsigma_\sigma] \1[\sigma^{-1} \le \sigma_i^{-1}] + \1[\cB'(V^-) \ge \varsigma_\sigma] \1[\sigma^{-1} > \sigma_i^{-1}] \right) \\
&= -\frac{\pi}{p} \sum_{\sigma \in \cS} \1[\sigma^{-1} \le \sigma_i^{-1}] = -\pi \varsigma_i
\end{align*}
where the first equality follows from (\ref{eq:Q}), the second from telescoping in $\ell$, the third follows from the fact that $\cB'(V^-) < \varsigma_i$ and $\cB'(V^+) \ge \varsigma_i$ by \Cref{lem:Vchange} and monotonicity of $\varsigma_\sigma$ in $\sigma^{-1}$.

We require $\arg_+ P_x(L) - \arg_+ Q(L) = \arg_+ \cG_x(L) = 0$, that is $\arg_+ P_x(L) = -\pi \varsigma_i$. This happens exactly when
\begin{align} \label{eq:Lsandwich}
\sigma_i^{-1} e^x < L < \sigma_{i+1}^{-1} e^x.
\end{align}
Since, $L = (\sigma_i^{-1} e^V, \sigma_{i+1}^{-1} e^W)$ where $V \le W$, this is only possible if $x = V = W$. However, this is the case where the branch pole $\sigma_i^{-1} e^V \in J_V$ is cancelled by the branch zero $\sigma_i^{-1} e^x$ from $P_x$, likewise for $\sigma_{i+1}^{-1} e^W$. Thus neither endpoints of $L$ are branch poles. This completes the proof of the lemma.
\end{proof}

The lemma below allows us to extract data from the phase equation.

\begin{lemma} \label{lem:rootarg}
Fix $x \in \R$ and suppose a continuous root function $\zeta(y)$ of (\ref{eq:compeq}) approaches $a \in \cJ_x$ as $y \to +\infty$ $(-\infty)$. Writing $\zeta = a + \e e^{i\theta}$ with $\theta \in [0,\pi]$, (\ref{eq:phaseq}) has the form
\begin{align}
\arg_+ \cG_x(a^-) + (\arg_+ \cG_x(a^+) - \arg_- \cG_x(a^-)) \cdot \theta + \e O( \theta \wedge (\pi - \theta)) = 0.
\end{align}
\end{lemma}
\begin{proof}
Clearly, $\arg(1 - a^{-1} \zeta) = \theta$. For $b \neq a$, we have
\begin{align} \label{eq:Oarg}
\arg(1 - b^{-1} \zeta) = \left\{\begin{array}{cc}
-\e O(\theta\wedge(\pi - \theta)) & \mbox{if $a < b$}, \\
-(\pi - \e O(\theta\wedge(\pi - \theta)) & \mbox{if $a > b$}.
\end{array} \right.
\end{align}
In the case $b = 0$, consider $\arg(-\zeta)$ in which case we still have (\ref{eq:Oarg}). The lemma follows from (\ref{eq:Gx}).
\end{proof}

The next lemma connects roots of (\ref{eq:compeq}) with branch poles of $\cG_x$.

\begin{lemma} \label{lem:Qrootcount}
Fix $x \in \R$.

\begin{enumerate}
    \item For each component $L \subset \hatR \setminus \cJ_x$ such that $\arg_+ \cG_x(L) = 0$ and each endpoint $\fa$ of $L$ such that $\fa$ is a branch pole of $\cG_x$, there exists a continuous root function $\zeta_\fa(y)$ of (\ref{eq:compeq}) such that
    \[ \lim_{y\to-\infty} \zeta_\fa(y) = \fa \]
    which is uniquely defined on $(-\infty,y_0)$ for some suitably negative $y_0$.
    
    \item Conversely, there exists some suitably negative $y_0$ such that if $\zeta$ is a root of (\ref{eq:compeq}) at $(x,y)$ for some $y < y_0$, then $\zeta = \zeta_\fa(y)$ for some component $L \subset \hatR \setminus \cJ_x$ such that $\arg_+ \cG_x(L) = 0$ with endpoint $\fa$ which is a branch pole of $\cG_x$.
    
    \item A component $L \subset \hatR \setminus \cJ_x$ contains exactly one root of (\ref{eq:compeq}) for $y$ near $-\infty$ if and only if $L$ contains at least one root of (\ref{eq:compeq}) for all $y \in \R$.
\end{enumerate}
\end{lemma}
\begin{proof}
By (\ref{eq:magneq}), as $y \to -\infty$, any root of (\ref{eq:compeq}) must get arbitrarily close to a branch pole of $\cG_x$. By Lemma \ref{lem:rootarg}, the set of such branch poles $\fa$ must be those for which $\arg \cG_x(\fa^+) = 0$ or $\arg \cG_x(\fa^-) = 0$ (both cannot hold simultaneously); thus $\fa$ is an endpoint of some component $L \subset \hatR \setminus \cJ_x$ such that $\arg_+ \cG_x(L) = 0$. \Cref{lem:rootarg} further indicates that if $\zeta$ is near $\fa$ for $y$ sufficiently negative, then $\arg_+ \cG_x(\zeta) = 0$ which means that $\zeta \in L$.

(1) For each $\fa$ which is a branch pole of $\cG_x$ and an endpoint of some $L$ with $\arg_+ \cG_x(L)$, we can see from (\ref{eq:compeq}) that there exists some root $\zeta_\fa(y) \in L$ near $\fa$ for $y$ sufficiently negative. By (\ref{eq:compeq}) and the constraint that roots near $\fa$ must be real for $y$ sufficiently negative, we have that $\zeta_\fa(y)$ is uniquely defined on $(-\infty,y_0)$ for some suitably negative $y_0$. This root function is necessarily continuous for $y \in (-\infty,y_0)$.

(2) By the discussion in the first paragraph and uniqueness of $\zeta_\fa$, we have that any root $\zeta$ of (\ref{eq:compeq}) must coincide with some $\zeta_\fa(y)$ for $y$ sufficiently small.

(3) If $L$ contains exactly one root of (\ref{eq:compeq}), then one endpoint is a branch pole and the other is a branch zero. Roots are added to $L$ if nonreal pairs coalesce and are removed from $L$ if real roots coalesce. This implies that the parity of the number of roots in $L$ must be preserved. Thus for any value of $y$, $L$ must contain a root of (\ref{eq:compeq}).
\end{proof}

We are now in a position to prove the existence of the map $\zeta(x,y)$.

\begin{proposition} \label{prop:2roots}
The $(x,y)$-companion equation has at most $2$ nonreal roots for a given $(x,y) \in \R^2$. If (\ref{eq:compeq}) has a nonreal root, then $x\in (V_0,V_n)$
\end{proposition}
\begin{proof}
Fix $x \in \R$. By the contour integral formula for roots of functions applied to $\cG_x(\zeta) - e^{-y}$ which is analytic in $w = e^{-y}$, we have that the number $\nu$ of roots of (\ref{eq:compeq}) is constant as $y$ varies in $\R$.

Let $\cL_x$ be the set of components of $\hatR \setminus \cJ_x$ containing exactly one root of (\ref{eq:compeq}) for $y$ near $-\infty$. By Lemma \ref{lem:Qrootcount}, if $x \in (V_0,V_n)$ then the total number of roots of (\ref{eq:compeq}) is
\[ \nu = |\cL_x| + 2 \]
where the $2$ is contributed by the two endpoints of the external component. Moreover, since there is at least one root contained in each component $L \in \cL_x$ for all $y \in \R$, we must have at least $|\cL_x|$ real roots of (\ref{eq:compeq}) at any $y \in \R$. Thus there are at most two nonreal roots of (\ref{eq:compeq}) at any given $(x,y) \in\R^2$.

Similarly by Lemma \ref{lem:Qrootcount}, if $x \notin (V_0,V_n)$ then the total number of roots of (\ref{eq:compeq}) is $\nu = |\cL_x|$. Moreover, since there is at least one root contained in each component $L \in \cL_x$ for all $y \in \R$, all the roots are real for all $y \in \R$.
\end{proof}

\begin{remark} \label{rem:2roots}
We note that a continuous extension $\wzeta(x,y)$ of $\zeta(x,y)$ onto $\R^2$ is not uniquely defined. However, the component of $\hatR \setminus \cJ_x$ which $\wzeta(x,y)$ belongs to when it is real, i.e. $(x,y) \notin \sL$, \emph{is} well-defined. From the proof of Proposition \ref{prop:2roots}, for any fixed $x \in I$ we see that $\wzeta(x,y)$ is contained in the external component for $y$ sufficiently negative. This is because the external component has $2$ roots and all other components have at most one root for $y$ sufficiently negative, and we require a coalescing of real roots to get a pair of nonreal roots. Thus we have the freedom to choose $\wzeta$ so that $\wzeta(x,y) \to \min \cJ_x$ as $y \to -\infty$ for all $x\in I$, or so that $\wzeta(x,y) \to \max \cJ_x$ as $y \to -\infty$ for all $x\in I$; note that it is not clear that $\wzeta$ is uniquely determined upon making this choice, and we will not need this fact.
\end{remark}

\begin{remark} \label{rem:2roots+}
As $y$ increases from $-\infty$, we have that the first coalescing of roots (i.e. the first double root of (\ref{eq:compeq})) appears when the two roots starting at opposite endpoints of the external component coalesce. If $y = y_0$ is  when the first double root appears, then each point in the external component must have been a root of the $(x,y)$-companion equation for some $y \le y_0$.
\end{remark}

\subsection{Diffeomorphism between Roots and \texorpdfstring{$\HH$}{H}} \label{sec:homeo}
Proposition \ref{prop:2roots} gives us the existence of a map $\zeta:\sL \to \HH$. We now show that this map is a diffeomorphism. Throughout this subsection, it will be convenient to use the symbol $\zeta$ as a variable in $\C$. To avoid ambiguity we will write $\zeta(x,y)$ when referring to the root function.

It is useful to consider the spectral curve of the RPP (see \cite{KOS}), as they can be viewed as the source of the surjectivity of $\zeta(x,y)$. To this end, we start by rewriting the companion equation into an alternative convenient form.

\begin{definition}
Define the analytic function
\begin{align} \label{eq:Qrel}
\cQ(\zeta,\eta) = \eta \cdot Q(\zeta)^{-1} - 1
\end{align}
and the \textit{spectral curve for periodic weights} $s_0,\ldots,s_{p-1}$
\begin{align}
P(z,w) = \prod_{\sigma \in \cS} (1 - \sigma z)^{\frac{1}{p}} - w
\end{align}
where the branch is chosen so that the argument of $1 - \sigma z$ is $0$ when $z < \sigma^{-1}$.
\end{definition}

We will write $(z,w) \in P$ to mean $P(z,w) = 0$. Then $\zeta \in \HH$ is a solution to the companion equation for some $(x,y) \in \R^2$ if and only if there exists $(z,w) \in P$ such that
\begin{align} \label{eq:char}
\cQ(\zeta,\eta) = 0, \\ \label{coord}
(\zeta,\eta) = (e^x z, e^y w). 
\end{align}

\begin{remark}
This definition of the spectral curve differs from the definition in \cite{KOS}, which gives $\prod_{\sigma \in \cS}(1 - \sigma z) - w^p$. This difference arises from a rescaling of the $x$-axis, namely that \cite{KOS} contracts the $x$-direction by $p$ to make the corresponding fundamental domain $\Z\times \Z$ periodic whereas our fundamental domain is $p\Z \times \Z$ periodic.
\end{remark}

\begin{lemma} \label{lem:spec}
The set of $(z,w) \in P$ with $z \in \HH$ satisfies $0 < - \arg w < \pi - \arg z$. Moreover, for any pair $\varphi,\theta$ with $0 < \varphi < \pi - \theta$, there exists a unique $(z,w) \in P$ with $\arg w = -\varphi$, $\arg z = \theta$.
\end{lemma}
\begin{proof}
Fix $\theta \in(0,\pi)$. By geometric considerations, the map $\rho \mapsto \arg(1 - \rho e^{i\theta})$ strictly decreases on $(0,\infty)$, approaching $0$ as $\rho \to 0$ and approaching $-(\pi - \theta)$ as $\rho \to \infty$. Thus the map
\[\vartheta: \rho \mapsto \frac{1}{p} \sum_{\sigma \in \cS} \arg(1 - \sigma \rho e^{i\theta}) \]
also strictly decreases on $(0,\infty)$, approaching $0$ as $\rho \to 0$ and approaching $-(\pi - \theta)$ as $\rho \to \infty$.

Taking $z = \rho e^{i\theta}$, we have that $\vartheta(\rho) = \arg w$. Since $\vartheta$ is a bijection, any point $(\theta,-\varphi)$ such that $0 < \varphi < \pi - \theta$ uniquely determines a point $(z,w) \in P$ where $(\arg z, \arg w) = (\theta,-\varphi)$.
\end{proof}

\begin{proof}[Proof of Theorem \ref{thm:homeo}]
We first establish that $\zeta(x,y)$ is a bijection from $\sL$ onto $\HH$. Given $\zeta \in \HH$, the pair $(\zeta,\eta) = (\zeta,Q(\zeta))$ solves (\ref{eq:char}). To show that $\zeta$ is a solution to the companion equation for some $(x,y) \in \R^2$ (and therefore in $\sL$), we must find $(z,w) \in P$ so that $(\zeta,\eta) = (e^x z, e^y w)$. Alternatively said, we want $(z,w) \in P$ so that $\arg(\zeta,\eta) = \arg(z,w)$, and by Lemma \ref{lem:spec} this $(z,w) \in P$ is unique if it exists. By Lemma \ref{lem:spec} again, it is enough to show that $0 < -\arg \eta < \pi - \arg \zeta$.

By geometric considerations (see proof of Lemma \ref{lem:spec}), given $\zeta \in \HH$ and $a,b \in [0,\infty]$ such that $a < b$, we have
\[ -(\pi -  \arg \zeta) \leq \arg(1 - a^{-1} \zeta) < \arg(1 - b^{-1} \zeta) \leq 0 \]
where we mean $\arg(-\zeta) = -(\pi - \arg \zeta)$ for $\arg(1 - \infty \zeta)$. By (\ref{eq:1-s})
\begin{align} \label{eq:eta1}
\begin{multlined}
\arg \eta = \frac{1}{p}  \sum_{\sigma \in \cS} \Biggl[ \arg(1 - e^{-V_0} \sigma \zeta) \\
+ \sum_{\ell=1}^n \1[ \cB'(V_{\ell-1},V_\ell) < \varsigma_\sigma] \left( \arg(1 - e^{-V_\ell} \sigma \zeta) - \arg(1 - e^{-V_{\ell-1}} \sigma \zeta) \right) \Biggr]
\end{multlined}
\end{align}
and by (\ref{eq:1+s})
\begin{align} \label{eq:eta2}
\begin{multlined}
\arg \eta = \frac{1}{p} \sum_{\sigma \in \cS} \Biggl[ \arg(1 - e^{-V_n} \sigma \zeta) \\
+ \sum_{\ell=1}^n \1[ \cB'(V_{\ell-1},V_\ell) \ge \varsigma_\sigma ] \left( \arg(1 - e^{-V_{\ell-1}} \sigma \zeta) - \arg(1 - e^{-V_\ell} \sigma \zeta) \right) \Biggr].
\end{multlined}
\end{align}
In (\ref{eq:eta1}), each summand in the summation over $\ell$ is $\ge 0$ with some $> 0$. This implies
\[ \arg \eta > \sum_{\sigma \in \cS} \frac{1}{p} \arg(1 - e^{-V_0} \sigma \zeta) \ge -(\pi - \arg\zeta). \]
In (\ref{eq:eta2}), each summand in the summation over $\ell$ is $\le 0$ with some $< 0$. This implies
\[ \arg \eta < \sum_{\sigma \in \cS} \frac{1}{p} \arg(1 - e^{-V_n} \sigma \zeta) \le 0. \]
So we have $0 < -\arg \eta < \pi - \arg \zeta$. This proves $\zeta(x,y)$ is surjective. Moreover, we see that the procedure from $\zeta$ to $(x,y)$ uniquely determines $(x,y)$, thus $\zeta(x,y)$ is also injective.

Since $\zeta(x,y)$ is differentiable on $\sL$, it remains to show the differentiability of the inverse map. However, the chain of procedures $\zeta \mapsto (\zeta,\eta) = (\zeta,Q(\zeta)) \mapsto (\arg \zeta,\arg\eta)$ is differentiable. Furthermore, the map $(\theta,-\varphi) \mapsto (z,w) \in P$ where $\arg(z,w) = (\theta,-\varphi)$ is differentiable. Therefore $\zeta \mapsto \left( \frac{\zeta}{z},\frac{\eta}{w}\right) = (e^x,e^y)$ is differentiable. It follows that the inverse of $\zeta(x,y)$ is differentiable.
\end{proof}

\subsection{Double Roots and the Frozen Boundary} \label{ssec:frozen}
In this section, we describe the \emph{frozen boundary} which we define to be the boundary of the liquid region $\sL$. This requires understanding the double roots of the companion equation.

\label{sec:frzbdry}
By taking the logarithm and differentiating, a root of the companion equation is a double root if it solves
\begin{align} \label{eq:ogdub}
\frac{1}{p} \sum_{\sigma \in \cS} \frac{1}{\zeta - e^x\sigma^{-1}} - \Sigma(\zeta) = 0
\end{align}
where
\begin{align}
\Sigma(\zeta) &= \frac{1}{p} \sum_{\substack{\ell \in [[0,n]] \\ \sigma \in \cS}} \left(\1[\cB'(V_\ell^+) \ge \varsigma_\sigma] - \1[\cB'(V_\ell^-) \ge \varsigma_\sigma] \right) \frac{1}{\zeta - e^{V_\ell} \sigma^{-1}}
\end{align}
and we take $e^{V_0} = 0$ if $V_0 = -\infty$ and $\frac{1}{\zeta - e^{V_n}\alpha^{-1}} = 0$ if $V_n = +\infty$. We have the alternative expressions, corresponding to (\ref{eq:1-s}) and (\ref{eq:1+s}) respectively,
\begin{align} \label{eq:V0sigi}
\Sigma(\zeta) &= \frac{1}{p} \sum_{\sigma \in \cS} \Biggl[ \frac{1}{\zeta - e^{V_0} \sigma^{-1}} + \sum_{\ell=1}^n \1[ \cB'(V_{\ell-1},V_\ell) < \varsigma_\sigma] \left( \frac{1}{\zeta - e^{V_\ell} \sigma^{-1}} - \frac{1}{\zeta - e^{V_{\ell-1}}\sigma^{-1}} \right) \Biggr] \\ \label{eq:Vnsigi}
&= \frac{1}{p} \sum_{\sigma \in \cS} \Biggl[ \frac{1}{\zeta - e^{V_n} \sigma^{-1}} + \sum_{\ell=1}^n \1[\cB'(V_{\ell-1},V_\ell) \ge \varsigma_\sigma] \left( \frac{1}{\zeta - e^{V_{\ell-1}} \sigma^{-1}} - \frac{1}{\zeta - e^{V_\ell} \sigma^{-1}} \right) \Biggr]
\end{align}
taking the same modifications as above for $V_0 = -\infty$ and $V_n = +\infty$.

We reexpress the double root equation to obtain a parametrization of the double roots of (\ref{eq:compeq}) in terms of $\zeta \in \R$. Define the function
\[ f(u) = \frac{1}{p} \sum_{\sigma \in \cS} \frac{1}{1 - \sigma^{-1} u} \]
to rewrite the double root equation as 
\begin{equation} \label{eq:fdub}
f(e^x \zeta^{-1}) = \zeta \Sigma(\zeta).
\end{equation}
We want to invert $f$, though we require multiple inverses.

Take $f(\infty) = 0$ so that $\lim_{|u| \to \infty} f(u) = f(\infty)$. Then $f$ is defined on $\hatR \setminus \{\sigma\}_{\sigma \in \cS}$ and is invertible on each connected component. The connected components are $E_0 = (\sigma_1,0)$ (that is $(\sigma_1,+\infty) \cup \{\infty\} \cup(-\infty,0)$), $E_j = (\sigma_{j+1},\sigma_j)$ for $1 \le j < d$, and $E_d = (0,\sigma_d)$. Let $\varphi_j$ denote the inverse of $f$ restricted to $E_j$. For $j \in [[1,d-1]]$, we have $\varphi_j:\R \to E_j$ whereas $\varphi_0:(-\infty,1) \to E_0$ and $\varphi_d:(1,+\infty) \to E_d$. Note that $\varphi_0$ maps $(-\infty,0)$ to $(\sigma_1,+\infty)$, $0$ to $\infty$, and $(0,1)$ to $(-\infty,0)$.

Suppose $e^x \zeta^{-1} \in E_j$ for some $j \in [[0,d]]$. Then $\zeta \in (-\infty,\sigma_1^{-1}e^x)$ if $j = 0$, $\zeta \in (e^x \sigma_j^{-1}, e^x \sigma_{j+1}^{-1})$ if $j \in [[1,d-1]]$, and $\zeta \in (e^x \sigma_d^{-1},+\infty)$ if $j = d$. In particular, $\arg_+ P_x(\zeta) = \varsigma_j$.

By inverting (\ref{eq:fdub}), the condition for being a double root of (\ref{eq:compeq}) becomes
\begin{align} \label{eq:xeq}
e^x &= \zeta \varphi(\zeta \Sigma(\zeta))\\ \label{eq:yeq}
e^y &= \frac{Q(\zeta)}{P_x(\zeta)} = \frac{Q(\zeta)}{\prod_{\sigma \in \cS} (1 - \frac{\sigma}{\varphi(\zeta \Sigma(\zeta))})^{\frac{1}{p}}}
\end{align}
where the choice of inverse $\varphi$ must make the right hand side of (\ref{eq:yeq}) positive. More specifically, we have $\arg_+ Q(\zeta) = \varsigma_j$ for some $j \in [[0,d]]$. Then choose $\varphi = \varphi_j$ which implies $\arg_+ P_x(\zeta) = \varsigma_j$ by our discussion above. Thus the argument $\arg_+$ of the right hand side of (\ref{eq:yeq}) is $0$. Moreover, $\varphi_j$ is the only inverse of $f$ which allows the right hand side of (\ref{eq:yeq}) to be positive. We record our choice of $\varphi$ below:
\begin{align} \label{eq:varphi}
\varphi = \varphi_j:~\arg_+ Q(\zeta) = \varsigma_j
\end{align}

The following proposition shows that (\ref{eq:xeq}) and (\ref{eq:yeq}) provides an $\R \setminus \cJ$-parametrization of the frozen boundary.

\begin{proposition}
Each $\zeta \in \R \setminus \cJ$ is a double root of (\ref{eq:compeq}) for some unique $(x,y) \in \R^2$.
\end{proposition}
\begin{proof}
Let $\zeta \in \R \setminus \cJ$. Consider first $\zeta = 0$ (when $0 \notin \cJ$); note that $V_0 > -\infty$ is equivalent to $0 \notin \cJ$. Then (\ref{eq:ogdub}) becomes
\[ \frac{1}{p} \sum_{\sigma \in \cS} e^{-x} \sigma = -\Sigma(0). \]
The right hand side is positive by (\ref{eq:Vnsigi}), and so we can solve for $x$. The companion equation also gives us
\[ e^{-y} = \cG_x(0) = e^{-\cB(V_0)} \]
so that we can solve for $y$. This proves the $\zeta = 0$ case.

For $\zeta \neq 0$, we show that the right hand sides of (\ref{eq:xeq}) and (\ref{eq:yeq}) are positive. The positivity of the right hand side of (\ref{eq:yeq}) follows from the choice of inverse $\varphi$. We note that the uniqueness of $(x,y)$ follows from the fact that the choice of inverse $\varphi$ of $f$ is the only one that gives positivity (of (\ref{eq:yeq})) as mentioned above.

It remains to show that the right hand side of (\ref{eq:xeq}) is positive. Suppose $j \in [[0,d]]$ so that $\varphi = \varphi_j$. We proceed by case analysis.

Suppose $j \neq 0,d$. Then $\varphi(\zeta \Sigma(\zeta)) \in E_j \subset \R_{> 0}$ and $\arg_+ P_x(\zeta) = \varsigma_j$. These imply $\varphi_j(\zeta \Sigma(\zeta)) > 0$ and $\zeta > 0$. Thus the right hand side of (\ref{eq:xeq}) is positive.

In preparation for the remaining cases, we note that (\ref{eq:V0sigi}) and (\ref{eq:Vnsigi}) imply
\begin{align} \label{eq:V0}
\zeta \Sigma(\zeta) &= \frac{1}{p} \sum_{\sigma \in \cS} \Biggl[ \frac{1}{1 - e^{V_0} (\sigma\zeta)^{-1}} + \sum_{\ell=1}^n \1[\cB'(V_\ell^-) < \varsigma_\sigma] \left( \frac{1}{1 - e^{V_\ell} (\sigma\zeta)^{-1}} - \frac{1}{1 - e^{V_{\ell-1}}(\sigma\zeta)^{-1}} \right) \Biggr] \\ \label{eq:Vn}
&= \frac{1}{p} \sum_{\sigma \in \cS} \Biggl[ \frac{1}{1 - e^{V_n} (\sigma\zeta)^{-1}} + \sum_{\ell=1}^n \1[\cB'(V_\ell^-) \ge \varsigma_\sigma] \left( \frac{1}{1 - e^{V_{\ell-1}} (\sigma\zeta)^{-1}} - \frac{1}{1 - e^{V_\ell} (\sigma\zeta)^{-1}} \right) \Biggr]
\end{align}

We also record
\begin{align} \nonumber
\arg_+ Q(\zeta) &= -\frac{\pi}{p} \sum_{\ell=0}^n \sum_{\substack{\sigma \in \cS \\ \sigma^{-1} e^{V_\ell} \in J_{V_\ell}}} (\1[\cB'(V_\ell^+) \ge \varsigma_\sigma] - \1[\cB'(V_\ell^-) \ge \varsigma_\sigma]) \1[ \sigma^{-1} e^{V_\ell} < \zeta ] \\ \label{eq:Qarg}
&= - \frac{\pi}{p} \sum_{\sigma \in \cS} \1[\cB'(V_{\ell_\sigma-1}^+) \ge \varsigma_\sigma] = - \frac{\pi}{p} \sum_{\sigma \in \cS} \1[ \cB'(V_{\ell_\sigma}^-) \ge \varsigma_\sigma]
\end{align}
where $\ell_\sigma = \max\{\ell: \sigma^{-1} e^{V_{\ell-1}} < \zeta\}$.

\underline{If $j = d$}, we must check $\zeta \Sigma(\zeta) > 1$, i.e. is in the domain of $\varphi_d$. Since $\arg_+ Q(\zeta) = -\pi$, by (\ref{eq:Qarg}) we have $\cB'(V_{\ell_\sigma}^-) \ge \varsigma_\sigma$ for all $\sigma \in \cS$. Plugging this into (\ref{eq:V0}), along with the fact that $\zeta \in (\sigma^{-1} e^{V_{\ell_\sigma}-1}, \sigma^{-1} e^{V_{\ell_\sigma}})$, we have $\zeta \Sigma(\zeta) > 1$. Thus $\zeta \Sigma(\zeta)$ is in the domain of $\varphi$. Since $\varphi(\zeta \Sigma(\zeta)) \in E_0$ and $\zeta > 0$, the right hand side of (\ref{eq:xeq}) is positive.

\underline{If $j = 0$}, we must check that $\zeta \Sigma(\zeta) < 0$ or $0 < \zeta \Sigma(\zeta) < 1$, i.e. is in the domain of $\varphi_0$. If $\zeta < 0$, then (\ref{eq:V0}) implies $\zeta \Sigma(\zeta) > 0$ and (\ref{eq:Vn}) implies $\zeta \Sigma(\zeta) < 1$. We then have $\varphi(\zeta \Sigma(\zeta)) < 0$. Thus $\zeta \varphi(\zeta \Sigma(\zeta)) > 0$. 

Otherwise, $\zeta > 0$. Since $\arg_+ Q(\zeta) = 0$, by (\ref{eq:Qarg}) we have $\cB'(V_{\ell_\sigma}^-) < \varsigma_\sigma$ for all $\sigma \in \cS$. Plugging this into (\ref{eq:Vn}), we have $\zeta \Sigma(\zeta)) < 0$. Then $\varphi(\zeta \Sigma(\zeta)) > 0$ so that the right hand side of (\ref{eq:xeq}) is positive.
\end{proof}

The next proposition explains what happens as $\zeta$ approaches an element of $\cJ$.

\begin{proposition} \label{prop:drparam}
Let $V$ be an non-differentiable point of $\cB$. As $\zeta \to \sigma^{-1} e^V \in J_V$ in (\ref{eq:xeq}) and (\ref{eq:yeq}), we have $x \to V$ and $y$ converges to a finite limit if $\cB'(V^+) - \cB'(V^-) > 0$, and $y \to +\infty$ if $\cB'(V^+) - \cB'(V^-) < 0$.
\end{proposition}
\begin{proof}
Let $\eta = \sigma_i^{-1} e^V$ for some $i \in [[1,d]]$. By (\ref{eq:Qarg})
\begin{align}
\arg_+ Q(\eta_+) = \frac{1}{p} \sum_{\tau^{-1} \le \sigma_i^{-1}} \1[\cB'(V^+) \ge \varsigma_\tau] + \frac{1}{p} \sum_{\tau^{-1} > \sigma_i^{-1}} \1[\cB'(V^-) \ge \varsigma_\tau].
\end{align}
If $\cB'(V^+) - \cB'(V^-) > 0$ ($< 0$), then by Lemma \ref{lem:Vchange},
\begin{align*}
\arg_+ Q(\eta^+) &= - \pi \varsigma_i ~(= -\pi \varsigma_{i-1}), \\
\arg_+ Q(\eta^-) &= - \pi \varsigma_{i-1} ~(= -\pi \varsigma_i),
\end{align*}
and $\Sigma(\eta^\pm) = \pm\infty (= \mp\infty)$. By (\ref{eq:varphi}) and the manner in which $\varphi_j$ maps into $E_j$ for $j \in [[0,d]]$, we have $\lim_{\zeta \to \eta^\pm} \varphi(\zeta \Sigma(\zeta)) = \sigma_i$. Thus $\lim_{\zeta \to \eta^\pm} \zeta \varphi(\zeta \Sigma(\zeta)) = e^V$. We also have that $Q$ has a branch zero (pole) at $\sigma_i^{-1} e^V$ of order $\frac{|S_{\sigma_i}|}{p}$.

Thus if $\cB'(V^+) - \cB'(V^-) > 0$, then as $\zeta \to \eta$ we have $e^y = \frac{Q(\zeta)}{P_x(\zeta)}$ has a finite limit since the branch zero of $Q$ is met by the branch zero of $P_x$ of equal order. In the case $\cB'(V^+) - \cB'(V^-) < 0$, then $e^y = \frac{Q(\zeta)}{P_x(\zeta)}$ approaches $+\infty$ as $\zeta \to \eta$. This proves the proposition.
\end{proof}

\begin{remark}
Proposition \ref{prop:drparam} asserts that the points $\sigma^{-1} e^V \in J_V$ such that $\cB'(V^+) - \cB'(V^-) < 0$ correspond to parts of the frozen region where the vertical coordinate is unbounded. These are the \emph{tentacles} in the frozen region which arise due to singular points of the back wall, see Figures \ref{fig:fb1} and \ref{fig:fb2}. The cusps correspond to other points in $J_V$.
\end{remark}

\begin{figure}[ht]
    \centering
    \includegraphics[scale=0.75]{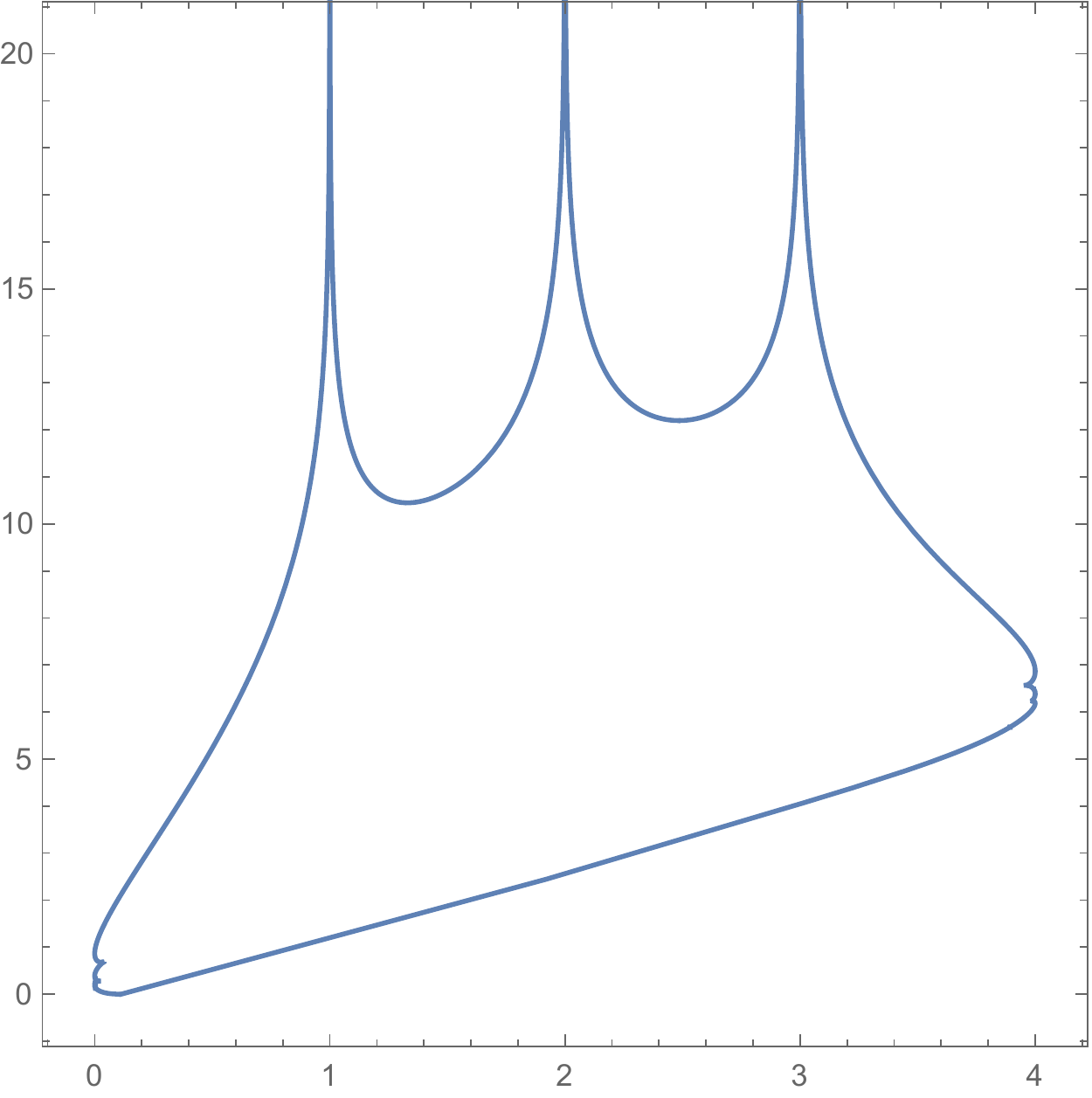}
    \caption{Frozen boundary for periodic weights $2,2,\frac{1}{4}$, $4$ linear back wall pieces with slopes $0, \frac{1}{3}, \frac{2}{3}, 1$ from left to right. Each non-differentiable point is equally spaced from one another and corresponds to a singular point.} \label{fig:fb1}
\end{figure}

\begin{figure}[ht]
    \centering
    \includegraphics[scale=0.75]{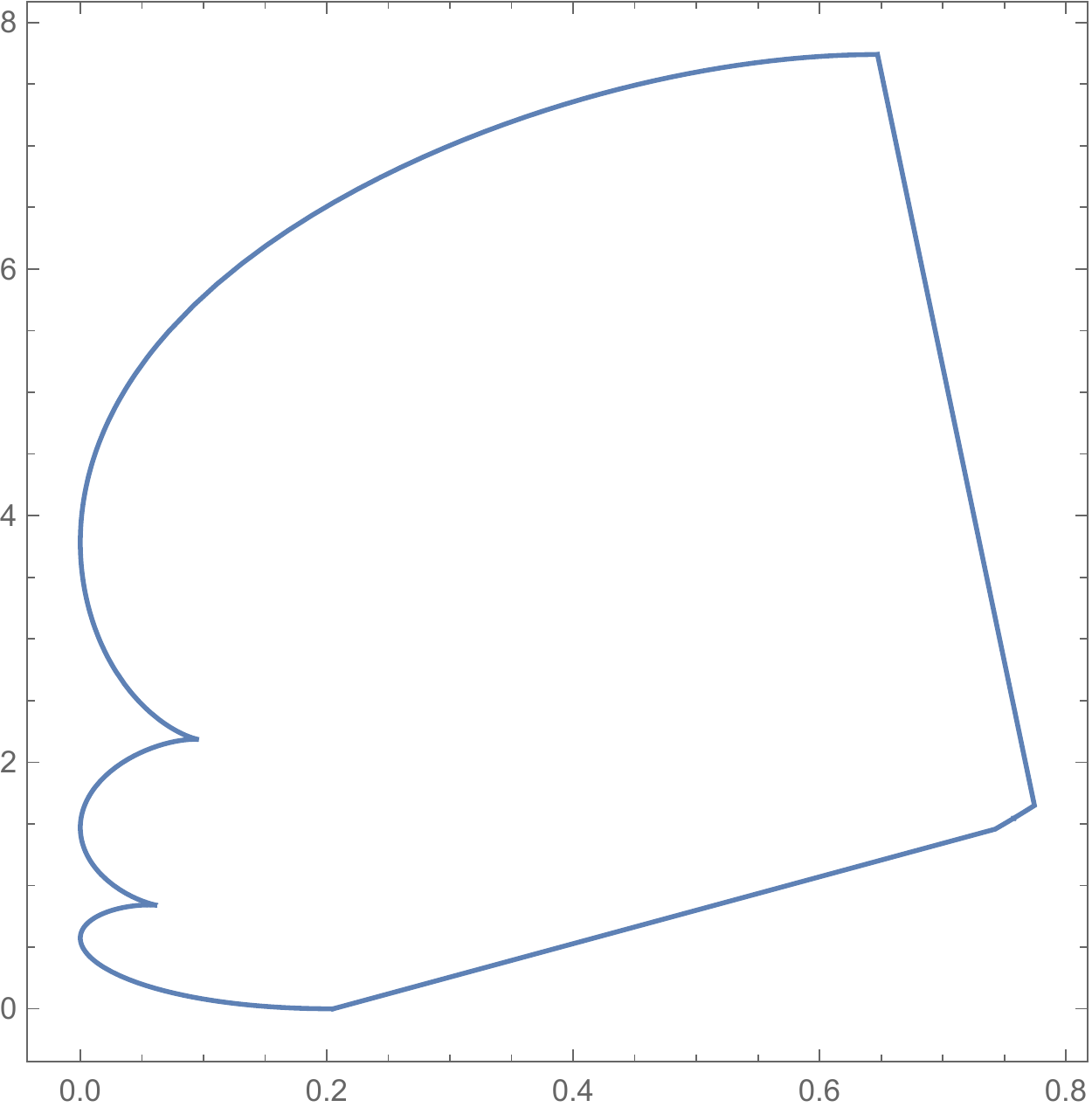}
    \caption{Frozen boundary for periodic weights $4,\frac{1}{4},2,\frac{1}{2},\frac{5}{4}, \frac{4}{5}$, $1$ linear back wall piece with slope $\frac{1}{2}$. See Figure \ref{fig:limshape} for a large sample.} \label{fig:fb2}
\end{figure}

\newpage
\section{Limit Shapes and the Gaussian Free Field} \label{sec:limgff}
In this section, we combine the results from the previous sections to obtain the main results of this article: the limit shape and Gaussian free field fluctuations for the height functions. These correspond to Theorems \ref{thm:macLLN} and \ref{thm:macGFF} given in Section \ref{sec:modelresults}. However, we provide more precise reformulations of these theorems.

Throughout this section, we fix a family $\P^{B,r,\vec{s}}_{\alpha,\ft}$ satisfying the Limit Conditions, as defined in the beginning of Section \ref{sec:backwall}, with limiting back wall $\cB: I \to \R$. We write $\cG_\x := \cG_\x^{\cB}$.

\begin{definition}
We say $x(\e) \in I^\e$ is \emph{separated from singular points} if $x(\e)$ is $d$-\emph{separated from singular points} (see Definition \ref{def:singpt}) for every $d > 0$.
\end{definition}

Let $\zeta = \zeta^{\cB}:\sL \to \HH$ denote the diffeomorphism from Theorem \ref{thm:homeo}. As in Remark \ref{rem:2roots}, we can extend $\zeta$ to a function $\wzeta: I \times \R \to \HH$ such that
\[ \lim_{\y\to-\infty} \wzeta(\x,\y) = \min \cJ_\x = \sigma_1^{-1} e^{V_0}. \]

\begin{theorem} \label{thm:MACLLN}
Suppose $\P^{B,r,\vec{s}}_{\alpha,\ft}$ satisfies the Limit Conditions with $\cB \in \fB^\Delta(\cS)$. There exists a deterministic Lipschitz $1$ function $\cH := \cH^{\cB}: I \times \R \to \R$ \emph{independent of $\alpha,\ft$} such that for any $(\x,\y) \in I\times \R$ and $x(\e) \in I^\e$ separated from singular points with $\e x(\e) \to \mathbf{x}$, we have the convergence of random measures on $\y \in \R$
\begin{align}
\e h\left( x(\e), \frac{\y}{\e}\right) \to \frac{1}{\alpha}\cH(\x,\y)
\end{align}
weakly in probability as $\e \to 0$. We have the following explicit description of $\cH$. For each $\x \in I$, we have $\cH(\x,\y) = 0$ for sufficiently negative $\y$, and
\begin{align} \label{eq:htarg}
\nabla \cH(\x,\y) := (\partial_\x\cH(\x,\y),\partial_\y \cH(\x,\y)) = \frac{1}{\pi}\left(\sum_{\sigma \in \cS} \frac{1}{p} \arg \left( 1 - \sigma^{-1} e^\x \wzeta(\x,\y) \right),\pi \cdot \1_{e^{-y} < \cG_\x(0)} - \arg \wzeta(\x,\y) \right)
\end{align}
for $(\x,\y) \in I\times \R$ such that the right hand side is well-defined and where the arg branches are chosen so that $\arg z = 0$ for $z > 0$.
\end{theorem}

\begin{remark} \label{rem:wd}
Although the choice of $\wzeta$ was not shown to be well-defined, we note that $\arg \wzeta(\x,\y)$ and $\arg(1 - \sigma^{-1} e^\x \wzeta(\x,\y))$ are well-defined (upon specification of the branch). This follows from Remark \ref{rem:2roots}.
\end{remark}

The condition in the indicator function $e^{-\y} < \cG_\x(0)$ is always true in the liquid region. To see this, let $\sL_{\x_0}$ be the slice $\{(\x_0,\y): \y \in \R, (\x_0,\y) \in \sL\}$. If $e^{-\y} \ge \cG_\x(0)$, then $0$ is a root of the $(\x,-\log \cG_\x(0))$-companion equation. However, by Remark \ref{rem:2roots+}, this implies that $-\log \cG_\x(0) \le \min \sL_\x$.

By the preceding discussion, we have the following corollary:
\begin{corollary}
The limiting height function $\cH$ from Theorem \ref{thm:MACLLN} has the following gradient for $(\x,\y) \in \sL$:
\begin{align}
\nabla \cH(\x,\y) = \frac{1}{\pi}\left(\sum_{\sigma \in \cS} \frac{1}{p} \arg \left( 1 - \sigma^{-1} e^\x \zeta(\x,\y) \right),\pi - \arg \zeta(\x,\y) \right).
\end{align}
\end{corollary}

\begin{theorem} \label{thm:MACGFF}
Suppose $\P^{B,r,\vec{s}}_{\alpha,\ft}$ satisfies the Limit Conditions with $\cB \in \fB^\Delta(\cS)$, $\zeta(x,y) := \zeta^{\cB}(x,y)$ denotes the homeomorphism of Theorem \ref{thm:homeo}, and
\[ \overline{h}(x,y) = h(x,y) - \E h(x,y) \]
denotes the centered height function. Then $\alpha \ft \sqrt{\pi} \cdot \overline{h}$ converges to the $\zeta$-pullback of the GFF in the following sense. For any $\mathbf{x}_1,\ldots,\mathbf{x}_m \in I$, integers $k_1,\ldots,k_m > 0$, and $x_1(\e),\ldots,x_m(\e) \in I^\e$ such that $x_i(\e)$ is $\ft k_i$-separated from singular points and $\e x_i(\e) \to \mathbf{x}_i$ for each $i \in [[1,m]]$, we have the convergence of random vectors
\begin{align}
\left( \alpha \ft \sqrt{\pi} \int \overline{h}\left( x_i(\e),\frac{\y}{\e} \right) e^{-k_i\ft \y} \, d\y \right)_{i \in [[1,m]]} \to \left(\int \fH(\zeta(\x_i,\y)) e^{-k_i\ft \y} \, d\y \right)_{i \in [[1,m]]}.
\end{align}
in distribution as $\e \to 0$, where $\fH$ is the Gaussian free field on $\HH$.
\end{theorem}

Let $x_0 \in \R$ and define
\begin{align}
\cD_{x_0} & = \{\zeta(x,y): x \le x_0, y \in \R, (x,y) \in \sL\}, \\
\cD_{x_0}^{\C} & = \cl(\{z: z \in \cD_{x_0}~\mbox{or}~ \bar{z} \in \cD_{x_0} \}).
\end{align}
The boundary of $\cD_{x_0}^{\C}$ is
\begin{align} \label{eq:Dbdry}
\partial \cD_{x_0}^{\C} = \cl\left(\{\zeta(x_0,y):y \in \R, (x_0,y) \in \sL\}\cup\{\overline{\zeta}(x_0,y):y \in \R, (x_0,y) \in \sL\}\right).
\end{align}
The points gained in taking the closure are all in $\R$.

By the parametrization of the frozen boundary in Section \ref{sec:frzbdry} and Proposition \ref{prop:drparam}, we have:

\begin{lemma} \label{lem:deform}
The set $\cD_{x_0}^{\C}$ contains a neighborhood of $J_{V_\ell}$ if $V_\ell < x_0$, and is separated from $J_{V_\ell}$ if $V_\ell > x_0$.
\end{lemma}

\subsection{Limit Shape: Proof of Theorem \ref{thm:MACLLN}} \label{sec:LLN}
Let $\x \in I$ and $x(\e) \in I^\e$ be separated from singular points with $\e x(\e) \to \x$. By Theorem \ref{thm:mom} and Proposition \ref{prop:aht}, the rescaled height function $\e h\left(x(\e), \frac{\y}{\e}\right)$ converges to some limiting distribution $\cH^{\alpha,\ft}$ determined by
\begin{align} \label{eq:hemom}
\int e^{-\ft k\y} \cH^{\alpha,\ft}(\x,\y)\,d\y = \frac{1}{\alpha k^2 \ft^2} \frac{1}{2\pi\i} \oint_{\cC} \! \frac{\cG_\x(z)^{k\ft}}{z} \, dz
\end{align}
for $k \in \Z_{> 0}$ and where the contour $\cC$ is described in Theorem \ref{thm:mom}.

We first show that $\cH^{\alpha,\ft}$ is independent of $\ft$, and dependent on $\alpha$ by a simple scaling factor. The function $y\mapsto \E \e h\left( x(\e),\frac{y}{\e} \right)$ is $\frac{1}{\alpha}$-Lipschitz by (\ref{eq:aht}), nonnegative and monotonically increasing. By considering a subsequence of $y\mapsto \E \e h\left( x(\e),\frac{y}{\e} \right)$, we deduce that $\cH^{\alpha,\ft}(x,y)$ is $\frac{1}{\alpha}$-Lipschitz, nonnegative, and monotonically increasing in $y$ in the sense that its weak derivative is a positive measure.

\begin{definition}
Given a measure $\mu$ on $\R$, let $\left\{ \int e^{-k\ft u} \mu(du)\right\}_{k \in \Z_{>0}}$ be the set of $\ft$-\emph{exponential moments} of $\mu$; for $\ft = 1$ we simply say \emph{exponential moments} of $\mu$. Denote by $E^{\ft}\mu$ the measure defined by
\[ (E^{\ft}\mu)(du) = \mu(d(\exp(-\ft u))). \]
Let us write $E\mu = E^1 \mu$.
\end{definition}

\begin{lemma} \label{lem:bdblw}
Fix $x \in I$. There exists some $M$ such that $\cH^{\alpha,\ft}(x,y) = 0$ on $\{y \le M\}$, and $\cH^{\alpha,\ft}$ is determined by its $\ft$-exponential moments.
\end{lemma}
\begin{proof}
Introduce the auxiliary measure $\mu_x^{\ft}$ with moments defined
\begin{align}\label{eq:mux}
\int_0^\infty \! u^{\ft(k-1)} \mu_x^{\ft}(du) = \frac{1}{\alpha(k\ft)^2} \frac{1}{2\pi\i} \oint_{\cC} \! \frac{G_x(z)^{k\ft}}{z} \, dz
\end{align}
for $k \in \Z_{> 0}$. Then $\mu_x^{\ft} = E^{\ft}\cH^{\alpha,\ft}(x,\cdot)$. It follows that $\mu_x^{\ft}$ is a finite positive measure supported in $[0,+\infty)$. By scaling by an appropriate $c > 0$, we may view $c\mu_x^{\ft}$ as a probability measure. From the form of the contour integral, we see that there is a large enough $L > 0$ such that $\E_{c\mu_x^{\ft}} u^{\ft k} \le L^k$. By a standard Markov inequality argument, this implies $\mu_x^{\ft}$ is compactly supported in $[0,+\infty)$. By changing variables, we obtain the lemma.
\end{proof}

Let $\cH := \cH^{1,1}$. The exponential moments of $\cH$ can be analytically extended in the variable $k$. By the preceding lemma and (\ref{eq:hemom}), for any $\alpha,\ft > 0$ we have 
\[ \cH^{\alpha,\ft}(x,y) = \frac{1}{\alpha} \cH^{1,1}(x,y) =: \cH(x,y). \]

To prove Theorem \ref{thm:MACLLN}, it remains to show (\ref{eq:htarg}). Note that we don't know $\partial_x \cH(x,y)$ exists at the moment. However, by Lemma \ref{lem:bdblw}, $\cH(x,y)$ is determined by $\partial_y \cH(x,y)$, and we will first compute $\partial_y \cH(x,y)$ from which it will be apparent that $\partial_x \cH(x,y)$ exists.

We reduce to the case $-\infty < V_0$, $V_n < \infty$. Suppose $V_0 = -\infty$ or $V_n = +\infty$. Let us revive the superscripts for the back walls for the sake of this reduction. We have our limiting back wall $\cB$, and construct a sequence $\cB_n$ of back walls which is simply the restriction of $\cB$ to $I \cap [-N,N]$ so $V_0 = -N$ and $V_n = N$. Since $\cG_x^{\cB_N} \to \cG_x^{\cB}$ uniformly on compactum in $\C \setminus \R$, we see that the exponential moments of $\cH^{\cB^N}$ converge to the exponential moments of $\cH^{\cB}$. Moreover, $\widetilde{\zeta}^{\cB_N}$ which solves (\ref{eq:compeq}) for $\cG^{\cB_N}$ converges to $\widetilde{\zeta}^{\cB}$ pointwise. Thus the reduction is valid by (\ref{eq:htarg}). We suppress the $\cB$ superscript for the remainder of the proof.

As mentioned above we first compute $\partial_y \cH(x,y)$. The computation for $\partial_x \cH(x,y)$ will then be similar and use many of facts collected from the computation of $\partial_y \cH(x,y)$.\\

\noindent
\textbf{Computation of Density for $\partial_y \cH(x,y)$.} By (\ref{eq:hemom}), we have
\[ \int \partial_y \cH(x,y) e^{-k y}\,dy = k \int \cH(x,y) e^{-k y} \, dy = \frac{1}{2\pi \i} \oint_{\cC} \frac{\cG_x(z)^k}{kz}\,dz\]
for $k \in \Z_{> 0}$. By Lemma \ref{lem:bdblw}, we know $\nu_x = E(\partial_y \cH(x,\cdot))$ is compactly supported in $[0,\infty)$. Since $V_0 > -\infty$, from our explicit formula (\ref{eq:Gx}) we see that the poles of $\cG_x$ are strictly positive. The set of poles of the integrand contained in $\cC$ is exactly $\{0\} \cup \fP_x$ where
\begin{align} \label{eq:fPx}
\fP_x = \{ \mbox{branch poles $p$ of $\cG_x^{\cB}$ in $J_V$ such that $V < x$} \} \subset \bigcup_{\substack{\cB'(V^+) - \cB'(V^-) < 0 \\V < x}} J_V.
\end{align}
Let us split $\cC$ into a contour $\gamma_0$ containing $0$ but no poles in $\fP_x$ and a contour $\gamma$ which does not contain $0$ but contains $\fP_x$. For $v$ large enough so that $\supp(\nu_x) < v$ and $\left| \frac{\cG_x(z)}{v} \right| < 1$ along $z\in\gamma$, we have the Stieltjes transform
\begin{align} \nonumber
S_{\nu_x}(v) &= \int_0^\infty \frac{\nu_x(du)}{u - v} = -\sum_{k=0}^\infty \int_0^\infty \frac{u^k \nu_x(du)}{v^{k+1}} = - \frac{1}{2\pi\i} \sum_{k=1}^\infty \oint_{\cC} \frac{\cG_x(z)^k}{v^k k} \frac{dz}{z} \\ \label{eq:StY}
&= - \sum_{k=1}^\infty \frac{\cG_x(0)^k}{v^k k} - \frac{1}{2\pi\i} \sum_{k=1}^\infty \oint_\gamma \frac{\cG_x(z)^k}{v^k k} \frac{dz}{z} \\ \nonumber
&= \log\left(1 - \frac{\cG_x(0)}{v} \right) + \frac{1}{2\pi\i} \oint_\gamma \log\left(1 - \frac{\cG_x(z)}{v} \right) \frac{dz}{z}.
\end{align}
For $v$ large enough so that $\supp(\nu_x) < v$ and $\left| \frac{\cG_x(z)}{v} \right| < 1$ along $\gamma$, define the auxiliary function
\begin{align}
\Ups_x(v) &= \log\left( 1 - \frac{\cG_x(0)^p}{v^p} \right) + \frac{1}{2\pi\i} \oint_\gamma \! \log\left( 1- \frac{\cG_x(z)^p}{v^p} \right) \frac{dz}{z}.
\end{align}
We can extend $\Ups_x(v)$ to the domain $\{ |\arg z| < 2\pi/p \} \setminus \supp \nu_x$. Indeed, setting $\omega = e^{2\pi\i/p}$, we have
\begin{align} \label{eq:UpS}
\Ups_x(v) = \sum_{j=0}^{p-1} \log\left( 1 - \frac{\omega^j\cG_x(0)}{v} \right) + \frac{1}{2\pi\i} \oint_\gamma \sum_{j=0}^{p-1} \frac{\log\left( 1 - \frac{\omega^j G_x(z)}{v} \right)}{z} \, dz = \sum_{j=0}^{p-1} S_{\nu_x}(v \omega^{-j})
\end{align}
for $v$ large, but since the right hand side is an analytic function in $\{ |\arg z| < 2\pi/p \} \setminus \supp \nu_x$ this identity extends to this larger domain.

The idea is that we want to evaluate the residues for the contour integral formula of the Stieltjes transform, but since the integrand is not meromorphic we consider an the auxiliary function $\Ups_x(v)$. The auxiliary function allows us to consider an alternative meromorphic integrand. The following lemma establishes that $\Ups_x(v)$ may be considered instead of the Stieltjes transform.

\begin{lemma} \label{lem:ups}
For $v_0 > 0$,
\begin{align} \label{eq:UpSrho}
\frac{1}{\pi}\Im_+ \Ups_x(v_0)= \frac{1}{\pi}\Im_+ S_{\nu_x}(v_0) = \rho_x(v_0)
\end{align}
where $\rho_x$ is the density of $\nu_x$ (with respect to the Lebesgue measure) and $\Im_+$ is the imaginary part taken from the limit from the upper half plane.
\end{lemma}
\begin{proof}
The second equality of (\ref{eq:UpSrho}) is a property of the Stieltjes transform of a measure. We prove the first equality of (\ref{eq:UpSrho}).

Take $v_0 > 0$. Since $\overline{S_{\nu_x}(v)} = S_{\nu_x}(\bar{v})$ for $v \in \C \setminus \mathrm{supp}\,\nu_x = \C \setminus [0,\infty)$,
\[\Im_+(S_{\nu_x}(v_0\omega^{-j}) + S_{\nu_x}(v_0\omega^j) ) = 0 \]
for $j \neq 0$ and $j \neq p/2$ if $p$ is even. If $p$ is even, then
\[ \Im_+ S_{\nu_x}(v_0\omega^{p/2}) = \Im_+ S_{\nu_x}(-v_0) = 0\]
since $\supp\nu_x \subset [0,\infty)$. By (\ref{eq:UpS}), this proves the lemma.
\end{proof}

By Lemma \ref{lem:ups}, we are left to find $\Im_+ \Ups_x(v) = \rho_x(v)$ and change variables $v = e^{-y}$ to compute $\partial_y \cH(x,y)$. To this end we consider the following equation
\begin{align} \label{eq:pcompeq}
\cG_x(z)^p = v^p.
\end{align}
The solutions of (\ref{eq:compeq}) are exactly solutions of
\begin{align} \label{eq:jeq}
\cG_x(z) = \omega^j v
\end{align}
for some $j \in [[0,p-1]]$. A solution to the companion equation (\ref{eq:compeq}) is a solution to (\ref{eq:jeq}) with $j = 0$ and $v = e^{-y}$.

\begin{definition}
We say that $z(x,v)$ is a \emph{root function of} (\ref{eq:pcompeq}) if $z(x,v)$ is a root of (\ref{eq:pcompeq}) for each $(x,v)$ and continuous in $(x,v)$.
\end{definition}

Let $z(x,v)$ be a root function of (\ref{eq:pcompeq}). Note that as $v \to +\infty$, $z(x,v)$ converges to a pole of $\cG_x(z)^p$, equivalently a branch pole of $\cG_x(z)$.

\begin{definition}
If $\fp$ is a branch pole of $\cG_x(z)$, let us say that $\fp$ is the \emph{source} of a root function $z(x,v)$ of (\ref{eq:pcompeq}) if $\lim_{v \to +\infty} z(x,v) = \fp$.
\end{definition}

Let $\fZ_x$ be the multiset of $\fP_x$ where the multiplicity of an element $\fp \in \fP_x$ in $\fZ_x$ is the multiplicity of the pole in $\cG_x(z)^p$. We will abuse notation and treat $\fp \in \fZ_x$ as an element of $\fP_x$ in statements such as ``$\lim_{v \to +\infty} z(x,v) = \fp$''. Let $\cM$ be the set of $v \in \C$ such that $\cG_x(z)^p - v^p$ has a root of order $>1$.

\begin{lemma} \label{lem:upsform}
For $v \in \R_{> 0}$, we have that
\begin{align} \label{eq:upsroot}
\Ups_x(v) = \log \left( 1 - \frac{\cG_x(0)}{v} \right) -\sum_{\fp \in \fZ_x} (\log z_{\fp}(x,v) - \log \fp)
\end{align}
where $z_{\fp}(x,v)$ is a root function of (\ref{eq:pcompeq}). Furthermore the root functions $z_{\fp}(x,v)$ satisfy the following properties:
\begin{itemize}
    \item[(i)] For $v \in \R_{> 0} \setminus \cM$, the root functions $z_{\mathfrak{p}}(x,v)$ are distinct from one another.
    \item[(ii)] If $z_{\mathfrak{p}}(x,v_0)$ is a root of (\ref{eq:jeq}) for $j = 0$, then $z_{\mathfrak{p}}(x,v)$ is a root of (\ref{eq:jeq}) for $j=0$ for all $v > 0$.
    \item[(iii)] If $z_{\mathfrak{p}}(x,v_0) \in \cD_x^{\C}$ is a root of (\ref{eq:jeq}) for some $j \ne 0$, then $z_{\mathfrak{p}}(x,v) \in \cD_x^{\C}$ for all $v > 0$.
    \item[(iv)] If $v > 0$ and $z_{\fp}(x,v)$ is either a double root of (\ref{eq:pcompeq}) or a root of (\ref{eq:jeq}) for $j \ne 0$ such that $z_{\fp}(x,v) \in \C \setminus \R$, then there exists $\fq \in \fZ_x$ different from $\fp$ such that $z_{\fp}(x,v) = \overline{z_{\fq}(x,v)}$.
\end{itemize}
\end{lemma}
\begin{proof}
Since the poles of $1 - \frac{G_x(z)^p}{v^p}$ in $\gamma$ occur exactly in $\fP_x$, we have
\begin{align} \label{eq:upscont}
\Ups_x(v) &= \log \left( 1 - \frac{\cG_x(0)}{v} \right) -\frac{1}{2\pi\i} \oint_\gamma (\log z)\frac{d}{dz}\log \left( 1- \frac{\cG_x(z)^p}{v^p} \right) \, dz \\
&= \log \left( 1 - \frac{\cG_x(0)}{v} \right) -\sum_{\fp \in \fZ_x} (\log z_{\fp}(x,v) - \log \fp)
\end{align}
for $z_{\fp}(x,v)$ a root function with source $\fp$, where this formula is initially valid for $v$ large. For large enough $v$, there is no ambiguity in the choice of $z_{\fp}(x,v)$. We want to extend this formula to $v \in (0,+\infty)$. To do this, we require care in the choice of analytic extension to avoid ambiguity near double roots.

The set $\cM$ is finite. Indeed, the discriminant of $\cG_x(z)^p - v^p \in (\R[v])[z]$ is a polynomial in $\R[v]$, and the set of roots of the discriminant is exactly the set $\cM$. Thus, from some large $v_0 > 0$, $z_{\fp}(x,v)$ can be extended to any $v \in \R_{> 0} \setminus \cM$ via a path from $v_0$ to $v$ which avoids $\cM$. This extension is not unique, as windings around $\cM$ may change the value. This ambiguity is lost if we enforce the path from $v_0$ to $v$ to lie in a strip $U = \{ 0 \le \Im z < \delta\}$ where $\delta$ is chosen so that $\cM \cap U \subset \R$. For $v\in \R \setminus \cM$, let $z_{\fp}(x,v)$ denote the root function extended via a path from some large $v_0$ to $v$ within the strip $U \setminus \cM$, see Figure \ref{fig:analext}. We extend the definition to all $v \in \R$ by continuity.

\begin{figure}[ht]
    \centering
    \includegraphics[trim=0 0 0 365,clip, scale=0.7]{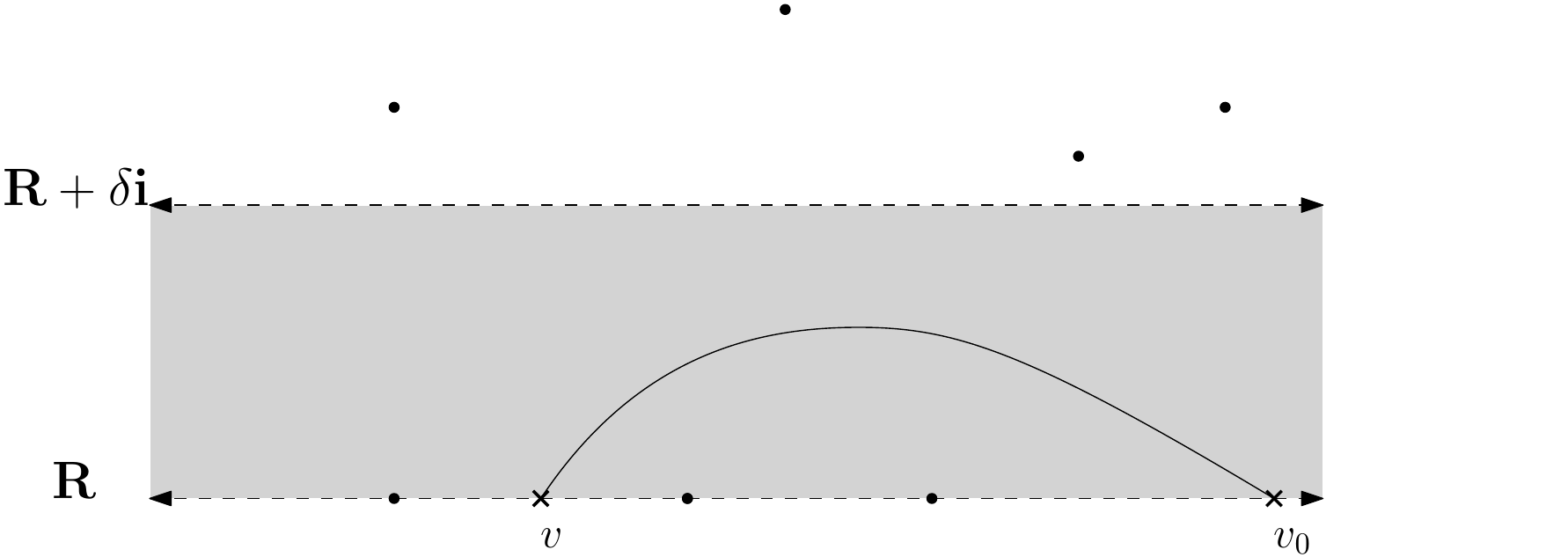}
    \caption{Dots represent points of $\cM$. We extend from $v_0$ to $v$ along a path that remains in the grey strip.}
    \label{fig:analext}
\end{figure}

We now prove that the root function $z_{\fp}(x,v)$ satisfy properties (i)-(iv). Property (i) follows from the definition of our root function, and since the roots must remain separated on $U \setminus \cM$.

For properties (ii) and (iii), we note that a point $z \in \C$ cannot be both a root of (\ref{eq:jeq}) for $j_1 = 0$ at $(x,v_1)$ and $j_2 \in [[1,p-1]]$ at $(x,v_2)$; the special position of $j = 0$ is a consequence of our choice of branch cut for $\cG_x$. This immediately implies property (ii). For property (iii), we use the additional fact that $z_{\mathfrak{p}}(x,v_1) \in \cD_x^{\C}$ and $z_{\mathfrak{p}}(x,v_2) \notin \cD_x^{\C}$ implies that $z_{\mathfrak{p}}(x,v)$ crosses the boundary of $\cD_x^{\C}$ at some $v_3$. The boundary of $\cD_x^{\C}$ is the closure of
\[\{\zeta(x,y):y \in \R, (x,y) \in \sL\}\cup\{\overline{\zeta}(x,y):y \in \R, (x,y) \in \sL\}\]
which means that $z_{\mathfrak{p}}(x,v_3)$ is a root of (\ref{eq:jeq}) for $j= 0$. By Property (ii), this implies $z_{\fp}(x,v)$ is a root of (\ref{eq:jeq}) for $j=0$ for all $v > 0$.

For property (iv), observe that the set of root functions $z_{\fp}(x,v)$ of (\ref{eq:jeq}) for $j \ne 0$ is exactly the set of all roots of (\ref{eq:jeq}) for $j \in [[1,p-1]]$ contained in $\cD_x^{\C}$. Indeed, this is the case as $v \to +\infty$, and by property (iii) this is therefore the case for all $v > 0$. If $z_{\fp}(x,v) \in \cD_x^{\C} \setminus \R$ for some $\fp \in \fZ_{x,1}$, then $\overline{z_{\fp}(x,v)} \in \cD_x^{\C}$ is a root of (\ref{eq:jeq}) for some $j \in [[1,p-1]]$. Property (iv) follows from this fact; note that the double root case is a limiting case.
\end{proof}

Given Lemma \ref{lem:upsform}, we may define the sub(multi)set $\fZ_{x,0}$ of $\fZ_x$ consisting of $\fp$ such that $z_{\fp}(x,v)$ is a root of (\ref{eq:jeq}) for $j = 0$. Similarly, let $\fZ_{x,1}$ be the sub(multi)set of $\fZ_x$ consisting of $\fq$ such that $z_{\fq}(x,v)$ is a root of (\ref{eq:jeq}) for $j \in [[1,p-1]]$. 

We rewrite (\ref{eq:upsroot}) as
\begin{align*}
\Ups_x(v) = \log\left(1 - \frac{\cG_x(0)^p}{v^p} \right) - \sum_{\fp \in \fZ_{x,0}} (\log z_{\fp}(x,v) - \log \fp) - \sum_{\fp \in \fZ_{x,1}} (\log z_{\fp}(x,v) - \log \fp).
\end{align*}
By property (iv) of Lemma \ref{lem:upsform}, this expression for $\Ups_x(v)$ implies
\begin{align} \label{eq:upsz0}
\Im_+ \Ups_x(v) = \pi \cdot \1_{v < \cG_x(0)} - \sum_{\fp \in \fZ_{x,0}} \arg z_{\fp}(x,v).
\end{align}
The argument branch is taken so that $\arg z_{\fp}(x,v) = 0$ for $v$ sufficiently large. This formula is valid except where $v = \cG_x(0)$ and $z_{\fp}(x,v) = 0$ for some $\fp \in \fZ_{x,0}$. Since we are evaluating the density of a measure, these finitely many exceptional points are immaterial.

By our choice of $\wzeta$, we have $\fp_0 = \min \cJ_x \in \fZ_{x,0}$ such that $\wzeta(x,y) = z_{\fp_0}(x,v)$ for some $v > 0$ where $e^{-y} = v$. From Remark \ref{rem:2roots}, we have that $z_{\fp_0}(x,v_0) = \wzeta(x,y_0)$ at some point $v_0 = e^{-y_0}$ where $\wzeta(x,y) \in \HH$; recall that nonreal pair of roots first appear (as we increase $y$) when the two roots in the external component coalesce. Moreover, we can find $y_0$ sufficiently negative so that the other $z_{\fp}(x,v)$ for $\fp \in \fZ_x \setminus \{\fp_0\}$ are confined in their respective components (which are not the external component). We note that $z_{\fp_0}(x,v_0) = \wzeta(x,y_0)$, rather than its conjugate, because we must have $0 \le \frac{1}{\pi} \Im_+ \Ups_x = \rho_x \le 1$ by the positivity of $\partial_y \cH$ and $1$-Lipschitz condition on $\cH$. Define the set
\begin{align*}
\begin{multlined}
\sC = \{(x,y) \in \sL: y \in \R, \wzeta(x,y) = z_{\fp}(x,e^{-y})~\mbox{for some}~\fp \in \fZ_{x,0},\\
~\mbox{and}~\wzeta(x,y) \neq \overline{z_{\fp}}(x,e^{-y})~\mbox{for all}~\fp \in \fZ_{x,0}\}.
\end{multlined}
\end{align*}
Since $(x,y_0) \in \sC$, we have $\sC$ is nonempty. We use a connectedness argument to show that $\sC = \sL$. Since the measures $\nu_x$ are compactly supported and bounded, and the moments depend continuously in $x$, we have that the Stieltjes transforms $S_{\nu_x}(v)$ are continuous in $(x,v)$ for $x \in I$ and $v \in \C \setminus \R$. Then $\Ups_x(v)$ is continuous in $(x,v)$ for $x\in I$ and $v \in \{ |\arg z| < 2\pi/p\} \setminus \supp\nu_x$. By the form (\ref{eq:upscont}), this continuity can be extended to $(x,v)$ where $v \in \R_{> 0}$ as long as $z_{\fp}(x,v) \neq 0$ for all $\fp \in \fZ_x$ and $v \ne \cG_x(0)$ (which is generically the case). It follows that $\sC$ must be both an open and closed subset of $\sL$, where the closed condition is immediate and the open condition follows from continuity. Since we showed that $\sC$ is nonempty, we deduce that $\sC = \sL$ since $\sL$ is homeomorphic to the simply connected set $\HH$ by Theorem \ref{thm:homeo}. With this fact and (\ref{eq:upsz0}), we have (for $v = e^{-y}$)
\begin{align}
\frac{1}{\pi} \Im_+ \Ups_x(v) = \1_{v < \cG_x(0)} - \frac{1}{\pi}\arg \wzeta(x,-\log v)
\end{align}
where the argument is chosen so that $\arg z = 0$ when $z > 0$. Indeed, the indicator takes value $1$ whenever $\wzeta$ is nonreal (see the discussion following Remark \ref{rem:wd}) and we require $0 \le \frac{1}{\pi} \Ups_x \le 1$, so this determines the choice of argument. Changing variables, we have
\[ \partial_y \cH(x,y) = \1_{e^{-y} < \cG_x(0)} - \frac{1}{\pi} \arg \wzeta(x,y).\]

\noindent
\textbf{Computation of Density for $\partial_x \cH(x,y).$} We follow a similar argument, and have already done most of the work for it. Notice that
\begin{align}
\frac{\partial_x \cG_x(z)}{\cG_x(z)} = \frac{1}{p} \sum_{\sigma \in \cS} \frac{\sigma e^{-x}z}{1 - \sigma e^{-x} z}
\end{align}
so that by (\ref{eq:hemom}), we have
\begin{align}
\int \partial_x \cH(x,y) e^{-ky} = \frac{1}{2\pi\i} \oint_\cC \frac{\cG_x(z)^{k-1}\partial_x\cG_x(z)}{kz} \, dz = \frac{1}{p} \sum_{\sigma \in \cS} \frac{1}{2\pi\i} \oint_\cC \frac{\cG_x(z)^k}{k} \frac{\sigma e^{-x}}{1 - \sigma e^{-x} z} \, dz
\end{align}
Recall that we had split $\cC$ into $\gamma$ and $\gamma_0$ before in the computation of $\partial_y \cH(x,y)$. This time, we can simply deform $\cC$ to $\gamma$ which encircle $\fP_x$ and not $0$ since there is no pole at $0$. Let $\xi_x = E(\partial_x \cH(x,\cdot))$. As before, we have for large $v$,
\begin{align*}
S_{\xi_x}(v) &= -\sum_{k=1}^\infty \frac{1}{p} \sum_{\sigma \in \cS} \frac{1}{2\pi\i} \oint_\gamma \frac{\cG_x(z)^k}{v^k k} \frac{\sigma e^{-x}}{1 - \sigma e^{-x} z} \, dz = \frac{1}{p} \sum_{\sigma \in \cS} \frac{1}{2\pi\i} \oint_{\gamma} \log\left(1 - \frac{\cG_x(z)}{v} \right) \frac{\sigma e^{-x}}{1 - \sigma e^{-x}z} \,dz
\end{align*}
We consider, also for large $v$,
\begin{align} \label{eq:Xieq}
\begin{multlined}
\Xi_x(v) = \frac{1}{p} \sum_{\sigma \in \cS} \frac{1}{2\pi\i} \oint_\gamma \log \left(1 - \frac{\cG_x(z)^p}{v^p} \right) \frac{\sigma e^{-x}}{1 - \sigma e^{-x}z} \,dz \\
= \frac{1}{p}\sum_{\sigma \in \cS} \frac{1}{2\pi\i} \oint_\gamma \log(1 - \sigma e^{-x} z) \frac{d}{dz} \log\left(1 - \frac{\cG_x(z)^p}{v^p} \right)\,dz.
\end{multlined}
\end{align}
As before with $\Ups_x(v)$, we can extend $\Xi_x(v)$ to $\{ |\arg z| < 2\pi/p\} \setminus \supp\xi_x$, and we have an analogue of Lemma \ref{lem:ups}
\begin{align}
\frac{1}{\pi} \Im_+ \Xi_x(v) = \frac{1}{\pi} \Im_+ S_{\xi_x}(v) = \varrho(v)
\end{align}
where $\varrho(v)$ is the density of $\xi_x$ with respect to Lebesgue. From (\ref{eq:Xieq}), we have the following formula
\begin{align}
\Xi_x(v) = \frac{1}{p} \sum_{\sigma \in \cS} \sum_{\fp \in \fZ_x} (\log(1 - \sigma e^{-x} z_{\fp}(x,v)) - \log(1 - \sigma e^{-x} \fp))
\end{align}
for $v > 0$ where $z_{\fp}(x,v)$ are defined as in Lemma \ref{lem:upsform}. The log branches here are chosen so that $\log(1 - \sigma e^{-x} z_{\fp}(x,v)) - \log(1 - \sigma e^{-x} \fp) = 0$ for $v$ large (so that $z_{\fp}(x,v)$ is near $\fp$). Define for $\sigma \in \cS$
\begin{align}
f_\sigma(x,v) & = \sum_{\fp \in \fZ_{x,1}} (\arg(1 - \sigma e^{-x} z_{\fp}(x,v)) - \arg(1 - \sigma e^{-x} \fp)) \\
g_\sigma(x,v) &= \sum_{\fp \in \fZ_{x,0}} (\arg(1 - \sigma e^{-x} z_{\fp}(x,v)) - \arg(1 - \sigma e^{-x} \fp))
\end{align}
so that
\[ \Im_+ \Xi_x(v) = \frac{1}{p}\sum_{\sigma \in \cS} (f_\sigma(x,v) + g_\sigma(x,v)). \]
Given the following lemma, we can compute that
\[ \frac{1}{\pi} \Im_+ \Xi_x(e^{-y}) = \frac{1}{\pi}\sum_{\sigma \in \cS} \frac{1}{p} \log\left( 1 - \sigma^{-1} e^x \wzeta(x,y) \right).\]
Changing variables gives us $\partial_x \cH(x,y)$ and completes the proof of the Theorem.

\begin{lemma}
For any $\sigma \in \cS$, both sets $\{y \in \R:  f_\sigma(x,e^{-y}) = 0\}$ and $\{y \in \R: g_\sigma(x,e^{-y}) = \arg(1 - \sigma e^{-x} \wzeta(x,y))\}$ are equal to $\R$.
\end{lemma}
\begin{proof}
We use a connectedness argument. Fix $\sigma \in \cS$, let $A_1 = \{y \in \R:  f_\sigma(x,e^{-y}) = 0\}$ and $A_2 = \{y \in \R: g_\sigma(x,e^{-y}) = \arg(1 - \sigma e^{-x} \wzeta(x,y))\}$. Both $A_1$ and $A_2$ are nonempty and closed subsets of $\R$. It suffices to prove that $A_1$ and $A_2$ are open subsets of $\R$.

Suppose $y_0 \in A_1$. We show that $[y_0,y_0 + \delta) \subset A_1$ for some $\delta > 0$. The argument for $(y_0-\delta,y_0]$ is identical, and so this shows that $A_1$ is open. Let $\fp \in \fZ_{x,1}$. If $z_{\fp}(x,e^{-y})$ is real for $y \in [y_0,y_0+\delta)$ for some $\delta > 0$, then $z_{\fp}(x,e^{-y})$ must be in some connected component of $\hatR \setminus \cJ_x$. Note that it cannot be in the external component since $\arg \cG_x = 0$ on the external component. Then $\arg(1 - \sigma^{-1} z_{\fp}(x,e^{-y}))$ is constant on $[y_0,y_0 +\delta)$.

Otherwise, for sufficiently small $\delta > 0$, we have $z_{\fp}(x,e^{-y}) \in \C \setminus \R$ for all $y \in (y_0,y_0 + \delta)$. This means that $z_{\fp}(x,e^{-y_0})$ is either a nonreal root of (\ref{eq:jeq}) or a double root. In either case, by property (iv) of Lemma \ref{lem:upsform}, there exists some $\fq \in \fZ_{x,1}$ different from $\fp$ such that $z_{\fq}(x,e^{-y_0}) = \overline{z_{\fp}}(x,e^{-y_0})$. If $z_{\fp}(x,e^{-y}) \in \C \setminus \R$ for $y \in (y_0,y_0+\delta)$, then $z_{\fq}(x,e^{-y}) = \overline{z_{\fp}}(x,e^{-y})$ for $y \in (y_0,y_0 + \delta)$ by continuity and property (i) of Lemma \ref{lem:upsform}. Then
\begin{align*}
\begin{multlined}
(\arg (1 - \sigma e^{-x}z_{\fp}(x,y)) - \arg (1 - \sigma e^{-x} z_{\fp}(x,y_0))) \\
+ (\arg (1 - \sigma e^{-x}z_{\fq}(x,y)) - \arg (1 - \sigma e^{-x} z_{\fq}(x,y_0))) = 0
\end{multlined}
\end{align*}
for $y \in (y_0,y_0+\delta)$. Summing over all $\fp \in \fZ_{x,1}$, we see that $f_\sigma(x,e^{-y}) - f_\sigma(x,e^{-y_0})= 0$ for $y \in [y_0,y_0+\delta)$. Thus $A_1$ is open.

The argument for the openness of $A_2$ is similar in structure. Before proceeding, note that from the argument that $\sC = \sL$ in the computation of $\partial_y \cH(x,y)$, we have that $(-\infty,y'] \subset A_2$ where $y'$ is the minimal $y$ such that (\ref{eq:compeq}) has a double root. We have that $\wzeta(x,y)$ is contained in the external component of $\hatR \setminus \cJ_x$ if and only if $y \in (-\infty,y']$ (see Remark \ref{rem:2roots+}).

Take $y_0 \in A_2 \setminus (-\infty,y')$. As before, we show that $[y_0,y_0 + \delta) \subset A_2$ for some $\delta > 0$, and the argument for $(y_0-\delta,y_0]$ is identical for $y_0 \in A_2 \setminus (-\infty,y']$. If $z_{\fp}(x,e^{-y})$ is real for $y \in [y_0,y_0 + \delta)$, then it must be confined to some component of $\hatR \setminus \cJ_x$ which is not the external component. As before, this implies $\arg(1 - \sigma e^{-x} z_{\fp}(x,e^{-y}))$ is constant for $y \in [y_0,y_0 + \delta)$.

Otherwise, there exists $\delta > 0$ small enough so that $z_{\fp}(x,e^{-y}) \in \C \setminus \R$ for $y \in (y_0,y_0 + \delta)$. Since $\sC = \sL$, this corresponds exactly to the case where $\wzeta(x,y) \in \C \setminus \R$, for $y \in (y_0,y_0 + \delta)$. Moreover, we must have $\wzeta(x,y) = z_{\fp}(x,e^{-y})$ for some $\fp \in \fZ_{x,0}$ and $z_{\fq}(x,e^{-y})$ are real for $\fq \in \fZ_{x,0}$, $\fq \neq \fp$ for $y \in [y_0,y_0 + \delta)$. In particular, $z_{\fq}(x,e^{-y})$ are confined to some junction component of $\hatR \setminus \cJ_x$ which is not the external component. Then
\[ g_\sigma(x,e^{-y}) - g_\sigma(x,e^{-y_0}) = -\arg(1 - \sigma e^{-x} \wzeta(x,y)) + \arg(1 - \sigma e^{-x} \wzeta(x,y_0)). \]
Thus $A_2$ is open.
\end{proof}

\subsection{Fluctuations: Proof of Theorem \ref{thm:MACGFF}} \label{sec:GFF}
\begin{proof}[Proof of Theorem \ref{thm:macGFF}]
For $x \in I$ and $k \in \Z_{> 0}$, by Proposition \ref{prop:aht} we have for $x(\e) \in I^\e$
\begin{align*}
\alpha \ft \sqrt{\pi} \int \! \left( h\left( x(\e), \frac{\y}{\e} \right) - \E h\left( x(\e), \frac{\y}{\e} \right) \right) e^{-k \ft \y} \, d\y &= \frac{\e \ft \sqrt{\pi} t^{kB(x)}}{k^2 (\log t)^2} \left( \wp_k(\pi^x;q,t) - \E \wp_k(\pi^x;q,t)\right) \\
&= \frac{\sqrt{\pi} t^{kB(x)}}{k^2 \ft \e} \left( \wp_k(\pi^x;q,t) - \E \wp_k(\pi^x;q,t)\right).
\end{align*}
Let $\x_1 \le \cdots \le \x_m$ in $I$ and $k_1,\ldots,k_m \in \Z_{> 0}$. Let $x_1(\e) \le \cdots \le x_m(\e)$ be points in $I^\e$ such that $x_i(\e)$ is $k_i\ft$-separated from singular points. By Theorem \ref{thm:gauss}, the random vector indexed by $i \in [[1,m]]$ whose components are given by
\[ \alpha \ft \sqrt{\pi} \int \left( h\left( x_i(\e), \frac{\y}{\e} \right) - \E h\left( x_i(\e), \frac{\y}{\e} \right) \right) e^{-k_i \ft \y} \, d\y\]
converges to the Gaussian vector whose $i$th and $j$th component for $i < j$ has the covariance
\begin{align} \label{eq:cov1}
\frac{\pi}{(2\pi \iota)^2k_1k_2 \ft^2} \oint_{\cC_j} \! \oint_{\cC_i} \! \frac{\cG_{\x_i}(z)^{k_i\ft} \cG_{\x_j}(w)^{k_j\ft}}{(z-w)^2} dz \, dw
\end{align}
where $\cC_1,\ldots,\cC_m$ are contours meeting the criteria in Theorem \ref{thm:gauss}, in particular $\cC_i$ is enclosed by $\cC_j$ for $i < j$ and the set of branch poles of $\cG_{x_i}$ enclosed by $\cC_i$ is exactly $\fP_{x_i}$ as defined in \eqref{eq:fPx}.

Suppose $\x_i < \x_j$ so that the contours $\cC_i$ and $\cC_j$ are separated. By Lemma \ref{lem:deform}, we can deform $\cC_i$ to the boundary of $\cD_{\x_i}^{\C}$ in (\ref{eq:cov1}). From (\ref{eq:Dbdry}), we can express the boundary as the union of
\[ \{\zeta(\x_i,\y): \y \in \R, (\x_i,\y) \in \sL\} \cup \{\overline{\zeta(\x_i,\y)}: \y \in \R, (\x_i,\y) \in \sL\} \]
up to finitely many points in the set difference. Let $Y_{\x_i} = \{\y \in \R: (\x_i,\y) \in \sL\}$ which is a union of finitely many open intervals. Parametrizing the boundary of $\cD_{\x_i}^{\C}$ by $\zeta(\x_i,\y)$ and its conjugate, we may rewrite (\ref{eq:cov1}) as the sum
\begin{align} \label{eq:cov2}
\begin{multlined}
- \frac{1}{4\pi k_1k_2 \ft^2} \int_{Y_{\x_j}} \! \int_{Y_{\x_i}} \! \frac{e^{-k_i\ft \y_1} e^{-k_j \ft \y_2}}{(\zeta(\x_i,\y_1)-\zeta(\x_j,\y_2))^2} \frac{\partial \zeta}{\partial \y_1}(\x_i,\y_1) \frac{\partial \zeta}{\partial \y_2}(\x_j,\y_2) \, d\y_1 \, d\y_2 \\
+ \frac{1}{4\pi k_1k_2 \ft^2} \int_{Y_{\x_j}} \! \int_{Y_{\x_i}} \! \frac{e^{-k_i\ft \y_1} e^{-k_j \ft \y_2}}{(\zeta(\x_i,\y_1)-\overline{\zeta}(\x_j,\y_2))^2} \frac{\partial \zeta}{\partial \y_1}(\x_i,\y_1) \frac{\partial \overline{\zeta}}{\partial \y_2}(\x_j,\y_2) \, d\y_1 \, d\y_2 \\
+ \frac{1}{4\pi k_1k_2 \ft^2} \int_{Y_{\x_j}} \! \int_{Y_{\x_i}} \! \frac{e^{-k_i\ft \y_1} e^{-k_j \ft \y_2}}{(\overline{\zeta}(\x_i,\y_1)-\zeta(\x_j,\y_2))^2} \frac{\partial \overline{\zeta}}{\partial \y_1}(\x_i,\y_1) \frac{\partial \zeta}{\partial \y_2}(\x_j,\y_2) \, d\y_1 \, d\y_2\\
- \frac{1}{4\pi k_1k_2 \ft^2} \int_{Y_{\x_j}} \! \int_{Y_{\x_i}} \! \frac{e^{-k_i\ft \y_1} e^{-k_j \ft \y_2}}{(\overline{\zeta}(\x_i,\y_1)-\overline{\zeta}(\x_j,\y_2))^2} \frac{\partial \overline{\zeta}}{\partial \y_1}(\x_i,\y_1) \frac{\partial \overline{\zeta}}{\partial \y_2}(\x_j,\y_2) \, d\y_1 \, d\y_2.
\end{multlined}
\end{align}
Integrate by parts on $\y_1$ and $\y_2$ for each summand in (\ref{eq:cov2}), observing that the boundary terms cancel since the value of $\zeta(\x,\cdot)$ at the end points of any connected component of $Y_\x$ is real, to obtain
\begin{align*}
\begin{multlined}
- \frac{1}{4\pi} \int_{Y_{\x_j}} \! \int_{Y_{\x_i}} \! e^{-k_1 \ft \y_1} e^{-k_2 \ft \y_1} \bigl( \log(\zeta(\x_i,\y_1) - \zeta(\x_j,\y_1)) - \log(\zeta(\x_i,\y_1) - \overline{\zeta}(\x_j,\y_1)) \\
- \log(\overline{\zeta}(\x_i,\y_1) - \zeta(\x_j,\y_1)) + \log(\overline{\zeta}(\x_i,\y_1) - \overline{\zeta}(\x_j,\y_1)) \bigr) \, d\y_1 \, d\y_1 \\
= -\frac{1}{2\pi} \int_{Y_{\x_j}} \! \int_{Y_{\x_i}} \! e^{-k_1 \ft \y_1} e^{-k_2 \ft \y_1} \log \left| \frac{\zeta(\x_i,\y_1) - \zeta(\x_j,\y_1)}{\zeta(\x_i,\y_1) - \overline{\zeta}(\x_j,\y_1)} \right|
\end{multlined}
\end{align*}
where the final equality follows from the fact that this covariance is real so there is no imaginary part arising from branch considerations of the logarithms. The latter term is exactly the covariance
\[ \cov\left( \int \fH(\zeta(\x_i,\y)) e^{-k_i\ft \y} \, d\y, \int \fH(\zeta(\x_j,\y)) e^{-k_j\ft \y} \, d\y \right). \]
The case $\x_i = \x_j$ follows from taking the limit.
\end{proof}

\appendix \section{} \label{sec:appendix}
We recall the notion of cumulants and some basic properties.

\begin{definition} \label{def:cum1}
For any positive integer $\nu$, let $\Theta_\nu$ be the collection of all set partitions of $[[1,\nu]]$, namely
\[ \Theta_\nu = \left\{ \{U_1,\ldots,U_d\}: d > 0, ~\bigcup_{i=1}^d U_i = [[1,\nu]], ~U_i \cap U_j = \emptyset~\forall i \ne j,~U_i \ne \emptyset ~\forall i \in [[1,d]] \right\}. \]
For a random vector $\vec{u} = (u_1,\ldots,u_m)$ and any $v_1,\ldots,v_\nu \in \{u_1,\ldots,u_m\}$, define the \emph{(order $\nu$) cumulant} $\kappa(v_1,\ldots,v_\nu)$ as
\begin{align} \label{eq:cum1}
\kappa(v_1,\ldots,v_\nu) = \sum_{\substack{d > 0 \\ \{U_1,\ldots,U_d\} \in \Theta_\nu}} (-1)^{d-1} (d-1)! \prod_{\ell=1}^d \E \left[ \prod_{i \in U_\ell} v_i \right].
\end{align}
\end{definition}

From the definition we see that for any random vector $\vec{u}$, the existence of all cumulants of order up to $\nu$ is equivalent to the existence of all moments of order up to $\nu$. Note that the cumulants of order $2$ are exactly the covariances:
\[ \kappa(v_1,v_2) = \cov(v_1,v_2). \]

We have the following alternative definition for cumulants.

\begin{definition} \label{def:cum2}
Let $\vec{u} = (u_1,\ldots,u_m)$ be a random vector. For any $v_1,\ldots,v_\nu \in \{u_1,\ldots,u_m\}$, define the \emph{(order $\nu$) cumulant} $\kappa(v_1,\ldots,v_\nu)$ as
\begin{align} \label{eq:cum2}
\kappa(v_1,\ldots,v_\nu) = (-i)^\nu \left. \frac{\partial^\nu}{\partial t_1 \cdots \partial t_m} \log \E \left[ \exp\left( \i \sum_{j=1}^\nu t_j v_j \right) \right] \right|_{t_1 = \cdots = t_\nu = 0}.
\end{align}
\end{definition}

For further details see \cite[Section 3.1, Section 3.2]{PT} wherein the agreement between Definitions \ref{def:cum1} and \ref{def:cum2} is shown by taking the second definition and proving (\ref{eq:cum1}).

As a consequence of the second definition we have the following lemma.

\begin{lemma} \label{lem:gausscum}
A random vector is Gaussian if and only if all cumulants of order $\ge 3$ vanish.
\end{lemma}

The next lemma, used in the proof of Theorem \ref{thm:gauss}, allows the replacement of one sequence of random vectors with another without affecting the cumulants (up to some order) given that the two sequences vanish in the limit in $L^2$ norm.

\begin{lemma} \label{lem:replace}
Consider two families of random vectors $\vec{u}^\e = (u_1^\e,\ldots,u_m^\e)$ and $\vec{v}^\e = (v_1^\e,\ldots,v_m^\e)$ indexed by $\e > 0$. Fix an integer $\eta > 0$ and suppose that we have the following limit
\begin{align} \label{eq:applem1}
\lim_{\e\to 0} \kappa(u_{i_1}^\e, \ldots, u_{i_\nu}^\e) = \kappa_{i_1,\ldots,i_\nu}
\end{align}
for some $\kappa_{i_1,\ldots,i_\nu}$ (symmetric in the indices) which holds for any $i_1,\ldots,i_\nu \in [[1,m]]$ and $\nu \in [[1,2\eta]]$. If $\E|\vec{u} - \vec{v}|^2 \to 0$, then
\begin{align} \label{eq:applem2}
\lim_{\e\to 0} \kappa(v_{i_1}^\e, \ldots, v_{i_\nu}^\e) = \kappa_{i_1,\ldots,i_\nu}
\end{align}
for any $i_1,\ldots,i_\nu \in [[1,m]]$ and $\nu \in [[1,\eta]]$.
\end{lemma}
\begin{proof}
The hypothesis (\ref{eq:applem1}) implies the existence of limits
\begin{align} \label{eq:mulim}
\lim_{\e \to 0} \E u_{i_1}^\e \cdots u_{i_\nu}^\e = \mu_{i_1 \cdots i_\nu}
\end{align}
for some $\mu_{i_1 \cdots i_\nu}$ (symmetric in the indices) which holds for any $i_1,\ldots,i_\nu \in [[1,m]]$ and $\nu \in [[1,2\eta]]$. Then
\begin{align} \label{eq:replace}
\E[v_{i_1} u_{i_2} \cdots u_{i_\nu}] = \E[u_{i_1} u_{i_2} \cdots u_{i_\nu}] + \E[(v_{i_1} - u_{i_1}) u_{i_2} \cdots u_{i_\nu}].
\end{align}
If $\nu \in [[1,\eta]]$, then
\begin{align} \label{eq:cs}
\left| \E[(v_{i_1} - u_{i_1}) u_{i_2} \cdots u_{i_\nu}] \right| \le [\E(v_{i_1} - u_{i_1})^2]^{\frac{1}{2}} \cdot [\E(u_{i_2}^2 \cdots u_{i_\nu}^2)]^{\frac{1}{2}}
\end{align}
Since
\[ \E(v_{i_1} - u_{i_1})^2 \le \E|\vec{v} - \vec{u}|^2 \to 0, \]
(\ref{eq:cs}) and (\ref{eq:mulim}) imply that the right hand side of (\ref{eq:replace}) converges to $\mu_{i_1\cdots i_\nu}$. By iterating this argument in replacing $u_{i_2}$ with $v_{i_2}$, then $u_{i_3}$ with $v_{i_3}$, and so on, we have that
\[ \lim_{\e \to 0} \E v_{i_1}^\e \cdots v_{i_\nu}^\e = \mu_{i_1 \cdots i_\nu} \]
for $\nu \in [[1,\eta]]$. This implies (\ref{eq:applem2}), proving the lemma.
\end{proof}

We have the following formal versions of (\ref{eq:cum1}) and (\ref{eq:cum2}). Let $E_{\nu_1,\ldots,\nu_n} \in \C$ with $E_{0,\ldots,0} = 1$. Define the following formal power series
\begin{align*}
E(t_1,\ldots,t_\nu) & = \sum_{n_1,\ldots,n_\nu \ge 0} \frac{E_{n_1,\ldots,n_\nu}}{n_1! \cdots n_\nu!} t_1^{n_1} \cdots t_\nu^{n_\nu} \\
K(t_1,\ldots,t_\nu) &= \log E(t_1,\ldots,t_\nu) =: \sum_{n_1,\ldots,n_\nu \ge 0} \frac{K_{n_1,\ldots,n_\nu}}{n_1! \cdots n_\nu!} t_1^{n_1} \cdots t_\nu^{n_\nu}.
\end{align*}
Let $E(U) = E_{\nu_1,\ldots,\nu_n}$ where $\nu_j = 1$ if $j \in U$ and $0$ otherwise, and likewise define $K(U)$. Letting $\Theta_U$ be the collection of all set partitions of $U$, we have
\begin{align}
K(U) = \sum_{\substack{d > 0 \\ \{U_1,\ldots,U_d\} \in \Theta_U}} (-1)^{d-1} (d-1)! \prod_{\ell=1}^d E(U_\ell)
\end{align}
By exponentiating $K(t_1,\ldots,t_\nu)$ we also obtain
\begin{align}
E(U) = \sum_{\substack{d > 0 \\ \{U_1,\ldots,U_d\} \in \Theta_U}} \prod_{\ell=1}^d K(U_\ell)
\end{align}
where the sum is over all set partitions of $U$. This gives us the following lemma.

\begin{lemma} \label{lem:cuminv}
Suppose that $K$ and $E$ are functions which take values on nonempty subsets of $[[1,\nu]]$. Further suppose that
\begin{align} \label{eq:TU}
E(U) = \sum_{\substack{d > 0 \\ \{U_1,\ldots,U_d\} \in \Theta_U}} \prod_{i=1}^d K(U_i)
\end{align}
Then
\begin{align} \label{eq:LU}
K(U) = \sum_{\substack{d > 0 \\ \{U_1,\ldots,U_d\} \in \Theta_U}} (-1)^{d-1} (d-1)! \prod_{i=1}^d E(U_i)
\end{align}
\end{lemma}

\end{document}